\documentclass[12pt,reqno]{amsart}
\usepackage[a4paper]{geometry}
\usepackage{graphicx} 
\usepackage{tikz}

\usepackage{amssymb,amsmath,amsthm,enumerate,eucal}

\DeclareFontFamily{U}{mathx}{\hyphenchar\font45}
\DeclareFontShape{U}{mathx}{m}{n}{<5> <6> <7> <8> <9> <10>
<10.95> <12> <14.4> <17.28> <20.74> <24.88> mathx10}{}
\DeclareSymbolFont{mathx}{U}{mathx}{m}{n}
\DeclareFontSubstitution{U}{mathx}{m}{n}
\DeclareMathAccent{\widecheck}{0}{mathx}{"71}

\DeclareMathOperator{\Ker}{Ker}
\DeclareMathOperator{\dist}{dist}

\DeclareMathOperator*{\Res}{Res}

\renewcommand{\Re}{\operatorname{Re}}
\renewcommand{\Im}{\operatorname{Im}}

\newcommand{\abs}[1]{\lvert#1\rvert}
\newcommand{\Abs}[1]{\left\lvert#1\right\rvert}
\newcommand{\norm}[1]{\lVert#1\rVert}
\newcommand{\Norm}[1]{\left\lVert#1\right\rVert}
\newcommand{\jap}[1]{\langle#1\rangle}

\newcommand{\bbR}{{\mathbb R}}
\newcommand{\bbC}{{\mathbb C}}
\newcommand{\bbZ}{{\mathbb Z}}

\newcommand{\bH}{\mathbf{H}}

\newcommand{\calM}{\mathcal{M}}

\newcommand{\calU}{\mathcal{U}}

\numberwithin{equation}{section}

\theoremstyle{plain}
\newtheorem{theorem}{\bf Theorem}[section]
\newtheorem*{theorem*}{Theorem}
\newtheorem{lemma}[theorem]{\bf Lemma}
\newtheorem{proposition}[theorem]{\bf Proposition}
\newtheorem*{proposition*}{\bf Proposition}

\theoremstyle{definition}
\newtheorem{definition}[theorem]{\bf Definition}

\theoremstyle{remark}
\newtheorem*{remark*}{\bf Remark}
\newtheorem{remark}[theorem]{\bf Remark}

\newcommand{\wt}{\widetilde}

\begin{document}

\title[Hankel operators and elliptic functions]{Hankel operators with band spectra and elliptic functions}

\author{Alexander Pushnitski}
\address{Department of Mathematics, King's College London, Strand, London, WC2R~2LS, U.K.}
\email{alexander.pushnitski@kcl.ac.uk}

\author{Alexander Sobolev}
\address{Department of Mathematics, University College London, Gower Street, London WC1E 6BT, U.K.}
\email{a.sobolev@ucl.ac.uk}




\begin{abstract}
We consider the class of bounded self-adjoint Hankel operators $\bH$, realised as integral operators on the positive semi-axis, that commute with dilations by a fixed factor. 
By analogy with the spectral theory of periodic Schr\"{o}dinger operators, we develop a Floquet-Bloch decomposition for this class of Hankel operators $\bH$, which represents $\bH$ as a direct integral of certain compact fiber operators. As a consequence, $\bH$ has a band spectrum. We establish main properties of the corresponding band functions, i.e. the eigenvalues of the fiber operators in the Floquet-Bloch decomposition. A striking feature of this model is that one may have flat bands that co-exist with non-flat bands; we consider some simple explicit examples of this nature. Furthermore, we prove that the analytic continuation of the secular determinant for the fiber operator is an elliptic function; this link to elliptic functions is our main tool. 
\end{abstract}

\maketitle

\setcounter{tocdepth}{1}
\tableofcontents

\section{Introduction}
\label{sec.a}

\subsection{Overview}

Hankel operators can be described in several different (but unitarily equivalent) ways, see e.g. \cite[Chapter 1]{Peller} or \cite[Part B, Chapter 1]{Nikolski}. In this paper, we view Hankel operators as integral operators in $L^2(\bbR_+)$ of the form 
\[
\mathbf H: 
f\mapsto \int_0^\infty h(t+s)f(s)ds, \quad f\in L^2(\bbR_+),
\]
where $h$ is a fixed function on $\bbR_+$, which in this context will be called the \emph{kernel function}. 
In general, $h$ may be a distribution, but in concrete situations that are of interest to us below, $h$ will be a  function sufficiently regular away from zero. 

We mention the classical example $h(t)=1/t$, in which case 
$\mathbf H$ is known as the \emph{Carleman operator}. This operator is bounded and self-adjoint, 
its spectrum is purely absolutely continuous, 
coincides with the interval $[0,\pi]$ and has multiplicity two (see e.g. \cite[Section~10.2]{Peller}). For future reference we display the boundedness of the Carleman operator as the quadratic form estimate
\begin{equation}
\Abs{\int_0^\infty \int_0^\infty\frac{f(t)\overline{f(s)}}{t+s}dt\, ds}
\leq
\pi \norm{f}_{L^2}^2. 
\label{eq:carleman}
\end{equation}

We will be interested in spectral properties of a class 
of \emph{bounded} \emph{self-adjoint} 
Hankel operators $\mathbf H$. 
Boundedness in $L^2(\bbR_+)$ arises from the condition 
$$
\abs{h(t)}\leq C/t,\quad t>0
$$
that we assume throughout the paper; indeed, by \eqref{eq:carleman} it implies
\begin{equation}
\norm{\mathbf H}\leq \pi C
\label{eq:carleman1}
\end{equation}
for the operator norm of $\mathbf H$. 
Self-adjointness means that the Hankel kernel $h$ is real-valued. 

For a \emph{period} $T>0$ (that we fix once and for all), let us denote by $V_T$ the unitary dilation operator in $L^2(\bbR_+)$, 
$$
(V_Tf)(t)=e^{-T/2}f(e^{-T}t), \quad t>0, \quad f\in L^2(\bbR_+).
$$
We are interested in the class of Hankel operators $\mathbf H$ that commute with $V_T$:
\begin{equation}
\mathbf HV_T=V_T\mathbf H.
\label{a1}
\end{equation}
We call such operators $\mathbf H$ \emph{periodic Hankel operators}; the usage of this term, perhaps unexpected at this point, will be justified below. 
It is not difficult to see that in terms of Hankel kernels the periodicity condition \eqref{a1} means that 
$$
e^T h (e^Tt) = h(t), \quad t>0.
$$
Alternatively, such Hankel kernels can be written as
$$
h(t)=\frac{p(\log t)}{t},
$$
where $p$ is a $T$-periodic function on $\bbR$, i.e. $p(\xi+T)=p(\xi)$. Thus, we consider kernel functions that are multiplicative log-periodic perturbations of the Carleman kernel $h(t)=1/t$. 

The simplest example is
\begin{equation}
p(\xi)=A+\cos(\omega \xi), 
\label{a2}
\end{equation}
with some $A\in\bbR$, where
$$
\omega=2\pi/T
$$
is the dual period.
Already in this simple example, the spectral analysis of the corresponding periodic Hankel operator is an interesting and non-trivial problem, which will be discussed in Section~\ref{sec.e} below. We call the corresponding periodic Hankel operator $\mathbf H$ the \emph{Mathieu-Hankel} operator, by analogy with the Mathieu operator 
$$
-\frac{d^2}{dx^2}+\cos x\quad\text{ in $L^2(\bbR)$,}
$$
see e.g. \cite[Section~XIII.16]{RS4}.
We discuss this analogy in more detail in Section~\ref{sec.a4} below.

\subsection{Main results}
For a $T$-periodic function $a$ on the real line, we denote by $\widetilde a_m$ the Fourier coefficients,
$$
\widetilde a_m=\frac1T\int_0^T e^{-im\omega x}a(x)dx, \quad m\in\bbZ,
$$
where, as above, $\omega$ is the dual period.
To be precise, we consider the following class of periodic Hankel operators:
\begin{definition}\label{def.smooth}
Let $h(t)=p(\log t)/t$, where $p$ is a real-valued $T$-periodic function with the  Fourier coefficients $\widetilde p_\ell$ satisfying 
\begin{align}\label{eq:pdec}
\sum_{\ell\in\bbZ} |\widetilde p_\ell|\, \jap{\ell}^{1/2} < \infty,
\end{align}
where $\jap{\ell}=\sqrt{1+\ell^2}$. 
In this case we will say that $\mathbf H$ is a \emph{smooth} periodic Hankel operator. 
\end{definition}
In particular, under the assumption \eqref{eq:pdec} the Fourier series for $p$ 
converges absolutely and therefore $p$ is a continuous function. By \eqref{eq:carleman1} it follows that $\mathbf H$ is bounded. In Section~\ref{sec.b6}  we will see that if a periodic Hankel operator is positive semi-definite, then it is automatically smooth. 

Our main results are as follows:

\begin{theorem}\label{thm.main}
Let $\mathbf H$ be a smooth periodic Hankel operator. Then: 
\begin{enumerate}[\rm (i)]
\item
All eigenvalues of $\mathbf H$ have infinite multiplicity. There are either finitely many (possibly none) eigenvalues of $\mathbf H$ or infinitely many with $0$ being one and the only accumulation point. 
The eigenvalues of $\mathbf H$ obey the reflection symmetry: if $\lambda\not=0$ is an eigenvalue, then $-\lambda$ is also an eigenvalue. 
\item
The absolutely continuous spectrum of $\mathbf H$ has multiplicity two and can be represented as  
\begin{align}\label{eq:ac}
\sigma_{\text{\rm ac}}(\mathbf H)={\rm Closure}(\cup_n\sigma_n),
\end{align}
where $\{\sigma_n\}$ is a finite (possibly empty) or countable collection of closed  
bounded intervals such that:
\begin{itemize}
\item
neither of the intervals $\sigma_n$ contains zero;
\item
if $\sigma_n$ and $\sigma_m$ are two intervals in this collection, then $\sigma_n\cap\sigma_m$ is either empty or consists of one point;
\item 
if $\sigma_n$ and $\sigma_m$ are two intervals in this collection, then $(-\sigma_n)\cap\sigma_m=\varnothing$;
\item 
if the collection $\{\sigma_n\}$ is infinite, then these intervals accumulate at zero (and at no other point). 
\end{itemize}
\item
The singular continuous spectrum of $\mathbf H$ is absent.
\end{enumerate}
\end{theorem}

In line with the terminology widely used in the theory of periodic Schr\"odinger operators (see Section~\ref{sec.a4} below), we call the intervals $\sigma_n$ of absolutely continuous spectrum the \emph{spectral bands}, and the non-zero eigenvalues of $\mathbf H$ of infinite multiplicity \emph{flat bands}. 

\begin{remark*}
\begin{enumerate}[1.]
\item
Examples show that non-trivial spectral bands may coexist with flat bands. In other words, there are smooth periodic Hankel operators $\mathbf H$ that have both non-zero eigenvalues and non-trivial absolutely continuous spectrum. In fact, such examples are given by \eqref{a2} with some special choices of $A$, see Section~\ref{sec.e} below. 
\item
Examples show  that there are smooth periodic Hankel operators with empty absolutely continuous spectrum, i.e. for these operators \emph{all spectral bands are flat}. 
One such example is furnished again by \eqref{a2} with $A=0$, see Section~\ref{sec.e} below. 
\item
From the reflection symmetry of the eigenvalues is follows that if $\mathbf H$ is a positive semi-definite smooth periodic Hankel operator, then it has no flat bands. 
\item
Examples show that the point $\lambda=0$ may or may not be an eigenvalue of $\mathbf H$; see Section~\ref{sec.d9} and Section~\ref{sec.trivker}. 

\item
If two spectral bands $\sigma_n$ and $\sigma_m$ 
have a common point (i.e. if they \emph{touch}), 
then from the point of view of the description of the spectrum of $\mathbf H$ they can be ``merged'' to form a  larger spectral band. It is, however, useful to think of them 
as two separate bands, because they correspond to different \emph{band functions}, see Theorem~\ref{thm.branches} and the discussion afterwards. 
\item
The curious property $(-\sigma_n)\cap\sigma_m=\varnothing$ can be rephrased by saying that he spectral bands of $-\mathbf H$ and $\mathbf H$ do not overlap, i.e. the absolutely continuous spectrum of the operator $\abs{\mathbf H}$ has multiplicity two. 
\end{enumerate}
\end{remark*}

The motivation for this paper comes from two distinct sources, which we describe below.

\subsection{Motivation I: The Megretskii-Peller-Treil theorem}
\begin{theorem}\cite{MPT}\label{thm.MPT}
A bounded self-adjoint operator $\mathbf H$ with a scalar spectral measure $\mu$ and the spectral multiplicity function $\nu$ is unitarily equivalent to a Hankel operator if and only if the following conditions are satisfied:
\begin{enumerate}[\rm (i)]
\item
either $\Ker \mathbf H=\{0\}$ or $\dim\Ker \mathbf H=\infty$;
\item
$\mathbf H$ is not invertible;
\item
$\abs{\nu(t)-\nu(-t)}\leq2$ almost everywhere with respect to the absolutely continuous part of $\mu$;
\item
$\abs{\nu(t)-\nu(-t)}\leq1$ almost everywhere with respect to the singular part of $\mu$. 
\end{enumerate}
\end{theorem}
The proof of the ``if'' part of this theorem is very difficult. Although in principle this proof is constructive, in practice is it very hard to exhibit an \emph{explicit} example of a Hankel operator with a prescribed spectral measure and spectral multiplicity function. 

The Carleman operator yields an example of a Hankel operator whose spectrum is purely absolutely continuous and coincides with a single interval. 
There are other constructions of Hankel operators whose absolutely continuous spectrum is the union of \emph{overlapping} intervals; see \cite{Howland,PYaf}  (and the foundational work \cite{Power}). 
It would be natural to ask whether there are explicit examples of Hankel operators whose spectrum is purely absolutely continuous and consists of the union of several \emph{disjoint} intervals. 
In this paper, we construct a natural family of examples of this type with explicit kernel 
functions $h(t)$. 

We note that the symmetry of eigenvalues in Theorem~\ref{thm.main}(i) agrees with Theorem~\ref{thm.MPT}(iv) and could be derived from it (even though our proof is independent of Theorem~\ref{thm.MPT}). At the same time, the curious property $(-\sigma_n)\cap\sigma_m=\varnothing$ of spectral bands in Theorem~\ref{thm.main}(ii) has a different nature.

\subsection{Motivation II: periodic Schr\"odinger operators}\label{sec.a4}
A Schr\"odinger operator in $L^2(\bbR)$ is the appropriately defined unbounded self-adjoint operator
$$
-\frac{d^2}{dx^2}+Q
$$
where $Q$ stands for the operator of multiplication by the real-valued function $q=q(x)$. If $q$ is periodic, i.e.  $q(x+T)=q(x)$, then the above operator is termed the \emph{periodic Schr\"odinger operator}. 
Periodicity of $q$ enables one to use the powerful method known as the \emph{Floquet-Bloch decomposition}, see \cite{Kuchment}, which leads to the band structure of the spectrum.

There are some obvious analogies between the theory of periodic Hankel operators that we develop in this paper and the theory of periodic Schr\"odinger operators in $L^2(\bbR)$; we will refer to them as the \emph{Hankel case} and \emph{Schr\"odinger case} for short. 
The group of dilations in $L^2(\bbR_+)$ in the Hankel case plays the same role as the group 
of shifts in $L^2(\bbR)$ in the Schr\"odinger case. As a consequence, 
some of our results for the Hankel case are completely analogous to the Schr\"odinger case. For example, for \emph{positive semi-definite} periodic Hankel operators all spectral bands are non-flat and the band functions exhibit the same \emph{alternation pattern} as in the Schr\"odinger case, see Theorem~\ref{thm:alterpos}. But the existence of flat bands for general smooth periodic Hankel operators is a new and rather unexpected phenomenon.

One should bear in mind one difference between the Schr\"odinger and Hankel cases: the Schr\"odinger operators are unbounded, while the Hankel operators that we consider are bounded. So a more accurate analogy would be to compare periodic Hankel operators with \emph{resolvents} of periodic Schr\"odinger operators. 
In particular, the spectral bands in the Schr\"odinger case accumulate at infinity, while the spectral bands in the Hankel case (if there are infinitely many of them) accumulate at zero. 

\subsection{The Floquet-Bloch decomposition and the band functions} 
In full analogy with the Schr\"odinger case, we rely on the Floquet-Bloch decomposition. 
This decomposition allows us to represent a smooth periodic Hankel operator $\mathbf H$ as a direct integral of compact self-adjoint \emph{fiber operators} $H(k)$ in $\ell^2(\bbZ)$, where the parameter $k$ ranges over the 
\emph{dual period cell} $\Omega=[-\omega/2,\omega/2)$:
\begin{align}\label{eq:fb}
\calU\mathbf H\calU^*=\int^{\oplus}_{\Omega}\,H(k)dk, \quad \Omega = [-\omega/2,\omega/2),
\end{align}
for a suitable unitary operator $\calU$. 
See Section~\ref{sec.d} for the precise definitions and statements. 
In the Schr\"odinger case, the fiber operators analogous to $H(k)$ have discrete spectrum but are not compact; instead they have compact resolvent. 

The operators $H(k)$ in $\ell^2(\bbZ)$ have the explicit matrix representation
\begin{equation}
[H(k)]_{n,m} = 
B\big(\tfrac12-i\omega n-ik, \tfrac12+i\omega m+ik\big)\, \wt p_{n-m}, \quad  n,m\in\bbZ,
\label{a3}
\end{equation}
where $B$ is the Beta function and $\wt p_n$ are the Fourier 
coefficients of the function $p$ in the kernel function $h(t)={p(\log t)}/{t}$. 
It is clear from \eqref{a3} that the operator valued function  
$H(k)$ can be extended to the complex plane as a meromorphic function with poles at the points of the lattice 
\begin{align}\label{eq:perl}
{\sf\Lambda}=\{ \omega n + i(m + \tfrac{1}{2}): n,m\in\bbZ\}.
\end{align} 
In particular, $H(k)$ is a self-adjoint analytic family 
(in the sense of \cite[Ch. VII, \S 3]{Kato}) 
in the strip $\abs{\Im\, k}<1/2$. 
It follows that the non-zero eigenvalues of $H(k), k\in\bbR,$ 
are real-valued real analytic functions of $k$, see Section~\ref{sec.dd}. 
We will prove that 
these can be extended to the whole real line as non-vanishing functions:
\begin{theorem}\label{thm.branches}
There is a finite or countable list $\mathcal E$ 
of \underline{non-vanishing} real-valued real analytic 
functions 
$E = E(k)$ on $\bbR$ that represent all non-zero eigenvalues of $H(k)$ in the following sense:
\begin{itemize}
\item
for each $k\in\bbR$ and each eigenvalue 
$E_*\not = 0$ of $H(k)$  of multiplicity $m\geq1$ 
there are exactly $m$ functions 
$E_1, E_2, \dots, E_m \in \mathcal E$ 
(not necessarily distinct, i.e. a function can be repeated on the list $\mathcal E$ according to multiplicity)
such that 
$E_* = E_1(k) = E_2(k) = \dots = E_m(k)$; 
\item
conversely, for each $k\in\bbR$ 
and each $E\in\mathcal E$ there is an eigenvalue of $H(k)$ that coincides with $E(k)$. 
\end{itemize}
Furthermore, if $E\in\mathcal E$, then the functions 
$E(-k)$ and $E(k+\omega n)$ are also in $\mathcal E$ for 
all $n$ (they are not necessarily distinct from $E$).
\end{theorem}

The last statement means that for every function $E$, 
either it is even or the reflected function 
$E(-k)$ is also on the list $\mathcal E$ 
(with the same interpretation for the shifted functions $E(k+\omega n)$).

We emphasize that we consider $\mathcal E$ as a \emph{list} rather than a \emph{set}, 
i.e. the same function can be listed several times in $\mathcal E$, 
which corresponds to multiplicities of eigenvalues. 
(In particular, a flat band can have multiplicity $>1$, see Section~\ref{sec.ddd-flat}.) 
However, we do not fix any specific order in the list $\mathcal E$.
We call the functions $E\in\mathcal E$ the \emph{band functions}. 
Their study is the main focus of this paper. 

It is a standard consequence of the Floquet-Bloch decomposition 
\eqref{eq:fb} that 
the operator $\left.\mathbf H\right|_{(\ker\, \mathbf H)^\perp}$ 
is unitarily equivalent to the direct sum of the operators of multiplication by 
the band functions $E\in\mathcal E$ in $L^2(\Omega)$, see Theorem~\ref{thm:osum}. 
Therefore the non-zero spectrum of $\mathbf H$ is the union of 
intervals formed by images of the functions $E\in\mathcal E$: 
\begin{align*}
\sigma = \{ E(k): k\in [0, \omega/2]\}. 
\end{align*}
Here we take only a half of the period cell $\Omega=[-\omega/2,\omega/2)$ because of the last part of Theorem~\ref{thm.branches}, i.e. because of the property $E(-k)\in\mathcal E$ for every $E\in\mathcal E$.
The intervals $\sigma$ are exactly the spectral bands featured in Theorem~\ref{thm.main}, see Section~\ref{subsect:proofmain} for details.
The flat bands are produced by constant (\textit{flat}) band functions, whereas 
the bands $\sigma_n$ of the absolutely continuous spectrum correspond to non-constant 
(\textit{non-flat}, or NF for short) 
band functions.   
The proof of Theorem~\ref{thm.main}(ii) is based on the 
fine properties of the NF band functions that are established in Sections~\ref{sec.dd} 
and \ref{sec.ddd}.

\subsection{The secular determinant and elliptic functions}
The centerpiece of our approach is the \emph{secular determinant}
$$
\Delta(k;\lambda)=\det(I-\lambda^{-1}H(k));
$$
initially, we consider it for $\lambda\in\bbR$, $\lambda\not=0$ 
and $k\in \bbR$. 
For smooth periodic Hankel operators, 
the operator $H(k)$ is trace class and so the determinant is well-defined. It is clear that the band functions $E(k)$ correspond to the zeros of the secular determinant, i.e. each band function is a 
real analytic branch of the solution of the equation $\Delta(k; E(k))=0$.

The matrix representation \eqref{a3} for $H(k)$ suggests that for any $\lambda\not=0$, the secular determinant can be extended to the whole complex plane as a meromorphic function of $k$
with simple poles at the points of the lattice \eqref{eq:perl}. 
This is indeed the case, and this function turns out to be \emph{elliptic with the period lattice}
\begin{align}\label{eq:perm} 
{\sf M} = \{\omega m + 2in, \ m, n\in\bbZ\}.
\end{align}
As $k\mapsto \Delta(k; \lambda)$ has exactly two poles (at $i/2$ and $-i/2$) 
in the fundamental domain $\bbC/{\sf M}$ 
and they are simple, $\Delta(k; \lambda)$ is an elliptic function of 
order two, see Section~\ref{sec.cc}.

The properties of this elliptic function translate into the properties of the band functions. For example, since an elliptic function of order two takes every value 
exactly twice on the fundamental domain, the non-flat band functions 
are strictly monotone (with a non-vanishing derivative) 
on $(0,\omega/2)$, and the spectral bands $\sigma_n$ and $\sigma_m$ for $n\not = m$ in 
Theorem~\ref{thm.main} may have at most one 
common point, see Sections~\ref{sec.dd} and \ref{sec.ddd}. These facts eventually ensure that 
the absolute continuous spectrum of $\mathbf H$ has multiplicity two, as stated in Theorem~\ref{thm.main}.

\subsection{The structure of the paper}
In Section~\ref{sec.d} we develop the Floquet-Bloch decomposition for $\bH$ 
and prove the matrix representation formula for $H(k)$. 
In Section~\ref{sec.c} we show that $H(k)$ extends to the 
complex plane as a meromorphic function. 
In Section~\ref{sec.cc} we introduce the secular 
determinant $\Delta(k;\lambda)$ and prove its key properties. 
Sections~\ref{sec.dd} and \ref{sec.ddd} are central to the paper; 
here we prove Theorems~\ref{thm.main} and \ref{thm.branches} and establish finer properties of the band functions. 
In Section~\ref{sec.alternation} we establish a global property of the band functions that we call the alternation pattern; it is related to the sign of monotonicity (i.e. is it increasing or decreasing?) of a given band function on the interval $(0,\omega/2)$. We prove that this sign of monotonicity alternates from band to band.
 In Section~\ref{sec.e} we consider in detail the Mathieu-Hankel operator \eqref{a2}. 

\subsection{Notation and terminology}
Fourier transform in $L^2(\bbR)$: we denote
\begin{equation}
\widehat f(\xi)=\frac1{\sqrt{2\pi}}\int_{-\infty}^\infty f(x)e^{-i\xi x}dx
\quad\text{ and }\quad
\widecheck{g}(x)=\frac1{\sqrt{2\pi}}\int_{-\infty}^\infty g(\xi)e^{i\xi x}d\xi.
\label{eq:Fourier}
\end{equation}
We use $\norm{\cdot}$ for the operator norm and $\norm{\cdot}_{\mathbf S_1}$ for the trace norm. 
The period $T>0$ is fixed and $\omega=2\pi/T$ is the dual period. $\Gamma$ is the Gamma function and $B$ is the Beta function. For $x\in\bbR$, we write $\jap{x}=\sqrt{1+x^2}$.

If $v = \{v_n\}_{n\in\bbZ}$ is a bounded (double-sided) sequence 
of complex numbers, we will denote by $\calM[v]$ the operator of multiplication 
by this sequence in $\ell^2(\bbZ)$. In other words, the operator $\calM[v]$ 
is the diagonal matrix with the entries $v_n$ on the diagonal.  
We denote by $J$, $S$ and $K$ the following operators in $\ell^2(\bbZ)$:
\begin{equation}
(J u)_n = u_{-n}, \quad (Su)_n = u_{n-1},\quad (Ku)_n=(-1)^nu_n, \quad u\in \ell^2(\bbZ).
\label{eq:defJS}
\end{equation}
By a slight abuse of notation, we also denote by $\calM[a]$ the operator of multiplication by a bounded function  
$a$ in the corresponding $L^2$ space with respect to the Lebesgue measure.

By $C, c$ with or without indices, we denote positive constants whose precise value is unimportant. The value of $C, c$ may vary from line to line.

\subsection{Acknowledgements}
We are grateful to Jean Lagac\'e and Gerald Teschl for useful discussions. 
 We are also grateful to the referees for the careful reading of the manuscript and many constructive suggestions.

\section{The Floquet-Bloch decomposition}
\label{sec.d}
\subsection{Overview}
In Sections~\ref{sec.d2}--\ref{sec.d3}, we define the unitary operators that give a direct integral decomposition of $L^2(\bbR_+)$, adapted to operators commuting with dilations.  In Section~\ref{sec.d4} we state the main result, which is the matrix representation formula \eqref{a3} for $H(k)$. The proof is given in Sections~\ref{sec.d5}--\ref{subsect:spho}. In Section~\ref{sec.b6} we focus on the case of positive semi-definite Hankel operators, and in Section~\ref{sec.d9} we consider an example.

\subsection{Background: standard Floquet-Bloch decomposition of $L^2(\bbR)$} 
\label{sec.d2}
First, as a warm-up we display the formulas corresponding to the standard Floquet-Bloch decomposition of $T$-periodic operators on $L^2(\bbR)$.

Let $\ell^2=\ell^2(\bbZ)$ and let us consider the constant fiber direct integral of Hilbert spaces:
\[
\mathfrak H=\int_{\Omega}^\oplus \ell^2\,  dk. 
\]
In other words, $\mathfrak H$ is the $L^2$-space of $\ell^2$-valued functions $F=F(k)$ on $\Omega$ with respect to the Lebesgue measure.
Clearly, $\mathfrak H$ is isomorphic to $L^2(\bbR)$ in a natural way: a function $f\in L^2(\bbR)$ corresponds to $F(k)=\{f(\omega m+k)\}_{m\in\bbZ}$, $k\in \Omega$. 

We define the unitary map $U:L^2(\bbR)\to\mathfrak H$ by 
\[
[(Uf)(k)]_n=\widehat f(k+\omega n)=\frac1{\sqrt{2\pi}}\int_{-\infty}^\infty f(x)e^{-ikx-i\omega nx}dx.
\]
The map $U$ is the combination of the Fourier transform $f\mapsto \widehat f$ and the natural isomorphism between $L^2(\bbR)$ and $\mathfrak H$, referred to in the previous paragraph. 
Observe that $U$ transforms the unitary operator in $L^2(\bbR)$ corresponding to the shift of variable $x\mapsto x+T$ into the operator of multiplication by $e^{ikT}$ in $\mathfrak H$.  

It follows that the map $U$ effects the Floquet-Bloch decomposition of $T$-periodic (i.e. invariant under shifts $x\mapsto x+T$) operators on $L^2(\bbR)$. The precise statement is as follows:
\begin{proposition}\label{prp.b1}
Let $A$ be a bounded self-adjoint operator in $L^2(\bbR)$ that commutes with the shift operator $f(x)\mapsto f(x+T)$. Then $A$ is decomposable as
\begin{equation}
U AU^*=\int_{\Omega}^\oplus A(k)\, dk,
\label{directint}
\end{equation}
with some bounded self-adjoint fiber operators $A(k)$ in $\ell^2$. 
\end{proposition}
\begin{proof} 
As $A$ commutes with shifts, we find that $UAU^*$ commutes with the operator of multiplication by $e^{ikT}$ in $\mathfrak H$. By taking linear combinations and operator norm limits, we find that $UAU^*$ commutes with the operators of multiplication by any \emph{continuous} function of $k$ in $\mathfrak H$. 
Moreover, by taking limits in strong operator topology, we find that $UAU^*$ also commutes with the operators of multiplication by any \emph{bounded} function of $k$ in $\mathfrak H$.
The proof is now complete by reference to \cite[Theorem XIII.84]{RS4}. 
\end{proof}
Moreover, if $A$ is a differential or pseudodifferential operator, the fiber operators $A(k)$ can be explicitly computed, which will be useful for us. In order to recall the corresponding formulas, it is convenient to introduce the standard pieces of notation: let $X$ and $D$ be the (unbounded) self-adjoint operators on $L^2(\bbR)$, defined by
$$
(Xf)(x)=xf(x), \quad (Df)(x)=-if'(x)
$$ 
on suitable dense domains. For any bounded measurable function $g$ on $\bbR$,  one can define the operators $g(X)$ (multiplication by $g(x)$) and $g(D)$ (convolution with $(2\pi)^{-1/2}\widecheck g$, see \eqref{eq:Fourier}). The formulas below are well-known in the theory of periodic Schr\"odinger operators, see e.g.  \cite[Theorem~XIII.87]{RS4}:

\begin{proposition}\label{prp.b2}
\begin{enumerate}[\rm (i)]
\item
If $b$ is a real-valued Schwartz class function and $A=b(D)$, then the fiber operator $A(k)$ in \eqref{directint} is the operator of multiplication by the sequence $\{b(k+\omega n)\}_{n\in\bbZ}$ in $\ell^2$. 
\item
If $a$ is a real-valued $T$-periodic continuous function on $\bbR$ and $A=a(X)$, then the fiber operator $A(k)$ in \eqref{directint} is the the constant (i.e. $k$-independent) operator in $\ell^2$ with the matrix
\[
\bigl[A(k)]_{n,m}= \widetilde a_{n-m}, \quad n,m\in\bbZ,
\]
where $\widetilde a_m$ are the Fourier coefficients of $a$.
\end{enumerate}
\end{proposition}

\subsection{The exponential change of variable}
\label{sec.d3}
Let $N:L^2(\bbR_+)\to L^2(\bbR)$ be the unitary operator corresponding to the exponential change of variable:
$$
(Nf)(x)=e^{x/2}f(e^x), \quad x\in\bbR.
$$
Clearly, $N$ maps dilations by the factor of $e^T$ in $L^2(\bbR_+)$ to shifts by $T$ in $L^2(\bbR)$. 

We use the operator 
\begin{align*}
\calU&=UN: L^2(\bbR_+)\to \mathfrak H, 
\\
[(\calU f)(k)]_m&=\frac1{\sqrt{2\pi}}\int_0^\infty f(t)t^{-\frac12-ik-i\omega m}dt
\end{align*}
in order to effect the Floquet-Bloch decomposition of periodic Hankel operators in $L^2(\bbR_+)$. 
Combining Proposition~\ref{prp.b1} with the exponential change of variable, we obtain

\begin{lemma}\label{lma.b2}
Let $A$ be a bounded self-adjoint operator in $L^2(\bbR_+)$ that commutes with the dilation operator $V_T$. Then $A$ is decomposable as
$$
\calU A\calU^*=\int_{\Omega}^\oplus A(k)\, dk,
$$
with some bounded self-adjoint fiber operators $A(k)$.
\end{lemma}

\subsection{The fiber operators: introduction}
\label{sec.d4}
Let $\mathbf H$ be a smooth periodic Hankel operator (see Definition~\ref{def.smooth}) with the kernel function $h(t)=p(\log t)/t$. Since $\abs{h(t)}\leq C/t$, the operator $\mathbf H$ is bounded on $L^2(\bbR_+)$. 
Since $p$ is periodic, the operator $\mathbf H$ commutes with dilations $V_T$ and therefore, by Lemma~\ref{lma.b2}, we have 
\begin{equation}
\calU\mathbf H\calU^*=\int_{\Omega}^\oplus H(k)\, dk,
\label{b4}
\end{equation}
for some bounded operator $H(k)$ in $\ell^2(\bbZ)$. 
The operator $H(k)$ can be considered as a doubly-infinite matrix 
$[H(k)]_{n,m}$ with $n,m\in\bbZ$. 
Our aim in this section is to identify this matrix. 
Below $\Gamma$ is the gamma function and 
\begin{equation}
B(\zeta_1, \zeta_2) = \frac{\Gamma(\zeta_1) \Gamma(\zeta_2)}{\Gamma(\zeta_1+\zeta_2)},\quad \zeta_1, \zeta_2\in\mathbb C, 
\label{eq:beta}
\end{equation}
is the beta function. The main result of this section is

\begin{theorem}\label{thm.b3}
Let $\bH$ be a smooth periodic Hankel operator. 
Then the fiber operator $H(k)$ in \eqref{b4} is a continuous trace class valued function of $k\in\Omega$ and has the matrix representation 
\begin{equation}
\label{b1}
[H(k)]_{n,m} = 
B\big(\tfrac12-i\omega n-ik, \tfrac12+i\omega m+ik\big)\, \wt p_{n-m}, \quad  n,m\in\bbZ.
\end{equation}
\end{theorem}

In Section~\ref{sec.c} below we shall see that in fact $H(k)$ is not only continuous in $k$, but real analytic and extends to complex $k$. 

Before embarking on the proof of Theorem~\ref{thm.b3}, we discuss an equivalent form of the matrix representation \eqref{b1}. Using \eqref{eq:beta}, we rewrite this representation as 
\begin{equation}
\label{eq:bk}
[H(k)]_{n,m} = \overline{\gamma_n(k)} \widetilde{\mathfrak s}_{n-m}\gamma_m(k), 
\end{equation}
where
\begin{align}\label{eq:sfunct}
\wt{\mathfrak s}_n = \frac{\wt p_n}{\Gamma(1-i\omega n)}, 
\quad
\gamma_n(k)  = \Gamma\big(\tfrac12 + i(\omega n + k)\big). 
\end{align}
(Notation $\wt{\mathfrak s}_n$ is motivated in a few lines below.)
The utility of \eqref{eq:bk} comes from the fact that, at least formally, 
it represents $H(k)$ as the product of three operators in $\ell^2(\bbZ)$: multiplication by 
$\gamma_n(k)$, convolution with the sequence $\widetilde{\mathfrak s}_n$ 
and again multiplication by $\overline{\gamma_n(k)}$. 
The technical difficulty here is that $\Gamma(1-i\omega n)\to0$ 
exponentially fast as $\abs{n}\to\infty$ and therefore the sequence 
$\widetilde{\mathfrak s}_n$ may grow exponentially; it follows that 
the operator of convolution with this sequence  is in general ill-defined. 
We circumvent this difficultly by first considering the case of coefficients 
$\widetilde p_n$ of rapid decay, so that 
\begin{align}\label{eq:conv}
\sum_{n=-\infty}^\infty |\wt{\mathfrak s}_n| <\infty
\end{align}
and then passing to the limit. Under the condition \eqref{eq:conv}, the periodic function 
\begin{align}\label{eq:sfunct1}
\mathfrak s(\xi)=\sum_{\ell=-\infty}^\infty \wt{\mathfrak s}_\ell e^{i\omega\ell\xi} 
\end{align}
is bounded. We recall that the operator of convolution with the sequence $\wt{\mathfrak s}_n$, i.e. the operator $\mathfrak S$ in $\ell^2(\mathbb Z)$ with the matrix entries
\begin{align}\label{eq:umn}
{\mathfrak S}_{n, m} = \wt{\mathfrak s}_{n-m}, \quad n, m\in\mathbb Z, 
\end{align}
is unitarily equivalent to the operator of multiplication by $\mathfrak s(\xi)$ in $L^2(0,T)$; in particular, this operator is bounded in $\ell^2(\bbZ)$. 
Furthermore, 
due to the well-known formula 
\begin{align}\label{eq:gamco}
\Abs{\Gamma\big(\tfrac{1}{2} + ix\big)}^2 = \frac{\pi}{\cosh \pi x}, \quad x\in\bbR,
\end{align}
the operator $\calM(\gamma)$ of multiplication by the sequence $\gamma$ in $\ell^2(\bbZ)$ belongs to trace class and depends continuously on $k\in\bbR$. 
Therefore, under the condition \eqref{eq:conv} the operator 
\eqref{eq:bk} is trace class and continuous in $k$.

\subsection{The case of $\widetilde p_n$ of rapid decay}
\label{sec.d5}

\begin{theorem}\label{thm.b1}
Let $\mathbf H$ be a smooth periodic Hankel operator, and let $\wt{\mathfrak s}_n$ be as defined in \eqref{eq:sfunct}. Assume that the series \eqref{eq:conv} converges. Then the kernel function $h(t)$ has the integral representation 
\begin{align}\label{eq:laplace}
h(t)=
\frac{p(\log t)}{t}
= \int_0^\infty e^{-\lambda t}\mathfrak s(\log(1/\lambda))d\lambda, \quad t>0,
\end{align}
with the $T$-periodic function $\mathfrak s$ given by the absolutely convergent series \eqref{eq:sfunct1}. 
Furthermore, the fiber operator $H(k)$ in the direct integral \eqref{b4} can be represented as in \eqref{eq:bk} or \eqref{b1}, and $H(k)$ is a continuous trace-class valued function of $k\in\Omega$. 
\end{theorem}
\begin{proof} 
By the integral representation for the Gamma function, for every $\ell\in\bbZ$ we have
\begin{align*}
\Gamma(1-i\omega\ell)&=\int_0^\infty e^{-\lambda}\lambda^{-i\omega\ell}d\lambda
\\
&=t^{1-i\omega\ell}\int_0^\infty e^{-\lambda t}\lambda^{-i\omega\ell}d\lambda, \quad t>0.
\end{align*}
Let us rewrite this as 
\begin{equation}
\frac{e^{i\omega\ell \log t}}{t}
=
\frac1{\Gamma(1-i\omega\ell)}\int_0^\infty e^{-\lambda t}e^{i\omega\ell\log(1/\lambda)}d\lambda.
\label{eq:gamma}
\end{equation}
Multiplying this by $\wt  p_\ell$ and summing over $\ell$ (the series converges absolutely for every $t>0$ by \eqref{eq:conv}), we immediately obtain the  representation \eqref{eq:laplace}. After the exponential change of variable $\lambda=e^{-\xi}$ this rewrites as  
$$
h(t)=\int_{-\infty}^\infty e^{-te^{-\xi}}\mathfrak s(\xi)e^{-\xi}d\xi,
$$
and therefore 
\begin{align*}
h(e^x+e^y)&=
\int_{-\infty}^\infty e^{-(e^x+e^y)e^{-\xi}}\mathfrak s(\xi)e^{-\xi}d\xi
\notag
\\
&=e^{-(x+y)/2}
\int_{-\infty}^\infty e^{-(e^{x-\xi}+e^{y-\xi})}\mathfrak s(\xi)e^{(x-\xi)/2}e^{(y-\xi)/2}d\xi
\notag
\\
&=e^{-(x+y)/2}\int_{-\infty}^\infty \beta(x-\xi)\mathfrak s(\xi)\beta(y-\xi)d\xi,
\end{align*}
where
$$
\beta(\xi)=e^{-e^\xi}e^{\xi/2}. 
$$
We conclude that the operator $N$ of the exponential change of variable transforms our Hankel operator into the integral operator $N\mathbf HN^*$ in $L^2(\bbR)$ with the integral kernel
$$
[N\mathbf HN^*](x,y)=\int_{-\infty}^\infty \beta(x-\xi)\mathfrak s(\xi)\beta(y-\xi)d\xi, \quad x,y\in\bbR.
$$
Thus, $N\mathbf HN^*$ is the product (right to left) of the convolution with $\beta(-x)$, the multiplication by $\mathfrak s(\xi)$ and the convolution with $\beta(x)$. 
Observe that 
\begin{align*}
\widehat \beta(u)&=\frac1{\sqrt{2\pi}}\int_{-\infty}^\infty \beta(x)e^{-iux}dx
=\frac1{\sqrt{2\pi}}\int_{-\infty}^\infty e^{-e^x}e^{x/2}e^{-iux}dx
\\
&=\frac1{\sqrt{2\pi}}\int_0^\infty e^{-\lambda}\lambda^{-\frac12-iu}d\lambda
=\frac1{\sqrt{2\pi}}\Gamma(\tfrac12-iu),
\end{align*}
and so the convolution with $\beta$ can be written as $\Gamma(\frac12-iD)$ (see Section~\ref{sec.d2}). 
Putting this together, we find 
\begin{equation}
N\mathbf HN^*=\Gamma(\tfrac12-iD)\mathfrak s(X)\Gamma(\tfrac12-iD)^*
\label{b10}
\end{equation}
in $L^2(\bbR)$. Applying the unitary map $U$ we arrive at 
$$
\calU\mathbf H\calU^*
=
UN\mathbf H N^*U^*
=
(U\Gamma(\tfrac12-iD)U^*)(U\mathfrak s(X)U^*)(U\Gamma(\tfrac12-iD)^*U^*).
$$
By Proposition~\ref{prp.b2}, this implies that the entries of the operator 
$H(k)$ are given by  
\begin{align}\label{eq:bs}
[H(k)]_{n, m} = \Gamma(\tfrac12 - i\omega n - ik) \,
\wt {\mathfrak s}_{n - m}\, \Gamma(\tfrac12 + i\omega m + ik). 
\end{align} 
This leads to \eqref{eq:bk}. 
Using the definition \eqref{eq:sfunct} we obtain \eqref{b1}.

As pointed out before the Theorem, the operator \eqref{eq:bk} is trace class 
and it depends continuously on $k\in\bbR$ in trace norm. It follows that $H(k)$ is a continuous function of $k\in\bbR$ with values in trace class.  
\end{proof}

\begin{remark*} 
The representation \eqref{b10} of a Hankel operator 
as a pseudodifferential operator is due to D.Yafaev \cite{yaf_apde_2015,Ya1}, who worked with different (non-periodic) classes of Hankel operators. 
\end{remark*}

\subsection{Extension to smooth periodic Hankel operators}\label{subsect:spho}

The proof of Theorem~\ref{thm.b3} is based on the following technical statement. 

\begin{lemma}\label{lem:bpoly}
Let $p$ be a trigonometric polynomial. Then the operator $H(k)$, defined by \eqref{b1} for all $k\in\bbR$, satisfies the trace class bound
\begin{equation}
\norm{H(k)}_{\mathbf S_1}
\leq 
C\sum_{\ell=-\infty}^\infty\abs{\wt p_\ell}\jap{\ell}^{1/2}
\label{b3}
\end{equation}
uniformly in $k\in\mathbb R$. 
\end{lemma}
It will be convenient to postpone the proof of this lemma to Section~\ref{sec.c}, where the case of complex $k$ is also considered; see Remark~\ref{rem:bpoly}. Here, assuming this lemma, we complete the proof of Theorem~\ref{thm.b3}. 

\begin{proof}[Proof of Theorem~\ref{thm.b3}]
Recall that by assumption $\mathbf H$ is smooth, i.e. the series in \eqref{b3} converges. 
Let $p^{(N)}$ be the trigonometric polynomial obtained by truncating the series for $p$:
$$
p^{(N)}(\xi)=\sum_{\ell=-N}^N \wt p_\ell e^{i\omega\ell\xi}, 
$$
let $h^{(N)}(t) = p^{(N)}(\log t)/t$, and let $\mathbf H^{(N)}$ be the corresponding Hankel operator. For each $N$, by Theorem~\ref{thm.b1} we have the representation 
\begin{equation}
\calU\mathbf H^{(N)}\calU^*=\int_{\Omega}^\oplus  H^{(N)}(k)\, dk,
\label{b5}
\end{equation}
where $H^{(N)}(k)$ is the infinite matrix whose entries $[H^{(N)}(k)]_{n, m}$ are given by 
\eqref{b1} for all $m, n$ such that $|m-n|\le N$ and by $[H^{(N)}(k)]_{n, m} = 0$ otherwise. 
Consider both sides of \eqref{b5}. 
Observe that $p^{(N)}(\xi)\to p(\xi)$ uniformly over $\xi$ and therefore, by \eqref{eq:carleman1}, 
$$
\norm{\mathbf H^{(N)}-\mathbf H}\to0, \quad N\to\infty. 
$$
On the other hand, by \eqref{b3}, the operators $H^{(N)}(k)$ 
converge in trace norm uniformly over $k$. 
It follows that the integrals on the right-hand side 
of \eqref{b5} also converge in the operator norm. 

We conclude that the fiber operators $H(k)$ 
in the direct integral decomposition  \eqref{b4} 
for $\mathbf H$ are the trace norm limits of $H^{(N)}(k)$. 
In particular, $H(k)$ is trace class for every $k$ and is continuous (in trace norm) in $k\in\bbR$.
Finally, taking the limit of $[H^{(N)}(k)]_{n, m}$ as $N\to\infty$ 
we obtain formula \eqref{b1} for the matrix of $H(k)$.
\end{proof}

\subsection{Positive periodic Hankel operators: automatic smoothness}\label{sec.b6}
Here we discuss the ``automatic smoothness'' phenomenon: 
the kernel function $h(t)$ of a positive semi-definite periodic Hankel operator is infinitely differentiable, and therefore positive semi-definite periodic Hankel operators are automatically smooth.

First we need a statement which can be regarded as folklore, see e.g. \cite{Widom66,Ya1}. 
\begin{proposition}
\label{prop:positive}
Let $\mathbf H$ be a  bounded positive semi-definite Hankel 
operator with  a continuous real-valued kernel function $h(t), t\in (0, \infty)$.  
Then  
there exists a unique positive measure $d\sigma$ on $[0,\infty)$ such that 
\begin{align}\label{eq:positive}
h(t) = \int_0^\infty e^{-\lambda t} d\sigma(\lambda), \quad t>0,
\end{align}
and the following properties are satisfied:
\begin{enumerate}[\rm (i)]
\item \label{item:meas}
The measure $\sigma$ 
satisfies the bound $\sigma(\lambda) :=\sigma([0, \lambda])\le C\lambda$ for all $\lambda>0$. In particular, $h(t)$ is infinitely differentiable for $t >0$. 
\item \label{item:ker}
The kernel function $h$ satisfies the bound $|h(t)|\le Ct^{-1}$ for all $t >0$.
\end{enumerate}
\end{proposition}

Instead of the a priori assumption of continuity, in Proposition~\ref{prop:positive} 
one can start with  kernel functions $h(t)$ that are ``positive" in the sense of distributions, see e.g. 
\cite[Theorems 5.1 and 5.3]{Ya1}. For our purposes however 
it is sufficient to assume that $h$ is continuous 
on $(0, \infty)$. 
For completeness we provide a proof of Proposition~\ref{prop:positive} in the appendix.

For the case of positive \emph{periodic} Hankel operators we have the following corollary:

\begin{lemma}
\label{lma:positive}
Let $\mathbf H$ be a positive semi-definite Hankel operator with the kernel function
$h(t) = p(\log t)/t$, where $p$ is a $T$-periodic continuous function. 
Let $\sigma$ be the measure in \eqref{eq:positive}. Then:
\begin{enumerate}[\rm (i)]
\item
the Fourier coefficients $\wt p_\ell$ decay faster than any 
power of $\ell$ and therefore by Theorem~\ref{thm.b3}, 
$H(k)$ is trace class and has the matrix representation \eqref{b1};
\item
the measure $d\mu$ on $(-\infty, \infty)$ defined by 
$d\mu (\xi) =e^{\xi}  d\sigma(e^{-\xi})$, 
is $T$-periodic, i.e. ${\mu}((a, b)) = {\mu}((a+T, b+T))$ for all $a, b\in \mathbb R$;
\item
the Fourier coefficients $\wt{\mathfrak s}_\ell$ of the measure $\mu$ satisfy
$$
\wt{\mathfrak s}_\ell
=
\frac{\wt{p}_\ell}{\Gamma(1-i\omega\ell)},
$$
in agreement with our earlier notation \eqref{eq:sfunct}.
\end{enumerate}
\end{lemma}
\begin{proof}
(i) Since $h$ is infinitely differentiable, 
the function $p(\xi) = e^\xi h(e^\xi)$ is also infinitely differentiable, which implies (i).

(ii) Similarly to the proof of Theorem~\ref{thm.b1}, after the  exponential change of  variables $\lambda = e^{-\eta}$, formula \eqref{eq:positive} rewrites as 
\begin{align*}
p(\xi) = \int_{-\infty}^\infty e^{-e^{\xi-\eta}} e^{\xi-\eta}\, d{\mu}(\eta). 
\end{align*}
Periodicity of $p$ implies that 
$$
\int_{-\infty}^\infty e^{-e^{\xi-\eta}} e^{\xi-\eta}\, d{\mu}(\eta)
=
\int_{-\infty}^\infty e^{-e^{\xi-\eta}} e^{\xi-\eta}\, d{\mu}(\eta+T)
$$
and by the uniqueness of $\sigma$ in \eqref{eq:positive} we find that $d{\mu}$ is $T$-periodic. 

(iii) is a direct calculation using \eqref{eq:gamma}.
\end{proof}

Clearly, if $p$ is a trigonometric polynomial (or, more generally, if $\sum_\ell\abs{\wt{\mathfrak s}_\ell}<\infty$), then the measure $d{\mu}$ is absolutely continuous and can be written as $d{\mu}(\xi)=\mathfrak s(\xi)d\xi$ with a non-negative $T$-periodic weight function $\mathfrak s$; this function was used in the proof of Theorem~\ref{thm.b1}.

\subsection{Positive periodic Hankel operator: example}
\label{sec.d9}
Fix $\alpha\in[0,T)$ and let $d\mu$ be the $T$-periodic  purely atomic measure giving weight one to each of the points $\alpha+Tn$, $n\in\bbZ$:
$$
d\mu(\xi)=\sum_{n=-\infty}^\infty \delta(\xi-\alpha-Tn), 
$$
where $\delta$ is the delta-function.
Then $\wt{\mathfrak s}_n=(1/T)e^{i\alpha\omega n}$ and 
\[
[H(k)]_{n,m} = \frac{1}{T}\,
\Gamma(\tfrac12-i\omega n-ik)e^{i\alpha\omega(n-m)}\Gamma(\tfrac12+i\omega m+ik),
\]
which is a rank one operator. It is easy to see that in this case the 
corresponding Hankel operator  has a single spectral band and an infinite-dimensional kernel. 
This construction can be easily extended to the case of $d\mu$ being a finite linear combination of point masses on a period. This gives rise to \emph{finite band operators}. We plan to discuss this interesting class of operators elsewhere.

\section{The analytic extension of $H(k)$}
\label{sec.c}

\subsection{Main statement}
Let $\mathbf H$ be a smooth periodic Hankel operator, i.e. the series \eqref{eq:pdec}
for the Fourier coefficients $\wt p_n$ converges. We extend the definition \eqref{b1} of $H(k)$ to the complex domain, i.e. for complex $s$ we set
\begin{align}\label{eq:bentr}
[H(s)]_{n,m} = 
B\big(\tfrac12-i\omega n-is, \tfrac12+i\omega m+is\big)\, \wt p_{n-m}, \quad n,m\in\bbZ.
\end{align}
Since $B(a,b)$ is well defined for  all $a,b\not=0,-1,-2,\dots$, this definition makes sense for all $s\in\bbC\setminus {\sf\Lambda}$, where ${\sf\Lambda}$ is the 
lattice defined in \eqref{eq:perl}.
In this technical section, our objective is to show that under appropriate assumptions, 
the operator $H(s)$ is analytic in $s\in\bbC\setminus{\sf\Lambda}$ with values in trace class.  

We need two pieces of notation. 
First, for $\tau>0$ and $\delta\geq0$ we introduce the ``punctured strip''
\begin{align}
K_{\tau, \delta} = \{s\in\bbC: |\Im s| < \tau, \ 
\dist\,(s, {\sf \Lambda})> \delta\}.
\label{eq:strip}
\end{align} 
Next,  for $\tau\in\bbR$ let us define the weight function $\rho_\tau$ on $\bbZ$ by  
\begin{align}\label{eq:weight}
\rho_\tau(n)=
\begin{cases}
\jap{n}^{\frac12}, &\text{ if $\abs{\tau}<1$,}
\\
\jap{n}^{\frac12}\log(2+n^2), &\text{ if $\abs{\tau}=1$,}
\\
\jap{n}^{\abs{\tau}-\frac12}, &\text{ if $\abs{\tau}>1$.}
\end{cases}
\end{align}
The main result of this section is
\begin{theorem}\label{thm:btrace}
Suppose that for some $\tau_0>0$ the series 
$$
\sum_{\ell=-\infty}^\infty\abs{\widetilde p_\ell}\rho_{\tau_0}(\ell)
$$
converges. Then the operator $H(s)$, defined by \eqref{eq:bentr}, is a trace class valued analytic function of $s\in K_{\tau_0,0}$. 
Moreover, in every punctured strip $K_{\tau_0, \delta}$ with $\delta>0$ we have the estimate 
\begin{align}\label{eq:nontrunc}
\norm{H(s)}_{\mathbf S_1}\leq C 
\sum_{\ell=-\infty}^\infty\abs{\widetilde p_\ell}\rho_{\tau_0}(\ell),
\end{align}
with a constant $C$ depending only on $\tau_0$ and $\delta$. 
\end{theorem}

In the rest of this section, we give the proof.

\subsection{Auxiliary estimates for the Gamma function}
We need to prepare some simple estimates for the Gamma function. For these estimates we use Stirling's formula 
\begin{align}\label{eq:gamas}
\Gamma(z) = \sqrt{2\pi} z^{z-1/2} e^{-z} \big[ 1+ O(1/z)\big],\quad |\arg z|\le \pi-\varepsilon,
\quad |z|\to\infty,
\end{align}
where the constant in $O(1/z)$ depends only on $\varepsilon>0$.

\begin{lemma} \label{lem:gamma}
Assume that $a, b\in \bbR$ are such that $|a-1/2|\le A$ and $|a+ib+n|\ge \delta>0$ 
for all $n = 0, 1, \dots$. Then 
\begin{align}\label{eq:gm}
c \,\jap{b}^{a-1/2} e^{-\pi|b|/2}\le |\Gamma(a + ib)| \le C\, \jap{b}^{a-1/2} e^{-\pi|b|/2},
\end{align}
where the constants $c$ and $C$ depend only on $A$ and $\delta$. 
\end{lemma}
\begin{proof} Since $\Gamma(a + ib) = \overline{\Gamma(a - ib)}$ it suffices to 
consider the case $b \ge 0$. Furthermore, since $\Gamma(z)$ is continuous (and hence bounded) 
on the compact set
\begin{align*}
\{z = a+ib: |a-1/2|\le A,\, |b|\le 1,\, |a+ib+n|\ge \delta, \ n = 0, 1, \dots\},
\end{align*}
we may assume that $b\ge 1$. 
According to Stirling's formula \eqref{eq:gamas}, 
\begin{align*}
\Gamma(a + ib)
= &\ \sqrt{2\pi}\, \mathcal A(a, b) (ib)^{a+ib-1/2}\big(1+O(b^{-1})\big),\\
\mathcal A(a, b) = &\ \bigg[1+\frac{a}{ib}\bigg]^{a+ib-1/2} e^{-a-ib},
\end{align*}
where the constant in $O(b^{-1})$ may depend on $a$. 
We estimate $\mathcal A(a, b)$:
\begin{align}\label{eq:atwo}
c(a) \bigg|\bigg[1+\frac{a}{ib}\bigg]^{ib}\bigg|\le  |\mathcal A(a, b)|\le C(a)
 \bigg|\bigg[1+\frac{a}{ib}\bigg]^{ib}\bigg|,
\end{align}
with some positive constants $c(a), C(a)$.  
As 
\begin{align*}
\bigg|\bigg[1+\frac{a}{ib}\bigg]^{ib}\bigg|
= \exp\big(- b \arg (1-i a/b)\big)
\end{align*}
and 
$|\arg  (1-i a/b)|\le |a b^{-1}|$, the left-hand side and the right-hand side of 
\eqref{eq:atwo} are bounded and separated from zero, i.e. 
the function $\mathcal A(a, b)$ satisfies the bounds  
\begin{align*}
\wt c(a) \le |\mathcal A(a, b)| \le \wt C(a), 
\end{align*}
with some new positive constants $\wt c(a)$ and $\wt C(a)$. Furthermore,
\begin{align*}
(ib)^{ib} = \exp(ib\log (ib)) = \exp(ib \log b ) \exp(-\pi b/2), 
\end{align*}
so that 
\begin{align*}
|(ib)^{a+ib-1/2}| = b^{a-1/2} \exp(-\pi b/2).
\end{align*}
This leads to the required result. 
\end{proof}
 
\subsection{Proof of Theorem~\ref{thm:btrace}} 
We start by assuming that $p$ is a trigonometric polynomial. At the last step of the proof, we will lift this assumption. 

\emph{\underline{Step 1}: writing $H(s)$ as a sum of multiplication operators.}
We write the definition \eqref{eq:bentr} of $H(s)$ as
\begin{align*}
[H(s)]_{n,m} = &\sum_{\ell=-\infty}^\infty 
\widetilde p_\ell\,
B\big(\tfrac12-i(s+ \omega n),\,\tfrac12+i(s+ \omega m)\big)\,\delta_{n-m-\ell}
\\
= &\sum_{\ell=-\infty}^\infty 
\widetilde p_\ell\, \,{b}^{(\ell)}_n(s)\, \delta_{n-m-\ell},
\end{align*} 
where $\delta_n$ is the Kronecker symbol and 
\begin{align*}
{b}^{(\ell)}_n(s) = &\ 
B\big(\tfrac12-i(s+ \omega n),\, \tfrac12+i( s+ \omega(n-\ell))\big)\\[0.2cm]
= &\ \frac{\Gamma(\tfrac12-i(s+ \omega n))
\Gamma(\tfrac12+i( s+ \omega(n-\ell)))}{\Gamma(1-i\omega \ell)}. 
\end{align*}
Consequently, the operator $H(s)$ can be represented as the finite sum
\begin{align}\label{eq:mb}
H(s) = \sum_{\ell}\,\widetilde p_\ell
\, 
\calM[{b}^{(\ell)}(s)]\,S^\ell,
\end{align} 
where $\calM[{b}^{(\ell)}(s)]$ is the operator of multiplication by the sequence ${b}^{(\ell)}(s)$ and $S$ is the (unitary) shift operator, see \eqref{eq:defJS}. 

For every fixed $\ell$, the sequence ${b}^{(\ell)}_n(s)$ goes to zero exponentially fast as $n\to\infty$, see \eqref{eq:gm}. It follows that the diagonal operator $\calM[{b}^{(\ell)}(s)]$ belongs to the trace class and its trace norm can be found as follows:
\begin{equation}
\norm{\calM[{b}^{(\ell)}(s)]}_{\mathbf S_1}
=
\sum_{n=-\infty}^\infty \abs{{b}^{(\ell)}_n(s)}\,  .
\label{eq:sp}
\end{equation}
Furthermore, each ${b}^{(\ell)}_n(s)$ is a holomorphic function of $s\in\bbC\setminus{\sf\Lambda}$ and the series \eqref{eq:sp} converges locally uniformly in $s\in\bbC\setminus{\sf\Lambda}$, and therefore $\calM[{b}^{(\ell)}(s)]$ is a holomorphic trace class valued function of $s\in\bbC\setminus{\sf\Lambda}$. 

It remains to prove the estimate \eqref{eq:nontrunc}. In order to do this, below we carefully estimate the trace norm in \eqref{eq:sp} and use the triangle inequality in \eqref{eq:mb}. 

In the bounds below the positive 
constant $C$ may vary from line to line.

\emph{\underline{Step 2:} estimating the trace norm of $\mathcal M({b}^{(\ell)}(s))$.}
We write $s = k+i\tau$, with real $k$ and $\tau$, $\abs{\tau}<\tau_0$.
By the easily verifiable identity
$$
H(k+i\tau)^* = H(k-i\tau)
$$
it suffices to consider the case $0\leq\tau<\tau_0$; we assume this throughout the proof.

In what follows we write ${b}^{(\ell)}_n$ in place of ${b}^{(\ell)}_n(s)$ for readability.
Observe that if $s\in K_{\tau_0, \delta}$, then for any $n\in\bbZ$, 
the complex number $\tfrac12\pm i(s + \omega n) =a+ib$ satisfies the hypothesis of Lemma~\ref{lem:gamma}: 
$$
\abs{a-\tfrac12}=\abs{\Im s}=\tau\leq \tau_0, 
\quad
\dist(a+ib,\bbZ)=\dist(s \mp \tfrac{i}2+\omega n, i\bbZ)\geq \dist(s,{\sf\Lambda})>\delta.
$$
Thus according to \eqref{eq:gm}, we have  
\begin{align*}
\abs{{b}^{(\ell)}_n}
= &\ \frac{\abs{\Gamma(\tfrac12+\tau - i(k + \omega n))
\Gamma(\tfrac12 - \tau  + i(k + \omega(n-\ell)))}}{|\Gamma(1-i\omega \ell)|}
\\
\le &\ C \jap{\ell}^{-1/2} e^{\pi\omega |\ell|/2}
\jap{\omega n + k }^{\tau}e^{-\pi \abs{\omega n+k}/2}
\jap{\omega (n-\ell) + k}^{-\tau}e^{-\pi\abs{\omega(n-\ell)+k}/2}. 
\end{align*}
Let us assume that $\ell \ge 0$ (the case $\ell\leq0$ can be considered in the same way). We will rewrite the above bound for ${b}^{(\ell)}_n$ by considering separately the three cases $\omega n+k<0$, $0\leq \omega n+k\leq \omega \ell$ and $\omega n+k>\omega \ell$. 

Using the estimate $\jap{a+b}\leq2\jap{a}\jap{b}$, in the first case we find
$$
\jap{\omega n+k}^{\tau}\jap{\omega n+k-\omega\ell}^{-\tau}
=
\jap{\omega n+k-\omega\ell+\omega\ell}^{\tau}\jap{\omega n+k-\omega\ell}^{-\tau}
\leq
C\jap{\omega\ell}^\tau\leq C\jap{\ell}^\tau,
$$
and therefore  
$$
\abs{{b}^{(\ell)}_n}
\leq  
C \jap{\ell}^{\tau-1/2} 
e^{-\pi\abs{\omega n+k}},\quad\text{ if }\quad \omega n + k <0.
$$
The sum in the right-hand side of \eqref{eq:sp} over this set of indices can be estimated by
\begin{equation}
\sum_{\omega n+k<0}\abs{{b}_n^{(\ell)}}
\leq
C \jap{\ell}^{\tau-1/2} 
\sum_{\omega n+k<0}e^{-\pi\abs{\omega n+k}}
\leq 
C \jap{\ell}^{\tau-1/2}.
\label{c1}
\end{equation}
Similarly, in the third case we find 
$$
\abs{{b}^{(\ell)}_n}
\leq 
C \jap{\ell}^{\tau-1/2}
e^{-\pi\abs{\omega n+k-\omega\ell}},\quad\text{ if }\quad
\omega n + k > \omega\ell,
$$
and therefore
\begin{equation}
\sum_{\omega n+k>\omega\ell}\abs{{b}_n^{(\ell)}}
\leq 
C \jap{\ell}^{\tau-1/2} 
\sum_{\omega n+k>\omega\ell}e^{-\pi\abs{\omega n+k-\omega\ell}}
\leq
C \jap{\ell}^{\tau-1/2}.
\label{c2}
\end{equation}
Now consider the second (intermediate) case $0\leq \omega n+k\leq \omega \ell$. 
Here,  using the monotonicity of $\jap{\cdot}$ we find
$$
\jap{\omega n+k}^{\tau}\jap{\omega n+k-\omega\ell}^{-\tau}
\leq
\jap{\omega\ell}^\tau\jap{\omega n+k-\omega\ell}^{-\tau}
\leq
C\jap{\ell}^\tau\jap{\omega n+k-\omega\ell}^{-\tau}
$$
and therefore
$$
\abs{{b}^{(\ell)}_n}
\leq
C\jap{\ell}^{\tau-1/2}\jap{\omega n+k-\omega\ell}^{-\tau},
\quad\text{ if }\quad
0\leq \omega n+k\leq \omega \ell.
$$
It follows that 
$$
\sum_{0\leq \omega n+k\leq \omega \ell}\abs{{b}^{(\ell)}_n}
\leq
C\jap{\ell}^{\tau-1/2} 
\Sigma,
\quad
\Sigma=\sum_{0\leq \omega n+k\leq \omega \ell}
\jap{\omega n+k-\omega\ell}^{-\tau}. 
$$
The sum $\Sigma$ in the right hand side can be estimated as follows:
$$
\Sigma\, 
\le 
\jap{\omega\ell}^{-\tau}+
\int_0^{\ell}\, \langle \omega\ell -\omega x\rangle^{-\tau}\, dx 
\leq
\begin{cases}
C\jap{\ell}^{1-\tau}, & 0\leq\tau<1,
\\
C\log(2+\ell^2), & \tau=1,
\\
C, & \tau>1.
\end{cases}
$$
Recalling our definition 
\eqref{eq:weight} 
of the weight function $\rho_\tau$, we can write this as
$$
\jap{\ell}^{\tau-1/2}\Sigma\leq C \rho_\tau(\ell), 
$$
and therefore 
$$
\sum_{0\leq \omega n+k\leq \omega \ell}\abs{{b}^{(\ell)}_n}
\leq
C\rho_\tau(\ell).
$$
Coming back to the sums \eqref{c1} and \eqref{c2}, we see that they can be estimated by the same quantity, 
$$
\sum_{\omega n+k<0}\abs{{b}_n^{(\ell)}}
\leq 
C \jap{\ell}^{\tau-1/2}
\leq
C \rho_\tau(\ell), 
\quad
\sum_{\omega n+k>\omega\ell}\abs{{b}_n^{(\ell)}}
\leq
C \jap{\ell}^{\tau-1/2}
\leq
C \rho_\tau(\ell).
$$
Putting these bounds together, we obtain 
$$
\norm{\mathcal M({b}^{(\ell)}(s))}_{\mathbf S_1} 
=
\sum_{n=-\infty}^\infty \abs{{b}^{(\ell)}_n}
\leq
C \rho_\tau(\ell), \quad \ell \ge 0.
$$
The case $\ell<0$ is considered similarly, which finally yields
$$
\norm{\mathcal M({b}^{(\ell)}(s))}_{\mathbf S_1} 
\leq
C \rho_\tau(\ell),\quad \textup{for all} \quad \ell\in\bbZ.
$$

\emph{\underline{Step 3:} the triangle inequality.}
Using the triangle inequality to estimate the trace norm in \eqref{eq:mb}, we obtain
\begin{align*}
\norm{H(s)}_{\mathbf S_1}
\leq \sum_{\ell} \abs{\widetilde p_\ell} 
\, \|\mathcal M({b}^{(\ell)})\|_{\mathbf S_1}
\leq C \sum_{\ell}\abs{\widetilde p_\ell}\, \rho_\tau(\ell)
\leq C\sum_{\ell}\abs{\widetilde p_\ell}\, \rho_{\tau_0}(\ell),
\end{align*}
where the constant $C$ in the right hand side depends only on $\tau_0$ and $\delta>0$. 
From here we obtain \eqref{eq:nontrunc}. 

\emph{\underline{Step 4:} extension to general $p$.}
Finally, we extend the statement from the case of trigonometric polynomials $p$ to general $p$. 
Let $p$ be as in the hypothesis of the theorem, and let $p^{(N)}$, $N=1, 2, \dots$ be the approximating sequence of polynomials defined by 
\begin{equation}
p^{(N)}(\xi)=\sum_{\ell=-N}^N \wt p_\ell e^{i\ell \omega\xi}, \quad \xi\in\bbR. 
\label{b7}
\end{equation}
Let $H^{(N)}(s)$, $s\notin{\sf \Lambda}$, be the associated operator with entries \eqref{eq:bentr} if $|n-m|\le N$ and zero otherwise. By the already proven estimate \eqref{eq:nontrunc}, we see that the operator $H^{(N)}(s)$ is holomorphic in $s\in\bbC\setminus{\sf\Lambda}$, converges in trace norm as $N\to\infty$, and the convergence is uniform in every punctured strip $K_{\tau_0,\delta}$ with $\delta>0$. It is clear that the limit coincides with $H(s)$. Thus, $H(s)$ is also trace class valued, holomorphic in $K_{\tau_0,0}$ and the estimate \eqref{eq:nontrunc} holds true. 
The proof of Theorem~\ref{thm:btrace} is complete. 
\qed

\begin{remark}\label{rem:bpoly}
Since $\bbR\subset K_{\tau_0, \delta}$  for all $\tau_0 >0 $ and $\delta <1/2$, Theorem~\ref{thm:btrace} immediately implies Lemma~\ref{lem:bpoly}. 
\end{remark}

\section{The secular determinant $\Delta(k;\lambda)$}\label{sec.cc}

\subsection{Main statement}
Let $\mathbf H$ be a smooth periodic Hankel operator. 
Our main tool in the study of the band functions is the secular determinant 
\[
\Delta(k;\lambda):=\det(I-\lambda^{-1} H(k)),
\]
where $k\in\Omega$ and $\lambda\in\bbR$, $\lambda\not=0$.
By  Lemma~\ref{lem:bpoly}, $H(k)$ is trace class and so the determinant is well defined. 
It is clear that the zeros of the secular determinant determine all band functions.  

We will extend the secular determinant to the complex domain. By Theorem~\ref{thm:btrace}, the smoothness condition \eqref{eq:pdec} ensures that the operator $H(s)$, as defined by \eqref{eq:bentr}, is a trace class valued holomorphic function in the punctured strip $K_{1,0}$ (see \eqref{eq:strip}) of width 2, and so we can define
$$
\Delta(s;\lambda):=\det(I-\lambda^{-1}H(s)), \quad  s\in\bbC\setminus{\sf\Lambda},\quad \abs{\Im s}<1.
$$
Extending the secular determinant to this strip is not sufficient to our purposes; we wish to extend it to a meromorphic function in the whole complex plane as follows. 
Let $H^{(N)}(s)$ be the operator corresponding to the  polynomial $p^{(N)}$, see \eqref{b7}. Obviously, the polynomial $p^{(N)}$ satisfies the hypothesis of Theorem~\ref{thm:btrace} with any $\tau_0>0$, and therefore $H^{(N)}(s)$ extends to all $s\in\bbC\setminus{\sf\Lambda}$ as a trace class valued holomorphic function of $s$. Thus the determinant 
\[
\Delta^{(N)}(s; \lambda):= 
\det \big(I - \lambda^{-1} H^{(N)}(s)\big)
\]
is well defined for all $s\in\bbC\setminus{\sf\Lambda}$ as a holomorphic function of $s$. 
For the ``full" operator $H(s)$ we define the secular determinant 
as the limit 
\begin{align}\label{eq:deltagen}
\Delta(s; \lambda) := \lim_{N\to\infty} \Delta^{(N)}(s; \lambda),
\end{align}
if it exists. The main result of this section is:
\begin{theorem}
\label{thm:elliptic}
Let $\mathbf H$ be a smooth periodic Hankel operator and let $\lambda\in\bbC$, $\lambda\not=0$. 
Then:
\begin{enumerate}[\rm (i)]
\item\label{item:one}
The limit \eqref{eq:deltagen} exists for all $s\in\bbC\setminus{\sf \Lambda}$ uniformly in $s$ with $\dist(s,{\sf\Lambda})>\delta$. Furthermore, if $|\Im s|<1$, then the limit 
\eqref{eq:deltagen} coincides with $\det(I-\lambda^{-1}H(s))$.
\item 
The function $s\mapsto\Delta(s; \lambda)$ is even in $s$, doubly-periodic with the periods $2i$ and $\omega$ and satisfies the identities 
\begin{align}
\overline{\Delta(s;\lambda)}&=\Delta(\overline{s}; \overline{\lambda}),
\label{ee1}
\\
\Delta(s+i;\lambda)&= \Delta(s;-\lambda).
\label{ee21}
\end{align}
Furthermore, $\Delta(s; \lambda)$ is either constant for all $s\in\bbC$ or 
it is an elliptic function of order two with simple poles at the points of the lattice $\sf \Lambda$, and 
\begin{align}\label{eq:resd}
\Res\limits_{s=i/2}\Delta(s;\lambda)
= -\Res\limits_{s=-i/2}\Delta(s;\lambda).
\end{align}
\item
If $\Delta$ is non-constant, then 
the zeros of the derivative $\partial\Delta(s;\lambda)/\partial s$ are simple 
and the set of zeros is given by the points congruent to 
\[
\bigg\{0,\, \frac{\omega}{2},\, \frac{\omega}{2} + i,\, i\bigg\}
\]
with respect to the period lattice $\sf M$ defined in \eqref{eq:perm}.
\end{enumerate}
\end{theorem} 
In the rest of this section we prove this theorem. 
We note that \eqref{eq:resd} is a simple consequence of the fact that $\Delta(s;\lambda)$ is even in $s$.

\subsection{The secular determinant for trigonometric polynomials}
We start by establishing the properties of the secular determinant for trigonometric polynomials $p$. 
In this case the definition \eqref{eq:deltagen} rewrites directly as 
$$
\Delta(s;\lambda):=\det(I-\lambda^{-1}H(s)), \quad s\in\bbC\setminus{\sf\Lambda}.
$$
\begin{theorem}\label{thm.ee1}
Let $p$ be a trigonometric polynomial and let $\lambda\in\bbC$, $\lambda\not=0$. Then the conclusions of Theorem~\ref{thm:elliptic}(ii) hold true. 
\end{theorem}

Before embarking on the proof, we discuss a representation for the determinant that will be useful for us. 
First we represent the operator $H(s)$ as in 
\eqref{eq:bs}: 
\begin{align*}
H(s) = \calM[\gamma^{(-)}] \,\mathfrak S\, \calM[\gamma^{(+)}], 
\end{align*}
where $\mathfrak S$ is as in \eqref{eq:umn} and 
$\calM[\gamma^{(\pm)}]$ is the operator of multiplication by the sequence 
\begin{align*}
\gamma^{(\pm)}_n = \gamma^{(\pm)}_n(s) = \Gamma\big(\tfrac12 \pm i(\omega n + s)\big).
\end{align*}
Now we use the cyclicity property of the determinant: if $A$ and $B$ are bounded operators such that 
$AB\in \mathbf S_1$ and $BA\in \mathbf S_1$, then 
$\det(I+AB)=\det(I+BA)$, see e.g. \cite[Chapter IV]{Gohberg1969}.
Since $p$ is a polynomial, the operator $\mathfrak S$ is bounded, and the operators $\calM[\gamma^{(+)}]$ and $\calM[\gamma^{(-)}]$ are clearly trace class. It follows that 
$$
\det\bigl(I-\lambda^{-1}\calM[\gamma^{(-)}]\,\mathfrak S\, \calM[\gamma^{(+)}]\bigr)
=
\det\bigl(I-\lambda^{-1} \mathfrak S\, \calM[\gamma^{(+)}]\calM[\gamma^{(-)}]\bigr).
$$
Denoting 
\begin{align}\label{eq:wn}
w_n =& w_n(s) = \gamma^{(+)}_n(s) \gamma^{(-)}_n(s) = 
\Gamma(\tfrac12-i \omega n - is)\Gamma(\tfrac12+i\omega n+ is)
=\frac{\pi}{\cosh\pi(s+\omega n)}
\end{align}
for $n\in\bbZ$ and writing $\calM[w] =  \calM[\gamma^{(+)}]\,\calM[\gamma^{(-)}]$, we finally obtain 
\begin{align}\label{eq:delta1}
\Delta(s;\lambda)=\det(I-\lambda^{-1}\mathfrak S\calM[w(s)]), \quad s\in\bbC\setminus {\sf{\Lambda}}.
\end{align}

\begin{proof}[Proof of Theorem~\ref{thm.ee1}]
\emph{Step 1:} first we establish some algebraic relations involving the operators ${\mathfrak S}$ and $\calM[w(s)]$. 
Let $J$ and $S$ be the reflection and shift operators in $\ell^2(\bbZ)$, see \eqref{eq:defJS}.
We note that ${\mathfrak S}$ can be written as a finite sum
$$
\sum_{n=-\infty}^\infty\widetilde{\mathfrak s}_{n} S^n,
$$
where $S^{-n}=(S^*)^n$. 
It is straightforward to check that $JSJ=S^T$, where $S^T$ denotes the transposed matrix of $S$. 
It follows that 
$$
J{\mathfrak S}J={\mathfrak S}^T.
$$
Moreover, since ${\mathfrak S}$ is Hermitian, we have 
$$
{\mathfrak S}^T=\overline{{\mathfrak S}}.
$$
Furthermore, for any sequence $u$ we have
$$
J\calM[u]J=\calM[Ju].  
$$
Observe that $w_n(s)$ satisfies
$$
w_{-n}(-s)=w_n(s),\quad \overline{w_n(s)}=w_{n}(\overline{s}),
$$
and so $Jw(s)=w(-s)$, $\overline{w(s)} = w(\overline{s})$.

\emph{Step 2:} we prove the symmetry relations involving the secular determinant. 

\underline{Proof that $\Delta(s,\lambda)$ is even:} 
\begin{align*}
\Delta(-s;\lambda)&=\det\bigl(I-\lambda^{-1}{\mathfrak S}\calM[w(-s)]\bigr)
=\det\bigl(I-\lambda^{-1}{\mathfrak S}\calM[Jw(s)]\bigr)
\\
&=\det\bigl(I-\lambda^{-1}{\mathfrak S}J\calM[w(s)]J\bigr)
=\det\bigl(I-\lambda^{-1}J{\mathfrak S}J\calM[w(s)]\bigr)
\\
&=\det\bigl(I-\lambda^{-1}{\mathfrak S}^T\calM[w(s)]\bigr)
=\det\bigl(I-\lambda^{-1}\calM[w(s)]{\mathfrak S}^T\bigr)
\\
&=\det\bigl(I-\lambda^{-1}({\mathfrak S}\calM[w(s)])^T\bigr)
=\det\bigl(I-\lambda^{-1}({\mathfrak S}\calM[w(s)])\bigr)
\\
&=\Delta(s;\lambda),
\end{align*}
where we have used that $\det(I+A^T)=\det(I+A)$ at the last step. 

\underline{Proof of \eqref{ee1}:}
\begin{align*}
\overline{\Delta(s;\lambda)}&=\det\bigl(I-\overline{\lambda}^{-1}\overline{{\mathfrak S}}\calM[\overline{w(s)}]\bigr)
=\det\bigl(I-\overline{\lambda}^{-1}{\mathfrak S}^T\calM[w(\overline{s})]\bigr)
\\
&=\det\bigl(I-\overline{\lambda}^{-1}\calM[w(\overline{s})]{\mathfrak S}^T\bigr)
=\det\bigl(I-\overline{\lambda}^{-1}({\mathfrak S}\calM[w(\overline{s})])^T\bigr)
\\
&=\det\bigl(I-\overline{\lambda}^{-1}{\mathfrak S}\calM[w(\overline{s})]\bigr)
=\Delta(\overline{s};\overline{\lambda}).
\end{align*}

\underline{Proof of \eqref{ee21}:} 
by \eqref{eq:wn} we have
$$
w_n(s+i)=-w_n(s),
$$
and therefore
\begin{align*}
\Delta(s+i;\lambda)& = \det\bigl(I-\lambda^{-1}{\mathfrak S}\calM[w(s+i)]\bigr)
\\
&=\det\bigl(I+\lambda^{-1}{\mathfrak S}\calM[w(s)]\bigr)
=\Delta(s;-\lambda).
\end{align*}

\underline{Proof that $\Delta$ is periodic with periods $\omega$ and $2i$:} 
Applying \eqref{ee21} twice, we get
$$
\Delta(s+2i;\lambda)=\Delta(s+i;-\lambda)=\Delta(s;\lambda)
$$
and so $2i$ is a period. Let us prove that $\omega$ is a period. 
According to \eqref{eq:wn}, 
\begin{align*}
w_n(s+\omega) =  w_{n+1}(s).
\end{align*}
This implies that $\calM[w(s+\omega)] = S^* \calM[w(s)]S$ and so, using cyclicity,
$$
\Delta(s+\omega; \lambda)
=\det\bigl(I-\lambda^{-1}{\mathfrak S}S^* \calM[w(s)]S\bigr)
=\det\bigl(I-\lambda^{-1}S{\mathfrak S}S^* \calM[w(s)]\bigr).
$$
Since all powers of $S$ and $S^*$ commute, we get the invariance 
$$
S{\mathfrak S}S^* = {\mathfrak S}
$$
and so, finally, $\Delta(s+\omega; \lambda)=\Delta(s; \lambda)$. 

\emph{Step 3:} let us show that $\Delta(s; \lambda)$ has \emph{simple} poles at the points $-i/2$ and $i/2$. 
In fact, since $\Delta$ is even in $s$, it suffices to show this only for the point $s_0 = i/2$. 
Consider a small neighbourhood of this point. Out of all numbers $w_n(s)$ only $w_0(s)$ has a singularity 
at $s_0$; this singularity is a simple pole: 
\begin{align}\label{eq:pole}
w_0(s) = -\frac{i}{s-i/2} + O(1), \ s\to i/2.
\end{align}
Observe that $w_0(s)$ enters only one column of the matrix ${\mathfrak S}\calM[w(s)]$, 
and therefore, using the cofactor expansion along the column labeled ``0", 
 we can represent the determinant \eqref{eq:delta1} in the form 
\begin{align*}
\Delta(s;\lambda) = F_0(s) + F_1(s) w_0(s),
\end{align*}
where $F_0$ and $F_1$ are functions analytic in a neighbourhood of $s_0$. If $F_1(s_0) = 0$, then 
the function $\Delta(s;\lambda)$ is bounded and by Liouville's theorem it is therefore constant. If however, 
$F_1(s_0)\not = 0$, then due to \eqref{eq:pole}, the function $\Delta(s;\lambda)$ has a simple pole with residue 
$- i F_1(s_0)$. 
Identity \eqref{eq:resd} follows from the fact that $\Delta(s;\lambda)$ is even in $s$. 
\end{proof}

\subsection{Proof of Theorem~\ref{thm:elliptic}} 
(i)
Any $s\in\bbC\setminus{\sf \Lambda}$ can be represented as $s = i m  + s_0$, where $m\in\bbZ$ 
and $|\Im s_0|<1$. By \eqref{ee21},
\begin{align*}
\Delta^{(N)}(s; \lambda) = \Delta^{(N)}(s_0; (-1)^m\lambda).
\end{align*}
Since $|\Im s_0|<1$, by virtue of Theorem~\ref{thm:btrace}, under the smoothness condition \eqref{eq:pdec} the 
truncated operator $H^{(N)}(s_0)$ converges to $H(s_0)$ in trace norm, 
so that the limit \eqref{eq:deltagen} exists and coincides with 
$\det(I-(-1)^m\lambda^{-1}H(s_0))$. 
The uniformity of the convergence follows from the uniformity of the constant in the estimate \eqref{eq:nontrunc} of Theorem~\ref{thm:btrace}.
Finally, if $|\Im s|<1$, then $s_0=s$, and so the limit coincides with $\det(I-\lambda^{-1}H(s_0))$.

(ii)
The facts that $\Delta(\cdot;\lambda)$ is even, doubly-periodic and satisfies \eqref{ee1} 
and \eqref{ee21} follow  from Theorem~\ref{thm.ee1} by passing to the limit. 
Furthermore, the only possible singular points are congruent to $i/2$ or $-i/2$. 
Consider for example the point $i/2$ and let us prove that it is a simple pole of $\Delta(\cdot;\lambda)$. Since $i/2$ is a simple pole of $\Delta^{(N)}(\cdot;\lambda)$, for a small circular contour $\mathcal C$ around $i/2$ we find
\begin{align*}
\int_{\mathcal C} (s-i/2)^n \Delta(s; \lambda)\, ds 
= \lim_{N\to\infty}\int_{\mathcal C} (s-i/2)^n \Delta^{(N)}(s; \lambda)\, ds = 0, 
\end{align*}
for all $n\ge 1$. 
Therefore, the singularity of $\Delta$ at $i/2$ is a simple pole. The same 
applies to the point $-i/2$.
The relation \eqref{eq:resd}, as in the proof of Theorem~\ref{thm.ee1}, follows since $\Delta(s; \lambda)$ is even in $s$. 
 
(iii)
Denote for brevity $\Delta_s':=\partial \Delta/\partial s$. 
Since $\Delta$ is even, we find $\Delta_s'(0; \lambda)=0$. 
Similarly, from 
$$
\Delta(s; \lambda)=\Delta(-s;\lambda)=\Delta(\omega-s; \lambda)
$$
we find $\Delta_s'(\tfrac{\omega}2; \lambda) = 0$.
Because of \eqref{ee21} we also have $\Delta_s'(i; \lambda) = \Delta_s'(\tfrac{\omega}{2}+i; \lambda) = 0$. 
We have established that $\Delta_s'$ has zeros $0,\tfrac{\omega}{2},\tfrac{\omega}{2}+i,i$. 
For a non-constant $\Delta$, the function  $\Delta_s'(s; \lambda)$ is an elliptic function 
of order four, and hence it has exactly four zeros in the fundamental domain; we have found all of them. 

The proof of Theorem~\ref{thm:elliptic} is complete. \qed

\section{Analytic band functions I: Proof of Theorem~\ref{thm.branches}}
\label{sec.dd}

The aim of this section is to prove Theorem~\ref{thm.branches}. 

\subsection{Splitting away a unimodular factor}

By \eqref{eq:bk}, our fibre operator $H(k)$, $k\in{\bbR},$ has the form 
\begin{equation}
H(k) = V(k)^* H(0) V(k), 
\label{dd2}
\end{equation}
where $V(k)$ is the operator of multiplication by the sequence 
\begin{align*}
v_n(k) = \frac{\Gamma(\frac{1}{2} + i\omega n + ik)}{\Gamma(\frac{1}{2} + i\omega n)}\ .
\end{align*}
We give some heuristics for what comes next. 
The key ingredient of the proof of Theorem~\ref{thm.branches} is the differentiation of eigenvalues with respect to $k$ and using Gronwall's lemma. This requires differentiating $V(k)$ in \eqref{dd2} with respect to $k$. It turns out that the derivative of $v_n(k)$ is ``too large'' for this argument to work. The solution is to write $v_n(k)=\abs{v_n(k)}e^{i\varphi_n(k)}$; it turns out that the derivative of $\abs{v_n(k)}$ is ``not too large'', while the unimodular factor $e^{i\varphi_n(k)}$ does not affect the eigenvalues. 
Motivated by this and 
the formula \eqref{eq:gamco}, we denote
\begin{equation}
g_n(k) =\abs{v_n(k)}= \sqrt{\frac{\cosh(\pi \omega n)}{\cosh(\pi (\omega n + k))}}, 
\label{eq:v0}
\end{equation}
and write    
\[
v_n(k) = g_n(k) e^{i\varphi_n(k)}, 
\]
where $e^{i\varphi_n(k)}$ is a real analytic unimodular function.
Therefore for each $k\in\bbR$ we rewrite \eqref{dd2} as 
$$
H(k)=\Phi(k)^*G(k)H(0)G(k)\Phi(k), 
$$
where $\Phi(k)$ is the unitary operator of multiplication by $\{e^{i\varphi_n(k)}\}_{n\in\bbZ}$ and $G(k)$ is the operator of multiplication by $\{ g_n(k)\}_{n\in\bbZ}$ in $\ell^2(\bbZ)$. 
We denote 
$$
H_*(k):=G(k)H(0)G(k);
$$
since $H(k)$ and $H_*(k)$ are unitarily equivalent, it suffices to prove Theorem~\ref{thm.branches} for the eigenvalues of $H_*(k)$ instead of $H(k)$.

It is evident that the operator $G(k)$ is differentiable with respect to $k$.
For the derivative $G'(k)$, we have the following bound. 
\begin{lemma}\label{lma.gronwall}
For all $k\in\bbR$ and all $f\in \ell^2(\bbZ)$, we have 
$$
\Norm{G'(k)f}\leq \frac{\pi}{2}\norm{G(k)f}.
$$
\end{lemma}
\begin{proof}
We calculate the derivative
\begin{align*}
g'_n(k) =  - \frac{\pi}{2}\tanh (\pi(\omega n + k))  \, g_n(k).
\end{align*}
Consequently, the bound holds
\begin{align*}
\abs{g'_n(k)}\leq \frac{\pi}{2} g_n(k), \quad k\in{\bbR}.
\end{align*}
Since $G(k)$ is a diagonal operator, this gives the required bound. 
\end{proof}

\subsection{Proof of Theorem~\ref{thm.branches}}
Since $H(k)$ and $H_*(k)$ are unitarily equivalent, it suffices to prove the required statement for $H_*(k)$ in place of $H(k)$. First we need to check that $H_*(k)$ extends to a holomorphic function in a strip around the real line. We define this extension by defining $g_n(s)$ by the same formula \eqref{eq:v0} with $k$ replaced with $s = k+ i\tau$. 
Since 
\begin{align*}
\Re \cosh(\pi (\omega n + k+i\tau))
= \cosh(\pi (\omega n + k)) \, \cos (\pi\tau)>0,\quad |\tau| < \frac{1}{2},
\end{align*}
the functions $g_n(s)$ are well defined and holomorphic in the strip $\{s: \abs{\Im s}<1/2\}$. 
Because of the bound 
$$
\abs{g_n(s)}^2
\leq
\frac{\cosh(\pi\omega n)}{\abs{\cosh(\pi \omega n+\pi k+i\pi \tau)}}
\leq
\frac{\cosh(\pi\omega n)}{\cosh(\pi \omega n+\pi k)\cos(\pi \tau)}
\leq
\frac{e^{\pi\abs{k}}}{\cos(\pi \tau)}
$$
that holds for all $n\in\bbZ$, the operator $G(s)$ is uniformly bounded on compact 
subsets of this strip. Also, one checks directly that $G(s)^* = G(\overline s)$, 
so that  $(H_*(s))^* = H_*(\overline{s})$, since $H(0)$ is self-adjoint. Thus $H_*(s)$ is a self-adjoint analytic family in the strip $|\Im s\,|<1/2$ in the sense of \cite[Ch. VII, \S 3]{Kato}.

The rest of the proof follows closely F.~Rellich's argument in  
\cite[Satz 1]{Rellich1942}, see also \cite[Ch. VII, \S 3.5]{Kato}. 
Our starting point is the following well-known local result. 
Fix a $k_0\in {\bbR}$ and let 
$E_*\not=0$ be a (real) eigenvalue of $H_*(k_0)$ of multiplicity 
$n\geq1$. According to \cite[Theorem XII.13]{RS4}, in a sufficiently small 
neighbourhood of $k_0$ the operator $H_*(s)$ has exactly $n$ 
eigenvalues, and there exist $n$ not necessarily distinct 
single-valued functions $E_1(s),\dots,  E_n(s)$, 
analytic near $s=k_0$, with $E_n(k_0)=E_*$, such that 
$E_1(s),\dots, E_n(s)$ are the eigenvalues of $H_*(s)$ near $k_0$.

Our objective is to show that for real $k$, these analytic eigenvalue branches that  are initially defined only locally, can be continued as \underline{non-vanishing} real analytic functions to all $k\in{\bbR}$. Since $H_*(s)$ is analytic in the strip around the whole real line, each eigenvalue branch can be continued indefinitely \emph{as long as it does not cross zero}; crossing zero is the only possible obstruction to the analytic continuation. Let us show that the eigenvalue branches cannot cross zero. We use an argument based on Lemma~\ref{lma.gronwall} and on Gronwall's inequality. 

Let $ E= E(k)$ be an eigenvalue branch, defined on a \emph{maximal} interval $(a,b)$ 
and positive on this interval (the proof for a negative branch is exactly the same). 
By this we mean that $(a, b)$ is an interval such that $E(k)$ is positive for all 
$k\in (a, b)$ and it cannot be continued beyond any of the 
endpoints as a positive real analytic function.
We need to show that $(a, b) = \bbR$.  
To get a contradiction, suppose, for example, that $a$ is finite.  

Let $\phi=\phi(k)$ be the normalized eigenfunction corresponding to $E(k)$. 
According to \cite[Ch. II, \S4 and Ch. VII, \S 3]{Kato} or \cite[Section~XII.2]{RS4}, 
the eigenfunction $\phi$ can be chosen to be real analytic in $k$ as long as $E(k)$ 
remains  separated from zero. Thus for all $k\in (a, b)$ we can use the standard formula 
$$
E'(k)=\jap{H'_*(k)\phi,\phi},\quad \phi=\phi(k),
$$
for the derivative of an eigenvalue; here $\jap{\cdot,\cdot}$ is the inner product in $\ell^2$. 
(This formula can be derived in the same way as in \cite[Ch. II -(6.10)]{Kato}.)
We have
\begin{align*}
\jap{H'_*(k)\phi,\phi}&=\jap{G'(k)H_*(0)G(k)\phi,\phi}+\jap{G(k)H_*(0)G'(k)\phi,\phi}
\\
&=2\Re\jap{G'(k)H_*(0)G(k)\phi,\phi}
\end{align*}
and so by Lemma~\ref{lma.gronwall}
$$
\abs{E'(k)}
\leq 2\norm{G'(k)H_*(0) G(k)\phi}
\leq \pi \norm{ G(k)H_*(0) G(k)\phi}
=\pi E(k). 
$$
By Gronwall's inequality, for any $k,k_*\in(a,b)$ this implies that
$$
E(k_*)e^{-\pi\abs{k-k_*}}\leq E(k)\leq E(k_*)e^{\pi\abs{k-k_*}}.
$$
Let us fix some $k_*\in(a,b)$ here; we see that $E(k)$ is separated away from zero on $(a,k_*)$:
\[
E(k)\geq e^{-\pi(k_*-a)}E(k_*), \quad k\in(a,k_*).
\]
Now consider the interval $(\frac12e^{-\pi(k_*-a)}E(k_*), \infty)$. 
The operator $H_*(a)$ has only finitely many eigenvalues in this interval, 
and, as discussed above, near $k=a$ 
each of them can be represented by an analytic function. 
It follows that $E(k)$ coincides with one of these functions and so it can be continued analytically beyond $a$ as a positive function. 
This contradicts the maximality of $(a,b)$. Thus $(a, b) = \bbR$ and the proof is complete. 

To prove the last statement of the theorem, observe that each 
$E\in\mathcal E$ is a real 
analytic solution of the equation $\Delta(k; E(k)) = 0$, and, 
conversely, each real analytic solution of this equation is automatically 
a band function. 
By the symmetry $\Delta(k; \lambda) = \Delta(-k; \lambda)$ established in 
Theorem~\ref{thm:elliptic}, we have 
\begin{align*}
\Delta(k; E(-k))  = \Delta(-k; E(-k)) = 0 \quad \textup{for all} \quad k\in\bbR.
\end{align*}
Since $E(-k)$ is analytic, it is a band function. 
In the same way, using the periodicity $\Delta(k+\omega; \lambda) = \Delta(k; \lambda)$, one proves the 
equality $\Delta(k; E(k+\omega n)) = 0$ for all $n\in\bbZ$.
The proof of Theorem~\ref{thm.branches} is complete. 
\qed

\section{Analytic band functions II: proof of Theorem~\ref{thm.main}}
\label{sec.ddd}
\subsection{Overview}
In Section~\ref{sec.ddd2}, we introduce a set-up that allows us to consider the flat and non-flat band functions separately. In Section~\ref{subsect:abran}, we establish various properties of the non-flat band functions.  In Sections~\ref{sec.ddd4} and \ref{subsect:proofmain}, we prove Theorem~\ref{thm.main}.
Finally, in Section~\ref{sec.b8}, we illustrate the concept of band functions by considering the Carleman operator.

\subsection{Separating out the flat bands}
\label{sec.ddd2}
By Theorem~\ref{thm.branches} we have
\begin{equation}
\Delta(k;\lambda) = 
\prod_{E\in\mathcal E}(1-\lambda^{-1}  E(k)), \quad k\in\bbR.
\label{ddd1}
\end{equation}
We distinguish two types of band functions $E\in\mathcal E$: constant and non-constant (an example of spectrum containing both types will be discussed in Section~\ref{sec.e}). 
We will call $E$ a \emph{flat band function of $H(k)$} 
of multiplicity $m$, if it is constant and if it appears exactly 
$m$ times in the list $\mathcal E$. Accordingly, we will call the 
non-constant band functions \emph{non-flat}. 
Their multiplicity is defined in the same way. 
We will often use subscript $\flat$ for flat band functions, and subscript $\#$ for non-flat ones; 
accordingly, we write $\mathcal E$ as a disjoint union $\mathcal E=\mathcal E^\flat\cup\mathcal E^\#$, where $\mathcal E^\flat$ (resp. $\mathcal E^\#$) is the list of the flat (resp. non-flat) band functions.

The flat band functions possess the following symmetry.

\begin{lemma}\label{lma.d1}
If $E $ is a flat band function of $H(k)$, 
then so is $-E$, with the same multiplicity.
\end{lemma}

\begin{proof}
By Theorem~\ref{thm.branches}, the constant 
$E$ is a flat band function of multiplicity $m\ge 1$ 
if and only if  
$$
\Delta(s; \lambda) = (1- E/\lambda)^m\mathcal D(s;\lambda),  
$$
where $\mathcal D(s; \lambda)$ is an entire function of $1/\lambda$ such that 
$\mathcal D(\,\cdot\,; E)$ does not vanish identically.
It follows from \eqref{ee21} that 
\begin{align*}
\Delta(s; \lambda) = \Delta(s+i; -\lambda) = 
 (1 + E/\lambda)^m\mathcal D(s+i;-\lambda).
\end{align*}
Since $\mathcal D(s+i; E)$ 
does not vanish identically, this means that $\lambda = - E$ is also a 
flat band function of $H(k)$ of multiplicity $m$, as claimed. 
\end{proof}

Let us factor out from the secular determinant 
\emph{all} flat band functions counting multiplicities, and represent $\Delta(s; \lambda)$  as the product 
\begin{align*}
\Delta(s;\lambda)&= \Delta^\#(s;\lambda)\Delta^\flat(\lambda),\notag
\\
\Delta^{\flat}(\lambda) &= \prod_{E\in\mathcal E^\flat}(1- E/\lambda)
= \prod_{E\in\mathcal E^\flat,\  E>0}(1-(E/\lambda)^2),
\end{align*}
where \emph{the residual determinant} 
$\Delta^\#(s; \lambda), s\in \bbC\setminus{\sf\Lambda}$,  
is an entire function of $1/\lambda$, and for each 
$\lambda\in\bbR$, $\lambda\not=0$, the function 
$s\mapsto  \Delta^\#(s;\lambda)$ does not vanish identically. 
If all band functions are flat, we have $\Delta^\#(s;\lambda)=1$ and conversely if there are no flat band functions, we have $\Delta^{\flat}(\lambda)=1$.

The following lemma (which mirrors Theorem~\ref{thm:elliptic}) is obvious but we state it explicitly for clarity.
 
\begin{lemma}\label{lem:resid}
Let $\lambda\in\bbC$, $\lambda\not=0$. Then:
\begin{enumerate}[\rm (i)]
\item 
The function $s\mapsto\Delta^\#(s; \lambda)$ is analytic on $\bbC\setminus{\sf\Lambda}$, 
even in $s$, doubly-periodic with the periods $2i$ and $\omega$ and satisfies the identities 
\begin{align}
\overline{\Delta^\#(s;\lambda)}& = \Delta^\#(\overline{s}; \overline{\lambda}),
\notag
\\
\Delta^\#(s+i;\lambda)& = \Delta^\#(s;-\lambda).
\label{ee21a}
\end{align}
Furthermore, $\Delta^\#(s; \lambda)$ is either a \underline{non-zero} constant for all $s\in\bbC$ or 
it is an elliptic function of order two with simple poles at the points of the lattice $\sf \Lambda$. 
\item\label{item:resid2}
If $\Delta^\#(\cdot;\lambda)$ is non-constant, then 
the zeros of $\partial \Delta^\#(s;\lambda)/\partial s$ are simple 
and the set of zeros is given by all points congruent to 
$$
\bigg\{0,\, \frac{\omega}{2},\, \frac{\omega}{2} + i,\, i\bigg\}
$$
with respect to the period lattice ${\sf M}$. 
\end{enumerate}
\end{lemma}
\begin{proof}
For $\lambda\in\bbR$, the determinant $\Delta^\#(\cdot;\lambda)$ cannot be identically zero by construction. 
Furthermore, for $\lambda\in\bbC\setminus\bbR$, the determinant $\Delta^\#(\cdot;\lambda)$ cannot be identically zero because in this case this would imply that the self-adjoint operator $H(k)$ has a complex eigenvalue $\lambda$.
The rest follows from Theorem~\ref{thm:elliptic}, taking into account that 
$\Delta^\flat(\lambda) = \Delta^\flat(-\lambda)$ and $\overline{\Delta^\flat(\lambda)} 
= \Delta^\flat(\overline\lambda)$.
\end{proof}

\subsection{The non-flat band functions}\label{subsect:abran} 
In this section, we work exclusively with non-flat (NF for short) band functions.
By \eqref{ddd1}, the set of all NF band functions coincides with the set of all 
real analytic solutions of the equation $\Delta^\#(k; E(k)) = 0$. 
By Theorem~\ref{thm.branches},  
along with $E(k)$ the functions $E(-k)$ and $E(k+n\omega)$ 
are also NF for all $n\in\mathbb Z$. One or both of these may happen to 
coincide with $E(k)$. 

In what follows we extensively use the following fact. 
If  $f$ is an elliptic function of order two, then each value $\mu\in\bbC$ is taken exactly twice in the fundamental domain of $f$, i.e. the equation $f(s) = \mu$ has either two distinct zeros or one double zero in the fundamental domain.

\begin{theorem} \label{thm:anbr}
\begin{enumerate}[{\rm (i)}]
\item
Every NF band function appears in the list $\mathcal E^\#$ exactly once. 
In other words, no two NF band functions are identically equal. 
\item \label{item:anbr1}
No more than two NF band functions may intersect at any given point. In other words, 
for any $k_*\in\bbR$ and any 
$E_*\not=0$, there are no more than two NF band functions $E$ satisfying $E(k_*)=E_*$. 
\item \label{item:anbr1a}
If two NF band functions intersect at a point $k_*$, then $k_*=0\mod\omega/2$. 
\item \label{item:anbr2}
Every NF band function satisfies $E'(k)\not=0$ for $k\not=0\mod \omega/2$. 
In particular, each NF band function is strictly monotone on $(-\omega/2,0)$ and on $(0,\omega/2)$. 
\item\label{item:anbr3}
Let $E_1$ and $E_2$ be two NF band functions, and let $k_1, k_2\in\mathbb R$ 
be such that $E_1(k_1) = E_2(k_2)$. Then $k_1 = k_2\mod \omega$ or $k_1 = -k_2\mod \omega$. 
\item\label{item:anbr4}
For any two NF band functions $E_1$ and $E_2$ 
and any pair $k_1, k_2\in\mathbb R$ we always have $E_1(k_1)\not = - E_2(k_2)$. 
\end{enumerate} 
\end{theorem}

\begin{proof} 
(i) will follow from (iii) because if $ E_1\equiv E_2$ 
then in particular $E_1(k_*) = E_2(k_*)$ for all $k_*\in(0,\omega/2)$. 

(ii) 
Suppose  that $m$ NF band functions $E_1, E_2, \dots, E_m$ intersect at a given point $k_*$:
\begin{align*}
E_* := E_1(k_*) = E_2(k_*) = \dots = E_m(k_*).
\end{align*}
In particular, this means that $\Delta^\#(\cdot; E_*)$ is not a constant function (otherwise this constant would be zero, which is not allowed by the construction of $\Delta^\#$).
Representing $\Delta^\#$ as 
\begin{align*}
 \Delta^\#(k; E_*) = \prod_{j=1}^m\big(1- E_j(k)/E_*\big)\mathcal D(k) 
\end{align*}
with a real analytic $\mathcal D$, we conclude that 
$\Delta^\#(k; E_*) = O((k-k_*)^m)$ as $k\to k_*$. For  
$m \ge 3$ this bound contradicts Lemma~\ref{lem:resid}\eqref{item:resid2} since the zeros of the derivative 
$(\Delta^\#)'_k$ must be simple. Thus $m\le 2$, as required.

(iii)
Let $E_1$ and $E_2$ be two distinct NF band functions such that 
$E_1(k_*) = E_2(k_*)=:E_*$. 
Thus, $k=k_*$ is a double zero of 
$\Delta^\#(\cdot;E_*)$. 
Since $\Delta^\#(\cdot; E_*)$ is even, $k=-k_*$ is another double zero. 
If $k_*\not=0 \mod \omega/2$, then $k_*$ and $-k_*$ are not congruent and 
so we obtain two double zeros in a period cell -- contradiction.

(iv) 
Let us differentiate the equation $\Delta^\#(k; E(k)) = 0$ with respect to $k$:
\begin{align*}
(\Delta^\#)'_k(k; E(k)) + (\Delta^\#)'_\lambda(k; E(k))\, E'(k)=0.
\end{align*}
We note that $\Delta^\#(\cdot; E(k))$ cannot be constant, otherwise this constant would be zero and this contradicts the definition of $\Delta^\#$ (see Lemma~\ref{lem:resid}(i)). 
Thus, by Lemma~\ref{lem:resid}\eqref{item:resid2}, for any $k\not=0 \mod \omega/2$ the first term on the left-hand side is non-zero, and hence so is the second term. 
In particular, $E'(k)\not = 0$, as claimed.

(v) Denote $E_* = E_1(k_1) = E_2(k_2)$, so that $k_1, k_2$ are 
two zeros of the determinant $\Delta^\#(k; E_*)$.  
Assume first that $k_2\not = 0\mod \omega/2$. 
Since $k\mapsto\Delta^\#(k; E_*)$ is even, the 
point $-k_2$ is also a zero non-congruent to $k_2$. 
Since the determinant $\Delta^\#$ has exactly two non-congruent zeros, 
the value $k_1$ must be congruent to $k_2$ or to $-k_2$. 

Suppose that $k_2 = 0\mod\omega/2$. By Lemma~\ref{lem:resid}\eqref{item:resid2} the value 
$k_2$ is a double zero 
of $\Delta^\#(k; E_*)$. Thus $k_1 = k_2\mod \omega$.  

(vi) 
We may assume that $E_1$ is positive and $E_2$ is negative. 
To get a contradiction, assume that for some $k_1,k_2$ we have 
$$
E_1(k_1)= - E_2(k_2)=:E_*, \quad \textup{and hence}\ \quad 
\Delta^\#(k_1; E_*) = \Delta^\#(k_2;-E_*) = 0.
$$
By \eqref{ee21a}, 
$\Delta^\#(k_2+i;E_*) = \Delta^\#(k_2;-E_*)=0$, and so 
$s=k_2+i$ is another zero of the determinant $\Delta^\#(s;E_*)$.
If $k_1\not = 0 \mod \omega/2$, then by the fact that 
$\Delta^\#(k;E_*)$ is even,
the value $s=-k_1$ is yet another non-congruent zero. Thus we have more than two zeros of 
$\Delta^\#$ in the fundamental domain, which gives a contradiction.

If $k_1 = 0\mod\omega/2$, then 
by Lemma~\ref{lem:resid}\eqref{item:resid2}
the zero at $s = k_1$ is double. Therefore, together 
with the zero at $k_2+i$, this gives  
more than two zeros in the fundamental domain, as before.   
 This leads to a contradiction which completes the proof. 
\end{proof}

The next theorem deals with the points $k_0 = 0\mod \omega/2$, where a pair of NF band functions may intersect. The two alternative scenarios discussed in this theorem are illustrated in Figure~\ref{fig:double}.

\begin{theorem} \label{thm:double}
Let $k_0 = 0\mod\omega/2$. 
\begin{enumerate}[{\rm (i)}]
\item\label{item:deg}
Let $E\in\mathcal E^\#$ be  such that no other 
NF band function intersects it at the point $(k_0, E(k_0))$. 
Then $E(k) = E(2k_0-k)$ (i.e. $E$ is even with respect to $k\mapsto k_0-k$) and the point 
$k = k_0$ is a non-degenerate extremum of 
$E(k)$, i.e. $E'(k_0) = 0$ and $E''(k_0)\not = 0$. 
\item\label{item:cross}
Suppose that two distinct NF band functions $E_1$ and $E_2$ intersect at $k_0$: 
$E_* = E_1(k_0) = E_2(k_0)$. Then 
$E_2(k) = E_1(2k_0-k)$, and 
$E_1$ and $E_2$ intersect transversally, i.e. 
$E_1'(k_0) = - E_2'(k_0) \not = 0$.
\end{enumerate}
\end{theorem}

\begin{figure}[h!]
\begin{tikzpicture}
\node at (0,0){\includegraphics[width=6cm]{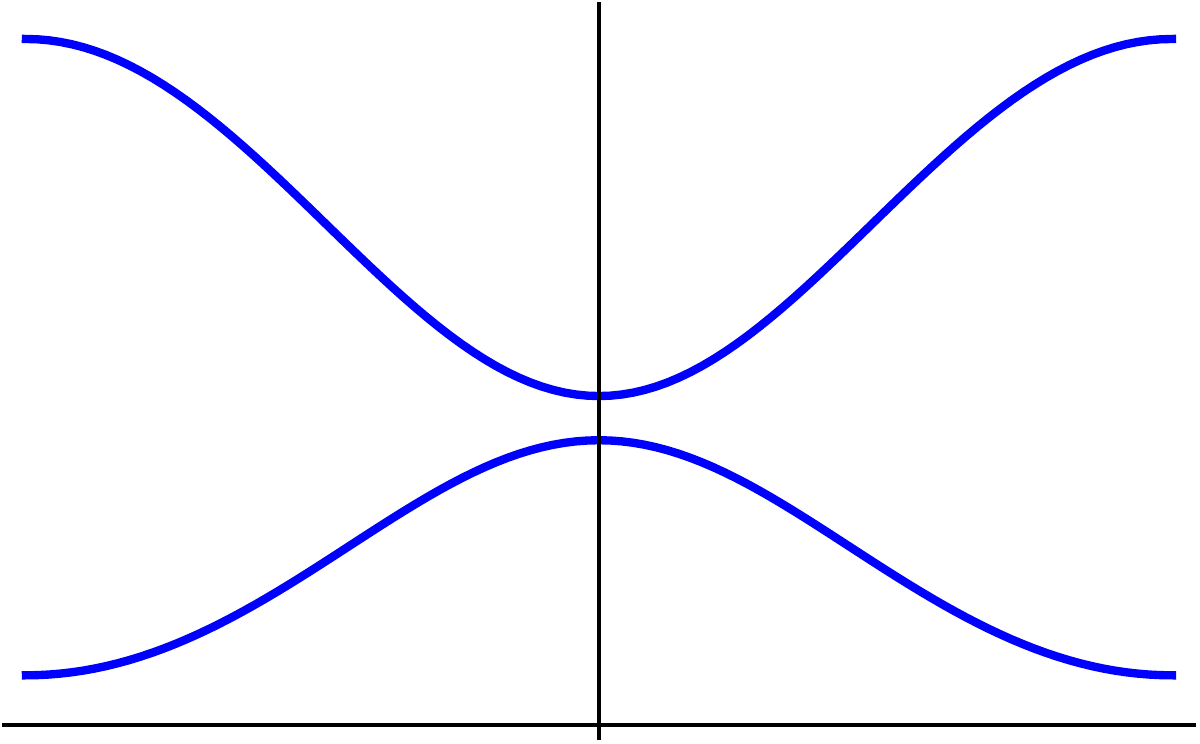}};
\draw (0.3,1.8) node{$E$};
\draw (0.1,-2.1) node{$k_0$};
\draw (3,-2) node{$k$};
\end{tikzpicture}
\quad
\begin{tikzpicture}
\node at (0,0){\includegraphics[width=6cm]{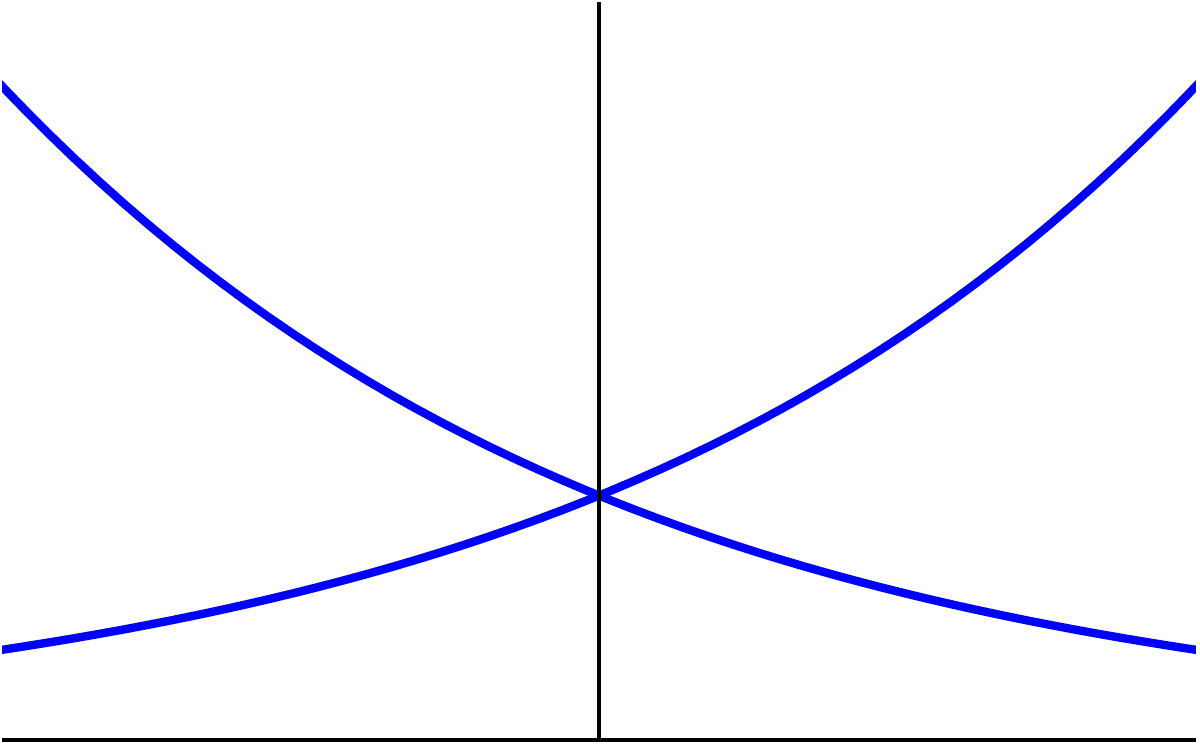}};
\draw (0.3,1.8) node{$E$};
\draw (0.1,-2.1) node{$k_0$};
\draw (3,-2) node{$k$};
\end{tikzpicture}
\caption{Behaviour of band functions near $k_0$.}
\label{fig:double}
\end{figure}
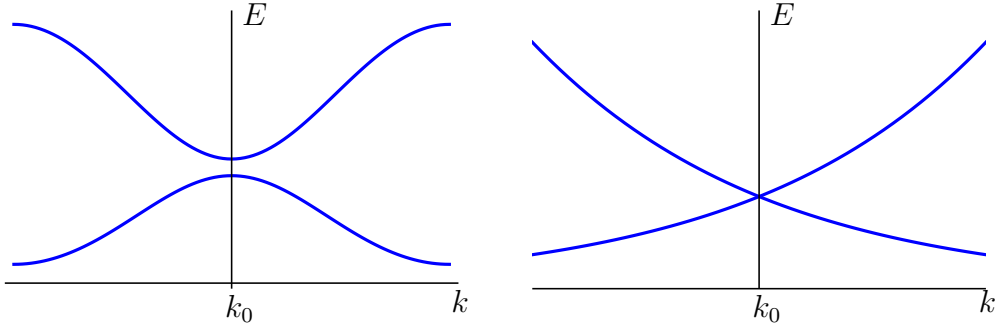

\begin{proof} By Theorem~\ref{thm.branches}, 
along with $E$, the functions $E(-k)$ and 
$E(k+\omega n)$, $n\in\mathbb Z$, are also NF band functions 
(not necessarily distinct from $E$).  We single out one branch of this type:
\begin{align*}
\wt {E}(k) = E(2k_0-k),
\end{align*} 
that satisfies $\wt {E}(k_0) = E(k_0)$.

{\rm (i)} 
By assumption, $E(k)$ is the only NF band function 
passing through the point 
$(k_0, E_*),\, E_* = E(k_0)$. 
On the other hand, $E(k_0) = \wt{E}(k_0)$. Hence 
the function $E(k)$ coincides with $\wt{E}(k)$, that is 
$E(2k_0-k) = E(k)$, which implies that $E'(k_0) = 0$.
To prove the non-degeneracy, 
rewrite the residual secular determinant as
\begin{align*}
\Delta^\#(k; E_*)=(1 - E(k)/E_*)\mathcal D(k),
\end{align*}
where $\mathcal D$ is analytic near $k=k_0$ and  
$\mathcal D(k_0)\not=0$. 
Remembering that $E'(k_0) = 0$, we differentiate twice:
\begin{align*}
(\Delta^\#)''_k(k_0; E_*) 
=&-(E''(k_0)/E_*)\mathcal D(k_0)-2(E'(k_0)/E_*)\mathcal D'(k_0)
\\ &+(1 - E(k_0)/E_*)\mathcal D''(k_0)
\\ =& -(E''(k_0)/E_*)\mathcal D(k_0).
\end{align*}
By Lemma~\ref{lem:resid}\eqref{item:resid2} 
the point $k_0$ is a non-degenerate zero of $(\Delta^\#)'_k$, so that 
$(\Delta^\#)''_k(k_0; E_*)\not = 0$. This implies that $E''(k_0) \not = 0$ as well. 

{\rm(ii)} 
We represent the residual secular determinant as 
$$
\Delta^\#(k; E_*)=(1 - E_1(k)/E_*)(1- E_2(k)/E_*)\mathcal D(k),
$$
where $\mathcal D(k)$ is analytic near $k_0$ with $\mathcal D(k_0)\not=0$. 
We easily compute that 
\begin{align*}
(\Delta^\#)''_k(k_0;E_*) = \frac{2 E_1'(k_0)\, E_2'(k_0)}{E_*^2}\,\mathcal D(k_0).
\end{align*}
By Lemma~\ref{lem:resid}\eqref{item:resid2} 
again, the left-hand side does not vanish, and hence 
$E_1'(k_0) \, E_2'(k_0)\not = 0$. 
Thus, both these derivatives are non-zero. 

Since $\wt{E}_1(k_0)=E_1(k_0)$, we see that the band function 
$\wt{E}_1$ also intersects $E_1$ at the same point. 
By Theorem~\ref{thm:anbr}\eqref{item:anbr1}, it follows that either $\wt{E}_1=E_1$ or $\wt{E}_1=E_2$. 
In the first case we find $E_1'(k_0)=0$, so this possibility is ruled out. Thus we find that $\wt{E}_1=E_2$, as claimed. 
\end{proof}

\subsection{Reduction to multiplication operators}
\label{sec.ddd4}
The following theorem is a preliminary step towards the proof of Theorem~\ref{thm.main}.
\begin{theorem}\label{thm:osum}
Let $\mathbf H$ be a smooth periodic Hankel operator, and let $\mathcal E$ 
be the collection of the band functions of $\mathbf H$ from Theorem~\ref{thm.branches}. 
Then $\mathbf H|_{(\Ker \mathbf H)^\perp}$ is unitarily equivalent to the orthogonal sum
\begin{equation}
\bigoplus_{E\in\mathcal E} \calM[E; \Omega],
\label{eq.osum}
\end{equation}
where $\calM[E; \Omega]$ is the operator of multiplication by $E(k)$ in 
$L^2(\Omega)$. 
\end{theorem}
\begin{proof}
This is a standard argument, compare e.g. \cite[Theorem~XIII.86]{RS4}. 
We use the representation \eqref{b4}, i.e. we consider the operator
$$
\wt{\mathbf H}:=\calU\mathbf H\calU^*=\int_{\Omega}^\oplus H(k)\, dk
\quad\text{ in }\quad
\mathfrak H=\int_{\Omega}^\oplus \ell^2\, dk.
$$
It suffices to prove the stated unitary equivalence for $\wt{\mathbf H}|_{(\Ker \wt{\mathbf H})^\perp}$.

Let us enumerate the band functions $E\in \mathcal E$ in an arbitrary way,  so that  
$\mathcal E=\{E_n\}_{n=1}^\infty$. For each $n\geq1$, let $\psi_n(k)$, $k\in\bbR$, 
be the normalized eigenvector corresponding to $E_n(k)$. 
According to the standard analytic perturbation theory, 
see e.g. \cite[Ch. II,\S4 and Ch. VII, \S 3]{Kato} or \cite[Section~XII.2]{RS4}, 
the eigenvectors $\psi_n(k)$ can be taken real analytic in $k$. Consider the subspaces 
$$
\mathfrak H_n=\{\psi\in\mathfrak H: \psi(k)=f(k)\psi_n(k),\  f\in L^2(\Omega)\}, \quad n\geq1.
$$
Then $\mathfrak H_n$ are pairwise orthogonal closed subspaces in $\mathfrak H$. 
It is clear that each subspace $\mathfrak H_n$ is invariant for 
$\wt{\mathbf H}$, and the restriction of 
$\wt{\mathbf H}$ to $\mathfrak H_n$ 
is unitarily equivalent to $\calM[E_n;\Omega]$. 
Finally, it is clear that the orthogonal complement to $\oplus_{n\geq1}\mathfrak H_n$ in $\mathfrak H$ coincides with the kernel of $\wt{\mathbf H}$. This completes the proof. 
\end{proof}

\subsection{Proof of Theorem~\ref{thm.main}}
\label{subsect:proofmain}
Let us consider the orthogonal sum \eqref{eq.osum} for $E\in \mathcal E^\flat$ and $E\in \mathcal E^\#$ 
separately. 
 
 (i) 
The terms $\calM[E; \Omega]$ in \eqref{eq.osum} with 
$E\in{\mathcal E}^\flat$ give eigenvalues of infinite multiplicity. These eigenvalues are symmetric with respect to zero by Lemma~\ref{lma.d1}. If there are infinitely many flat band functions, they must accumulate to zero (and to no other point) because they are eigenvalues of a compact operator $H(k)$ (for any $k\in\bbR$). This proves part (i) of the theorem. 

 (ii) 
If ${\mathcal E}^\#$ is empty, part (ii) of the theorem is trivially satisfied. 
Assume that it is not empty. The terms $\calM[E; \Omega]$ with $E\in{\mathcal E}^\#$ 
in \eqref{eq.osum} yield intervals of 
the absolutely continuous spectrum. In order to describe them, it is convenient to write 
\begin{equation}
\bigoplus_{E\in{\mathcal E}^\#} \calM[E; \Omega]
=
\biggl(\bigoplus_{E\in{\mathcal E}^\#} \calM[ E; (-\omega/2,0)]\biggr)
\bigoplus
\biggl(\bigoplus_{E\in{\mathcal E}^\#} \calM[E; (0,\omega/2)]\biggr),
\label{eq.osum2}
\end{equation}
where we split every operator of multiplication in $L^2(\Omega)$ into the orthogonal sum of the operators of multiplication in $L^2(-\omega/2,0)$ and $L^2(0,\omega/2)$. 
By the last statement of Theorem~\ref{thm.branches}, the two terms in the right-hand side of \eqref{eq.osum2} are unitarily equivalent to each other. Thus, it suffices to consider just one of them, bearing in mind that the multiplicity of the absolutely continuous spectrum has to be doubled. We consider the second term in \eqref{eq.osum2}.

Let us enumerate arbitrarily the NF band functions and write 
${\mathcal E}^\# = \{E_n\}$.
Then for every index $n$ 
the spectrum of $\calM[E_n, \Omega]$ coincides with the interval (NF spectral band)  
\begin{align}\label{eq:bandsgen}
\sigma_n=\{ E_n(k): k\in[0,\omega/2]\}.
\end{align}
By Theorem~\ref{thm:anbr}\eqref{item:anbr2},  
each real analytic function $E_n$ satisfies $E_n'(k)\not = 0$ 
on $(0,\omega/2)$, and hence the spectrum $\sigma_n$ 
is absolutely continuous of multiplicity one.
As all band functions are non-vanishing, $\sigma_n$ does not contain zero. 

Let us check that the intersection of any two NF bands can contain at most one point. Suppose, to get a contradiction, that for some $n\not=m$, the intersection of the interior of $\sigma_n$ with the interior of $\sigma_m$ is non-empty. This means that $E_n(k_1) = E_m(k_2)$ for some $k_1,k_2\in(0,\omega/2)$. By Theorem~\ref{thm:anbr}\eqref{item:anbr1a}, we have $k_1\not=k_2$ because two NF band functions cannot intersect inside $(0,\omega/2)$. On the other hand, by Theorem~\ref{thm:anbr}\eqref{item:anbr3}, either $k_1 = k_2\mod\omega$ or $k_1 = -k_2\mod \omega$. This is incompatible with $k_1,k_2\in(0,\omega/2)$. Thus the intersection of the interiors of bands is empty, as claimed.  

From here it follows that the multiplicity of the absolutely continuous 
spectrum of the second term in \eqref{eq.osum2} is one. 
Therefore, the multiplicity of spectrum of the whole sum in \eqref{eq.osum2} is two, as required. 

The intersection $(-\sigma_n)\cap\sigma_m$ is empty by 
Theorem~\ref{thm:anbr}\eqref{item:anbr4}.

Suppose the set $\mathcal E^\#$ is infinite. Fix some $k_*\in(0,\omega/2)$; 
since all values $E_n(k_*)$ are eigenvalues of the compact operator 
$H(k_*)$, we have $E_n(k_*)\to0$ as $n\to\infty$. 
Thus there exists a sequence of points each of which lies in the interior of 
$\sigma_n$, that converges to zero.  
As the interiors are disjoint, the spectral bands $\sigma_n$ cannot have 
any other accumulation point. 

 (iii) 
None of the terms in \eqref{eq.osum} has singular continuous spectrum, which proves part (iii).

The proof of Theorem~\ref{thm.main} is complete. 
\qed

\begin{remark}\label{rem:bands}
\begin{enumerate}[{\rm (i)}]
\item 
The $n$'th band $\sigma_n$ is given by $[E_n(0), E_n(\omega/2)]$ or $[E_n(\omega/2), E_n(0)]$, 
depending on the monotonicity of $E_n$. 
\item 
According to Theorem~\ref{thm.main}, any two spectral bands \eqref{eq:bandsgen} do not have common interior points, and hence they can only touch (i.e. have one common endpoint). We claim that if two bands touch, then the corresponding band functions necessarily intersect (see second scenario in Figure~\ref{fig:double}).
To see this, assume that two bands $\sigma_<$ and $\sigma_>$ touch and $\sigma_>$ is above $\sigma_<$, i.e. the corresponding band functions satisfy $E_<(k_1)\le E_>(k_2)$ for all $k_1, k_2\in [0, \omega/2]$. Suppose for definiteness that $E_<$ is decreasing on $(0,\omega/2)$, so that $\sigma_< = [E_<(\omega/2), E_<(0)]$. 
Since the bands $\sigma_<$ and $\sigma_>$ touch, we have $E_<(0) = E_>(k_*)$ where $k_* = 0$ or $k_*=\omega/2$. By Theorem~\ref{thm:anbr}\eqref{item:anbr3}, the only possible choice is $k_* =0$. Thus, $E_<$ and $E_>$ intersect at $k_*=0$. 
\end{enumerate}
\end{remark}

\subsection{The Carleman operator}\label{sec.b8} 
Let us illustrate the structure of band functions by considering the 
Carleman operator, i.e. the Hankel operator with the kernel function
$h(t)=1/t$. The Carleman operator has purely absolutely continuous spectrum $[0, \pi]$ of multiplicity two (see e.g. \cite[Section~10.2]{Peller}).
In this section we consider this operator as 
a periodic Hankel operator with the kernel function $p(\log t)/t$, where the function $p(x) = 1$ is thought of as periodic with an arbitrary period $T>0$. 
In this case the band functions and hence the spectral bands can be found explicitly. 
For this operator we have 
${\wt p}_0=1$ and ${\wt p}_\ell=0$ for 
$\ell\not=0$, so the matrix of $H(k)$ is diagonal. 
In other words, $H(k)$ is the operator of multiplication by the sequence
\[
E_n(k)=\Gamma(\tfrac12-i \omega n - ik)\Gamma(\tfrac12+i\omega n+ ik)
=\frac{\pi}{\cosh\pi(k+\omega n)},  \quad n\in\bbZ.
\]
The collection $\mathcal E = \{E_n\}_{n\in\bbZ}$ is the family 
of band functions of the Carleman operator.  
Clearly, $E_n(k)>0$ for all $n$ and $k$, 
which reflects the fact that the operator is positive semi-definite. 
Furthermore, we have $E_n(-k) = E_{-n}(k)$ and 
$E_n(k+\omega m) = E_{n+m}(k)$, i.e. the family 
$\mathcal E$ is invariant with respect to the reflection 
$k\mapsto -k$ and to the shifts $k\mapsto k+\omega m$, which  agrees with 
Theorem~\ref{thm.branches}. 
\begin{figure}
\begin{tikzpicture}
\node at (0,0){\includegraphics[width=0.7\textwidth]{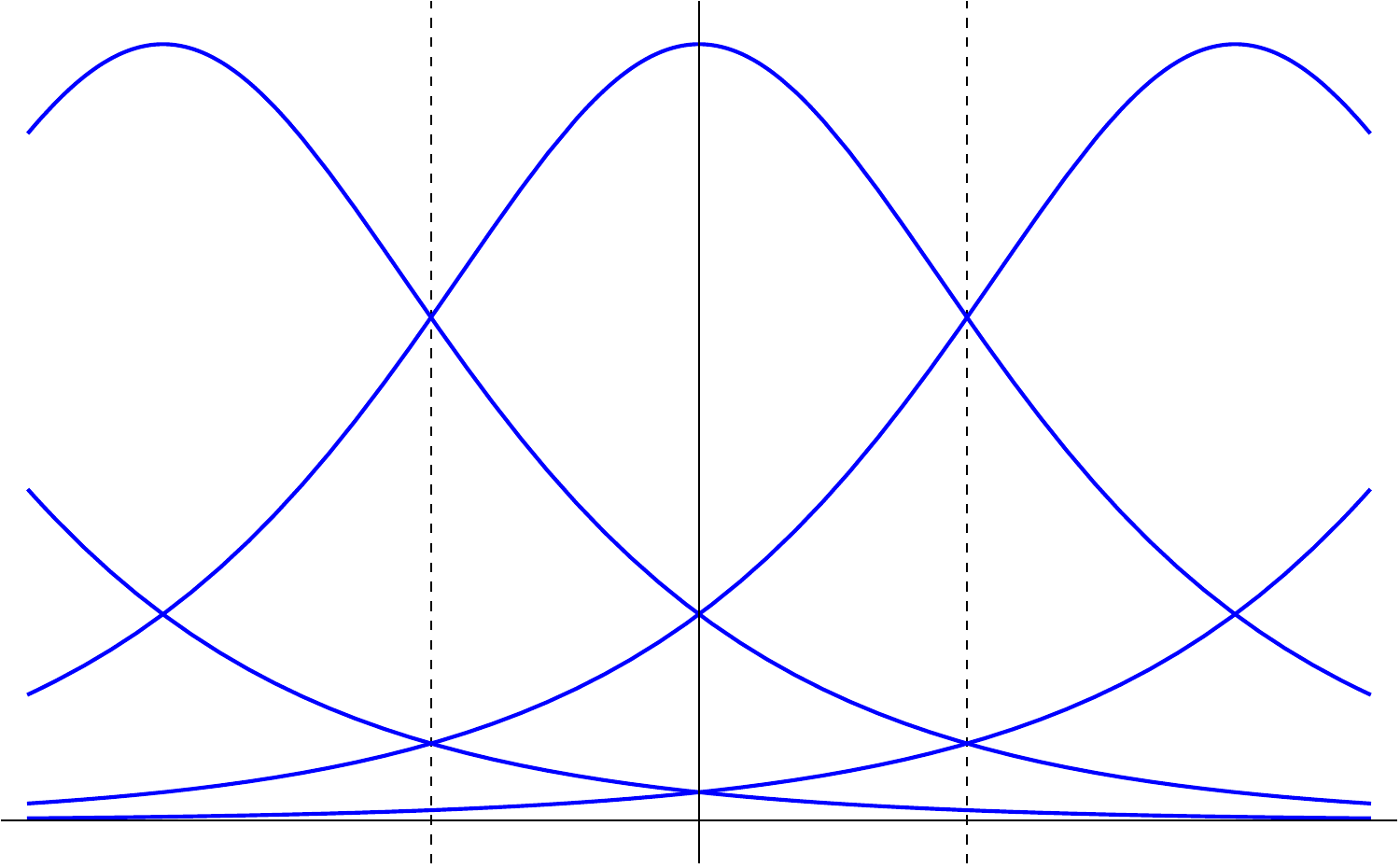}};
\draw (1.2,2.7) node{$\widetilde{E}_0$};
\draw (1.2,0.5) node{$\widetilde{E}_1$};
\draw (1.2,-1.5) node{$\widetilde{E}_2$};
\draw (0.2,3.5) node{$E$};
\draw (5.5,-3.2) node{$k$};
\draw (-0.3,-3.2) node{$0$};
\draw (2.4,-3.2) node{$\omega/2$};
\draw (-2.6,-3.2) node{$-\omega/2$};
\end{tikzpicture}
\caption{Band functions of the Carleman operator.}
\label{figure4} 
\end{figure}
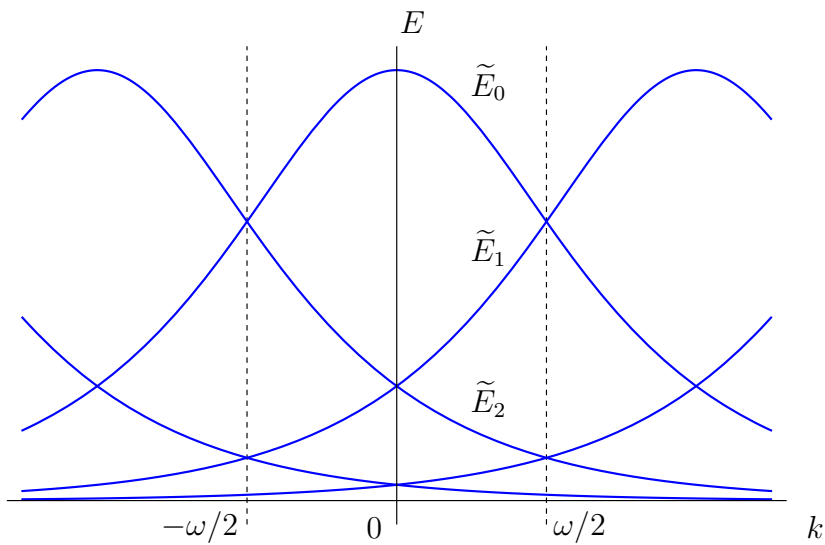

To find the spectral bands we enumerate the band functions 
in a different way. 
Let $\{\widetilde{E}_n\}_{n=0}^\infty$ be the enumeration of the family $\mathcal E$ such that
$$
\widetilde{  E}_0(k)>\widetilde{  E}_1(k)>\cdots>0, \quad k\in(0,\omega/2).
$$
In this case, the sequence 
$\{\widetilde{E}_n(k)\}_{n=0}^\infty$, $k\in(0,\omega/2)$, 
coincides with the set of eigenvalues of $H(k)$, listed in non-decreasing order. 
It is straightforward to conclude that 
\begin{align*}
\widetilde{E}_{2n}(k)&=\frac{\pi}{\cosh\pi(k+\omega n)}, \quad n\geq0,
\\
\widetilde{E}_{2n+1}(k)&=\frac{\pi}{\cosh\pi (k-\omega (n+1))}, \quad n\geq0.
\end{align*}
We see that all even band functions 
$\widetilde{ E}_{2n}$ are decreasing and all 
odd band functions $\widetilde{ E}_{2n+1}$ 
are increasing on $(0,\omega/2)$, which agrees with Theorem~\ref{thm:alterpos} below. 
The spectral bands (enumerated from the top down) are 
\begin{align*}
\sigma_{2n}&=[\wt E_{2n}(\omega/2),\wt E_{2n}(0)],\quad n\geq0,
\\
\sigma_{2n+1}&=[\wt E_{2n+1}(0),\wt E_{2n+1}(\omega/2)], \quad n\geq0.
\end{align*}
Observe that 
\begin{align*}
\widetilde E_{0}(0) = \pi, \  
&\ \widetilde E_{2n}(\omega/2) = \widetilde E_{2n+1}(\omega/2) = \frac{\pi}{\cosh \pi\omega (n+1/2)},\\
\widetilde E_{2n+2}(0) = &\ \widetilde E_{2n+1}(0) = \frac{\pi}{\cosh \pi\omega (n+1)},
\end{align*}
for all $n\geq0$.
Thus (in line with Theorem~\ref{thm.main}), the spectral bands have disjoint interiors and accumulate at zero, and there are no gaps between consecutive bands. As a consequence, the absolutely continuous spectrum given by \eqref{eq:ac}, coincides with $[0,\pi]$, as expected.  

\subsection{Degenerate flat bands}
\label{sec.ddd-flat}

As we saw earlier, properties of elliptic functions provide a powerful tool for the analysis of non-flat band functions. Unfortunately, this approach 
is not applicable when studying flat bands. 
For example, can a flat band function intersect a non-flat one? 
Can a flat band be degenerate? 
We do not have an answer to the first question. 
On the other hand, we can 
claim that in contrast to the non-flat bands, the flat ones can be degenerate. 
Below we explain how flat bands of multiplicity $>1$ can occur. 
 
We start with a $T$-periodic Hankel operator $\mathbf H$ that has a flat band $E_*$. (One can take, for example, the Mathieu-Hankel operator for a suitable choice of the parameter $A$, see Section~\ref{sec.e}.) Let us consider the same operator $\mathbf H$ as a Hankel operator with the period $2T$. Let us denote by $H^{(2T)}(k)$, $\widetilde{p}_n^{(2T)}$, $\widetilde{\mathfrak s}^{(2T)}_n$ etc. the same objects corresponding to the period $2T$. 

Clearly, we have $\omega^{(2T)}=\frac12\omega$ and 
\[
\widetilde{p}_n^{(2T)}
=
\begin{cases}
\widetilde{p}_{n/2}
&\text{ if $n$ even,}
\\
0
&\text{ if $n$ odd.}
\end{cases}
\]
It follows that $\widetilde{\mathfrak s}^{(2T)}_n=0$ if $n$ is odd and 
\[
\widetilde{\mathfrak s}^{(2T)}_{2n}
=
\frac{\widetilde{p}_n}{\Gamma(1-i\omega^{(2T)}2n)}
=
\frac{\widetilde{p}_n}{\Gamma(1-i\omega n)}
=
\widetilde{\mathfrak s}_n.
\]
According to \eqref{eq:bk}, we find
\[
[H^{(2T)}(k)]_{n,m} = \overline{\gamma_n^{(2T)}(k)} \widetilde{\mathfrak s}_{n-m}^{(2T)}\gamma_m^{(2T)}(k), 
\]
where
\[
\gamma_n^{(2T)}(k)  
= \Gamma\big(\tfrac12 + i(\omega^{(2T)} n + k)\big)
= \Gamma\big(\tfrac12 + i(\tfrac12\omega n + k)\big).
\]
So we can write
\[
\gamma_{2n}^{(2T)}(k)=\gamma_n(k), 
\quad
\gamma_{2n+1}^{(2T)}(k)=\gamma_n(k+\tfrac12\omega).
\]
Since $\widetilde{\mathfrak s}^{(2T)}_n=0$ for odd $n$, we find
\[
[H^{(2T)}(k)]_{n,m} =0
\]
if $n-m$ is odd. It follows that the two subspaces of $\ell^2(\bbZ)$ corresponding to even and odd indices are invariant for $H^{(2T)}(k)$. Let us write this more formally as follows: 
\[
H^{(2T)}(k)
=
\begin{pmatrix}
H^{(2T)}_{\mathrm{even}}(k) & 0
\\
0 &H^{(2T)}_{\mathrm{odd}}(k)
\end{pmatrix}
\]
with respect to the decomposition
\[
\ell^2(\bbZ)=\ell^2_{\mathrm{even}}(\bbZ)\oplus\ell^2_{\mathrm{odd}}(\bbZ), 
\]
where
\begin{align*}
\ell^2_{\mathrm{even}}(\bbZ)=&\ \{x\in\ell^2(\bbZ): x_{2n+1}=0, \quad n\in\bbZ\},\\
 \ell^2_{\mathrm{odd}}(\bbZ)=&\ \{x\in\ell^2(\bbZ): x_{2n}=0, \quad n\in\bbZ\}.
\end{align*}
%
%
For the matrix entries of $H^{(2T)}_{\mathrm{even}}(k)$ we have
\begin{align*}
[H^{(2T)}_{\mathrm{even}}(k)]_{2n,2m}
&=
[H^{(2T)}(k)]_{2n,2m} 
= 
\overline{\gamma_{2n}^{(2T)}(k)} \widetilde{\mathfrak s}_{2n-2m}^{(2T)}\gamma_{2m}^{(2T)}(k)
\\
&=
\overline{\gamma_{n}(k)} \widetilde{\mathfrak s}_{n-m}\gamma_{m}(k)
=
[H(k)]_{n,m} 
\end{align*}
and similarly 
\begin{align*}
[H^{(2T)}_{\mathrm{odd}}(k)]_{2n+1,2m+1}
&=
[H^{(2T)}(k)]_{2n+1,2m+1} 
= 
\overline{\gamma_{2n+1}^{(2T)}(k)} \widetilde{\mathfrak s}_{2n-2m}^{(2T)}\gamma_{2m+1}^{(2T)}(k)
\\
&=
\overline{\gamma_{n}(k+\tfrac12\omega)} \widetilde{\mathfrak s}_{n-m}\gamma_{m}(k+\tfrac12\omega)
=
[H(k+\tfrac12\omega)]_{n,m}.
\end{align*}
We conclude that $H^{(2T)}(k)$ is unitarily equivalent to the orthogonal sum
\[
H(k)\oplus H(k+\tfrac12\omega). 
\]
It follows that list of the band functions of $H^{(2T)}(k)$ is given by the union of two lists: $\mathcal E=\{E(k)\}$ and $\{E(k+\tfrac12\omega)\}$, where $\mathcal E$ is the list of the band functions of $H(k)$. It follows that if $E_*$ is a flat band function of $H(k)$ with multiplicity $m$, then it is a flat band function of $H^{(2T)}(k)$ with multiplicity $2m$.  In particular, the multiplicity of a flat band can be made as large as we wish by multiplying the period by a large integer.

\section{The alternation pattern for band functions}
\label{sec.alternation}

\subsection{Main statement}
Let $\mathbf H$ be a smooth periodic Hankel operator, and let $\mathcal E^\#$ be the list of the corresponding non-flat band functions. By Theorem~\ref{thm:anbr}\eqref{item:anbr2}, each band function $E\in\mathcal E^\#$ is either strictly increasing or strictly decreasing on $(0,\omega/2)$. For an individual band function, can one tell if it is increasing or decreasing? In this section, we give an answer. 
We establish a global property of the band functions of $\mathbf H$ which we call an \emph{alternation pattern}. It says that the band function $E$ is increasing or decreasing depending on the parity of its order in the list of all band functions, ordered by decreasing absolute value.

Recall that the band functions $E(k)$ are real analytic in $k\in\bbR$ and do not vanish. 
It follows that the absolute values $|E(k)|$ are also real analytic in $k\in\bbR$. 
Consider the non-flat band functions $E\in\mathcal E^\#$ and let us enumerate them in such a way that 
\begin{align}\label{eq:moden}
 |E_0(k)| > |E_1(k)| > \dots,\quad k\in (0, \omega/2), \quad E_n\in\mathcal E^\#.
\end{align}
This is possible due to Theorem~\ref{thm:anbr}\eqref{item:anbr1a}, \eqref{item:anbr4}, since the band functions do not intersect for these values of $k$.
Now we can establish the global alternation pattern for the band functions $E_n\in\mathcal E^\#$. 

\begin{theorem}\label{thm:mod}
Let $\mathbf H$ be a smooth periodic Hankel operator and let $\{E_n\}$ be its non-flat band functions enumerated as stated above. 
Then for each $k\in (0, \omega/2)$ we have
\begin{align}\label{eq:alter}
(-1)^{n+1} \frac{d}{d k} \, |E_n(k)| > 0, \quad n = 0, 1, \dots,
\end{align}
i.e. the even band functions $|E_0(k)|$, $|E_2(k)|$, 
etc. attain a maximum at $k=0$ and a minimum at $k=\omega/2$. 
Conversely, the odd band functions $|E_1(k)|$, $|E_3(k)|$,
etc. attain a minimum at $k=0$ and a maximum at $k=\omega/2$. 
\end{theorem}

In Section~\ref{sec.alternation2} we present simple corollaries of this theorem for positive operators and for operators bounded from below. 
Next, in Section~\ref{sec.alternation3} we discuss a standard Jacobi elliptic function which plays a key role in the proof of Theorem~\ref{thm:mod}.
After some preparations, the proof of Theorem~\ref{thm:mod} is given in Section~\ref{sec.d5a}. The proof is based on a representation \eqref{eq:repr} for the (residual) secular determinant which may be interesting in its own right.

\subsection{Spectral bands for positive Hankel operators} 
\label{sec.alternation2}

Let  $\mathbf H$ be a positive semi-definite periodic Hankel operator. 
For such an operator there are no negative flat bands, and therefore, by the symmetry of 
eigenvalues (see Theorem~\ref{thm.main}(ii)), there are no positive flat bands either, so all spectral 
bands are positive and non-flat. 
In this case, the 
ordering \eqref{eq:moden} translates into the standard enumeration of band functions in decreasing order:
$$
E_0(k)> E_1(k)>\cdots>0, \quad k\in(0,\omega/2), \quad E_n\in\mathcal E.
$$
Therefore Theorem~\ref{thm:mod} immediately leads to the 
following global alternation property.

\begin{theorem}\label{thm:alterpos}
Assume that 
$\mathbf H$ is a positive semi-definite smooth periodic Hankel operator.  
Then the non-zero spectrum of $\mathbf H$ does not have any flat bands and hence consists entirely of the 
non-flat bands. Furthermore, for each $k\in (0, \omega/2)$ we have 
\begin{align}\label{eq:alterpos}
(-1)^{n+1} \frac{d}{dk} E_n(k) > 0, \quad n = 0, 1, \dots, 
\end{align}
i.e. 
the even band functions $E_0(k)$, $E_2(k)$, 
etc. attain a maximum at $k=0$ and a minimum at $k=\omega/2$. 
Conversely, the odd band functions $E_1(k)$, $E_3(k)$, 
etc. attain a minimum at $k=0$ and a maximum at $k=\omega/2$. 
\end{theorem}

Theorem~\ref{thm:alterpos} should be compared with the corresponding result for the band functions of the Schr\"odinger operator, see e.g. \cite[Theorem XIII.89]{RS4}.

Let us relax the condition $\mathbf H\ge 0$ to the condition 
\[
\mathbf H+aI\geq0\quad \text{ for some $a>0$.}
\]
Since $\mathbf H$ has no spectrum below $-a$, using Lemma~\ref{lma.d1} we conclude that 
$\mathbf H$ has no flat bands above $a$. 
Assume that there are 
exactly $M+1\ge 1$ (non-flat) band functions of $\mathbf H$ above $a$. 
For these we have $E(k) = |E(k)|$, 
and as in Section~\ref{sec.d9}, the enumeration \eqref{eq:moden} rewrites as 
\begin{equation}
E_0(k)> E_1(k)>\dots> E_M(k)>a, \quad k\in(0,\omega/2). 
\label{eq:nml}
\end{equation} 
Thus we immediately obtain the following corollary of Theorem~\ref{thm:mod}.

\begin{theorem}\label{thm:noman}
Let $\mathbf H$ be a smooth periodic Hankel operator satisfying $\mathbf H+aI\geq0$ and such that for some $M\ge0$ there are exactly $M+1$ spectral band functions satisfying \eqref{eq:nml}. 
Then the inequality \eqref{eq:alterpos} holds for $n = 0, 1, \dots, M$. 
\end{theorem}

\subsection{A reference elliptic function}\label{sec.alternation3}
We remind that the lattices $\sf \Lambda$ and $\sf M$ are 
defined by \eqref{eq:perl} and \eqref{eq:perm} respectively. 

For $s\in\bbC\setminus{\sf \Lambda}$, we set
\begin{align}\label{eq:ellip}
\mathcal P(s) = \sum_{n\in\mathbb Z} \frac{\pi}{\cosh\pi(n\omega + s)}. 
\end{align}
In fact, $\mathcal P(s)$ is the standard Jacobi elliptic function ${\rm dn}(s)$ (up to appropriate scaling).  However, for the sake of transparency and in order to make this text self-contained, we prefer to work with the explicit series \eqref{eq:ellip}  instead of referring to the properties of ${\rm dn}(s)$. 

\begin{lemma}\label{prop:elf}
The function $\mathcal P$ is an even $\sf M$-periodic elliptic function with simple poles at the points of the lattice ${\sf \Lambda}$ with the residues
\begin{align}\label{eq:res}
\Res\limits_{s = i/2} \mathcal P(s) = - \Res\limits_{s = -i/2} \mathcal P(s) = -i
\end{align}
satisfying the condition
\begin{equation}
\mathcal P(s+i)=-\mathcal P(s).
\label{d5a}
\end{equation}
These properties uniquely specify $\mathcal P$. 
The zeros of $\mathcal P$ are congruent to the points $\tfrac{\omega+i}2$ and $\tfrac{\omega-i}2$. 
The restriction of $\mathcal P$ onto the real line is strictly decreasing on the interval $[0, \omega/2]$ with $\mathcal P'(k) < 0$ for all $k\in (0, \omega/2)$. 
\end{lemma} 
\begin{proof} 
The facts that $\mathcal P$ is even and $\sf M$-periodic are straightforward. 
Since $\cosh \pi s$ has the expansions $\cosh \pi s = i\pi(s-i/2) + \cdots$ and $\cosh\pi s = -i\pi (s+i/2)+ \cdots$, we find the residues \eqref{eq:res}. Since $\cosh(a+\pi i)=-\cosh a$, we obtain \eqref{d5a}. 

Uniqueness: if $\widetilde{\mathcal P}$ is another function satisfying the same properties, then by Liouville's theorem we find that $\widetilde{\mathcal P}(s)=\mathcal P(s)+\text{const}$. Condition \eqref{d5a} implies that $\text{const}=0$. 

Zeros: from \eqref{d5a} and periodicity we find $\mathcal P(s+\omega+i)=-\mathcal P(s)$. Since $\mathcal P$ is even, this gives $\mathcal P(s+\omega+i)=-\mathcal P(-s)$. Taking $s=-\tfrac{\omega+i}2$, we find $\mathcal P(\tfrac{\omega+i}{2})=0$. Combining with \eqref{d5a} gives $\mathcal P(\tfrac{\omega-i}{2})=0$. Since $\mathcal P$ is an elliptic function of order two, it has no other zeros on the period cell.

Let us prove the monotonicity of $\mathcal P$ on the interval $[0, \omega/2]$. 
The derivative $\mathcal P'$ is an $\sf M$-periodic function of order four and therefore vanishes at exactly four points on a period cell. Let us check that these are the points congruent to $0,i,\tfrac{\omega}{2},\tfrac{\omega}{2}+i$. As $\mathcal P$ is even, we find $\mathcal P'(0)=0$. Combining periodicity and evenness of $\mathcal P$, we find $\mathcal P(\tfrac{\omega}2+s)=\mathcal P(\tfrac{\omega}2-s)$ and so, differentiating, we find $\mathcal P'(\tfrac{\omega}2)=0$. By \eqref{d5a} we also find $\mathcal P'(i)=\mathcal P'(\tfrac{\omega}2+i)=0$.

The reasoning above shows that 
$\mathcal P(k)$ is strictly monotone in $k\in(0,\omega/2)$, 
with a non-vanishing derivative on this interval. 
We only need to check that $\mathcal P'(k)<0$ on this interval. 
We use the deformation in $\omega$ argument. Let us temporarily 
write $\mathcal P(k;\omega)$ to indicate the dependence on $\omega$ explicitly: 
$$
\mathcal P(k;\omega)=\frac{\pi}{\cosh \pi k}+\sum_{n\not=0} \frac{\pi}{\cosh\pi(n\omega + k)},
$$
and observe that the derivative (with respect to $k$) of the first term in the right-hand 
side  is negative on $(0,\infty)$, while the derivative of 
the second term goes to zero for any fixed $k$ as 
$\omega\to\infty$. 
It follows that for any fixed $k>0$, we have 
$\mathcal P'(k;\omega)<0$ for all sufficiently large $\omega>0$. 

Now suppose, to get a contradiction, that for some $\omega_0>0$ and some $k_0\in(0,\omega_0/2)$ we have $\mathcal P'(k_0;\omega_0)>0$. Since $\mathcal P'(k_0;\omega)<0$ for sufficiently large $\omega$, by the intermediate value theorem we find $\mathcal P'(k_0;\omega_*)=0$ for some $\omega_*>\omega_0$.
This is a contradiction with the fact that $\mathcal P'(k, \omega)\not = 0$ 
for $k\in (0, \omega/2)$ and all $\omega>0$.
\end{proof}

\subsection{A representation for $\Delta^\#(s;\lambda)$}
Our proof of Theorem~\ref{thm:mod} relies on a representation of the residual secular determinant $\Delta^\#$ via the reference elliptic function $\mathcal P(s)$.

\begin{theorem}\label{thm.c9}
Let $\mathbf H$ be a smooth periodic Hankel operator and let $\Delta^\#(s;\lambda)$ be the corresponding secular determinant. Then 
\begin{align}
\Delta^\#(s;\lambda) =-\alpha(\lambda)\mathcal P(s) + \beta(\lambda), 
\label{eq:repr}
\end{align}
where 
\begin{equation}
\alpha(\lambda)=-i\Res\limits_{s=i/2}\Delta^\#(s;\lambda)
\quad\text{ and }\quad
\beta(\lambda)=\Delta^\#(\tfrac{\omega+i}{2};\lambda).
\label{eq:alphabeta}
\end{equation}
The functions $\alpha$ and $\beta$ have the following properties: 
\begin{enumerate}[\rm (i)]
\item
$\alpha$ and $\beta$ are entire in $1/\lambda$; 
\item
$\overline{\alpha(\lambda)} = \alpha(\overline{\lambda})$ and 
$\overline{\beta(\lambda)} = \beta(\overline{\lambda})$, and in particular, 
$\alpha$ and $\beta$ are real-valued for real $\lambda\not = 0$;
\item \label{item:albe}
$\alpha(-\lambda) = - \alpha(\lambda)$ and $\beta(-\lambda) = \beta (\lambda)$;
\item \label{item:nonzero}
for any $E\in\mathcal E^\#$ and any $k\in\bbR$, we have $\alpha(E(k))\not=0$. 
\end{enumerate}
\end{theorem}

\begin{proof}
Fix $\lambda\in\bbC$, $\lambda\not=0$, and let $\alpha(\lambda)$ be as defined by \eqref{eq:alphabeta}.
Recall that, by Lemma~\ref{lem:resid}, the function $\Delta^\#(\cdot;\lambda)$ is even and elliptic with period $\sf M$ and simple poles at the points congruent to $\pm i/2$. 
Using \eqref{eq:res}, we see that the $\sf M$-periodic elliptic function  
$$
\Delta^\#(s;\lambda)+\alpha(\lambda)\mathcal P(s)
$$
has no poles, as hence it is constant. By Lemma~\ref{prop:elf}, we have $\mathcal P(\tfrac{\omega\pm i}{2})=0$, and so the constant can be determined as 
\begin{equation}
\beta(\lambda)=\Delta^\#(\tfrac{\omega\pm i}{2};\lambda).
\label{d1a}
\end{equation}
This gives the representation \eqref{eq:repr}. 
Furthermore, with arbitrary $s_0\not = \frac{\omega\pm i}{2}$ or $\pm\frac{i}{2}$ we find
\begin{align}\label{d1}
\alpha(\lambda) = \frac{\Delta^\#(s_0; \lambda)-\beta(\lambda)}{\mathcal P(s_0)}. 
\end{align}
Since $\Delta^\#(s_0; \lambda)$ is an entire function in $1/\lambda$ 
we obtain that $\beta(\lambda)$ and hence $\alpha(\lambda)$ are 
also entire in $1/\lambda$, which proves 
(i). 

The relations $\overline{\alpha(\lambda)}=\alpha(\overline{\lambda})$ and 
$\overline{\beta(\lambda)} = \beta(\overline{\lambda})$ follow from \eqref{d1a} and \eqref{d1} 
and the equalities $\overline{\mathcal P(s_0)} = \mathcal P(\overline{s_0})$ and 
$\overline{\Delta^\#(s; \lambda)} = \Delta^\#(\overline{s}; \overline{\lambda})$ 
(see Lemma~\ref{lem:resid}). 

Using the identities $\Delta^\#(s+i; \lambda) = \Delta^\#(s; -\lambda)$ and $\mathcal P(s+i)=-\mathcal P(s)$, we find
$$
-\alpha(-\lambda)\mathcal P(s)+\beta(-\lambda)
=\Delta^\#(s;-\lambda)=\Delta^\#(s+i;\lambda)=\alpha(\lambda)\mathcal P(s)+\beta(\lambda), 
$$
and (iii) follows. 

Let us prove (iv) by contradiction. Suppose $E\in\mathcal E^\#$,  $k\in\bbR$, and $\alpha(E(k))=0$. Then 
$$
0=\Delta^\#(k;E(k))=-\alpha(E(k))\mathcal P(k)+\beta(E(k))=\beta(E(k)). 
$$
Thus, $\alpha(E(k))=\beta(E(k))=0$ and so $\Delta^\#(\cdot;E(k))=0$,
which contradicts the definition of the residual determinant $\Delta^\#$. 
\end{proof}

Combining this theorem with the representation $\Delta(s;\lambda)=\Delta^\#(s;\lambda)\Delta^\flat(\lambda)$, we find
$$
\Delta(s;\lambda)=(-\alpha(\lambda)\mathcal P(s)+\beta(\lambda))\Delta^\flat(\lambda).
$$
We will not use the last relation.

\subsection{Proof of Theorem~\ref{thm:mod}}\label{sec.d5a} 
Denote 
$$
d(k; \lambda) = 
\Delta^\#(k; \lambda)\, \Delta^\#(k; -\lambda).
$$
Writing for brevity  $e_n(k) := |E_n(k)|$, according to \eqref{ddd1} we have
\begin{align}\label{eq:defndsh}
d(k; \lambda)
= 
\prod_{n=0}^\infty \big(1- e_n(k)^2/\lambda^2\big).
\end{align}
By Theorem~\ref{thm.c9}, we have
\begin{align}\label{eq:dshrep}
d(k; \lambda) =  
-\alpha(\lambda)^2 \mathcal P(k)^2 + \beta(\lambda)^2.
\end{align} 
Let us fix a number $k\in (0, \omega/2)$ 
and  
differentiate the relation $d(k; e_n(k))=0$ with respect to $k$:
\begin{equation}\label{eq:alter1}
d'_k(k; e_n(k)) + 
d'_\lambda(k; e_n(k)) e_n'(k)=0.
\end{equation}
By \eqref{eq:defndsh}, the function $\lambda\mapsto d(k;\lambda)$ 
has positive zeros exactly at the points $e_n(k)$, $n=0,1,\dots$, 
and by Theorem~\ref{thm:anbr} these zeros are simple. 
Since $d(k;\lambda)\to 1$ as $\lambda\to \infty$, it follows that 
\begin{align}\label{eq:derd}
(-1)^n\, d'_\lambda(k; e_n(k))>0,\quad n = 0, 1, \dots.
\end{align}
Furthermore, by representation \eqref{eq:dshrep}, we have
$$
d'_k(k; e_n(k)) 
= - 2 \alpha(e_n(k))^2 \mathcal P(k) \mathcal P'(k). 
$$
By Theorem~\ref{thm.c9}\eqref{item:nonzero}, we have $\alpha(e_n(k)) \not = 0$. 
By definition \eqref{eq:ellip} and Lemma~\ref{prop:elf}, we have $\mathcal P(k) >0$ and
$\mathcal P'(k)<0$ for $k\in(0,\omega/2)$. 
Thus, $d'_k(k; e_n(k)) >0$.
Using this and \eqref{eq:derd} 
together with \eqref{eq:alter1}, we arrive at \eqref{eq:alter}.
\qed

\section{The Mathieu-Hankel operator}
\label{sec.e}
\subsection{The setup}
In this section we consider the example where the function $p$ is a trigonometric polynomial of the form
$$
p(\xi)=A+B\cos\omega(\xi-\xi_0), 
$$
with  $A,B,\xi_0\in \bbR$. 
Then the corresponding function $\mathfrak s(\xi)$ is, according to \eqref{eq:sfunct} and \eqref{eq:sfunct1}, 
\begin{align*}
\mathfrak s(\xi) &= A + 
\frac{B}{2\Gamma(1-i\omega)}e^{i\omega(\xi-\xi_0)}
+
\frac{B}{2\Gamma(1+i\omega)}e^{-i\omega(\xi-\xi_0)}
\\
&=A+\frac{B}{\abs{\Gamma(1-i\omega)}}\cos(\omega(\xi-\xi_0-\xi_*)),
\end{align*} 
where $\xi_*=(\arg\Gamma(1-i\omega))/\omega$. 
In order to make the formulas below more readable, we would like to get rid of the inessential factor $B/\abs{\Gamma(1-i\omega)}$ and the phase shift $\xi_0+\xi_*$. 
This can be achieved by a scaling and a unitary transformation. Skipping the details of this reduction, we pass straight to the case
$$
\mathfrak s(\xi)=A+\cos\omega\xi
$$
with $A\in\bbR$. 
In this case we have
\[
\wt{\mathfrak s}_{0} = A,\quad  
\wt{\mathfrak s}_{1} =\wt{\mathfrak s}_{-1}= \frac{1}{2}
\]
and $\wt{\mathfrak s}_{n}=0$ for all other $n$. 

Our main focus below will be the dependence of the spectral characteristics of the corresponding Mathieu-Hankel operator $\mathbf H$ on the parameter $A\in\bbR$, and therefore we will indicate the dependence on $A$ explicitly in our notation, i.e. we will write $\mathbf H(A)$, $H(k;A)$, $E(k;A)$ etc. 

In line with this, we denote by $\mathfrak S(A)$ the tri-diagonal matrix \eqref{eq:umn}, i.e. 
\[
[\mathfrak S(A)]_{n,m}=\widetilde{\mathfrak s}_{n-m}. 
\]
As in \eqref{eq:bk}, we write the fiber operator $H(k; A)$ in the form 
\begin{align}\label{eq:hka}
H(k;A) = \calM[\gamma]^*{\mathfrak S}(A)\calM[\gamma]
= H(k; 0) + A\, \calM[\gamma]^*\,\calM[\gamma], 
\end{align}
where $\gamma = \gamma(k) = \{\gamma_n(k)\}$, 
$\gamma_n(k) = \Gamma(\tfrac12+i\omega n+ ik)$, $n\in\mathbb Z$.

\subsection{All bands are flat for $A = 0$} 

\begin{theorem}\label{thm:flat}
If $A = 0$, then the non-zero spectrum of the Mathieu-Hankel operator $\mathbf H(0)$ consists of infinitely many eigenvalues of infinite multiplicity only. 
\end{theorem} 

In other words, for $A=0$ \emph{all spectral bands of $\mathbf H(0)$ are flat}.
Of course, in accordance with Theorem~\ref{thm.main}(i) the eigenvalues of $\mathbf H(0)$ are symmetric with respect to $0$, i.e. $\lambda$ is an eigenvalue if and only if $-\lambda$ is an eigenvalue.

The proof is based on the symmetry of the secular determinant $\Delta(s; \lambda; A)$ of $\mathbf H(A)$, which is described in the following lemma. 

\begin{lemma}\label{lem:Asym}
The secular determinant of the the Mathieu-Hankel operator satifies the relation 
\begin{align}\label{eq:syma}
\Delta(s; \lambda; A) = \Delta(s; -\lambda; -A)
\end{align} 
for all $s\notin {\sf\Lambda}$, $\lambda\not = 0$ and $A\in\mathbb R$.
Consequently, the operator ${\mathbf H}(-A)$ is unitarily equivalent to $-{\mathbf H}(A)$. Furthermore, the determinant $\Delta(s; \lambda; 0)$ is doubly periodic in $s$ with periods $\omega$ and $i$. 
\end{lemma}

\begin{proof}
Consider the unitary operator $K$ of multiplication by the sequence $(-1)^n$ in $\ell^2(\bbZ)$, that is $(K u)_n = (-1)^n u_n$ for $u\in \ell^2(\bbZ)$. 
We use the representation \eqref{eq:delta1}, which in our current notation reads
$$
\Delta(s; \lambda; A)=\det\bigl(I-\lambda^{-1}{\mathfrak S}(A)\calM[w]\bigr).
$$
By inspection we find $K {\mathfrak S}(A)K  = - {\mathfrak S}(-A)$. 
Furthermore, as  
$K$ and $\calM[w]$ commute, we can write 
\begin{align*}
\Delta(s; \lambda; A) = &\ \det(I - \lambda^{-1} (K {\mathfrak S}(A)K ) K \calM[w]K)\\
= &\ \det(I + \lambda^{-1}{\mathfrak S}(-A) \, \calM[w]) = \Delta(s; -\lambda; -A),
\end{align*}
as claimed. It follows that the band functions of ${\mathbf H}(-A)$ coincide with the band functions of $-{\mathbf H}(A)$. Thus, for every $k$ the operators $H(k;-A)$ and $-H(k;A)$ are unitarily equivalent. It follows that the corresponding direct integrals are also unitarily equivalent. We conclude that ${\mathbf H}(-A)$ is unitarily equivalent to $-{\mathbf H}(A)$.

Furthermore, using identity \eqref{eq:syma} and the symmetry \eqref{ee21} we obtain that 
\begin{align*}
\Delta(s+i; \lambda; 0) = \Delta(s; -\lambda; 0) = \Delta(s; \lambda; 0). 
\end{align*}
Since we already know that $\Delta(s; \lambda; 0)$ is doubly-periodic with periods $2i$ and $\omega$, from here we find that it is in fact doubly-periodic with periods $i$ and $\omega$, as claimed. The proof is complete.
\end{proof}

\begin{proof}[Proof of Theorem~\ref{thm:flat}]
By Theorem~\ref{thm:elliptic}(ii) the determinant $\Delta(s; \lambda; 0)$ may have at most  
one pole $s = i/2$ in the 
fundamental domain of its periodicity lattice 
\begin{align*}
\{z = m\omega +in:\quad  m, n\in \mathbb Z\},
\end{align*}
and this pole is simple. 
Since there are no elliptic functions of order one, we find that for all $\lambda\not=0$ the function $s\mapsto\Delta(s; \lambda;0)$ is constant. It follows that the residual determinant $\Delta^\#(\cdot;\lambda;0)$ is also constant for all $\lambda\not=0$. This is only possible if  $\Delta^\#(\cdot;\lambda;0)$ is identically equal to $1$, i.e. there are no non-flat bands. We have proved that all spectral bands of $\mathbf H(A)$ for $A=0$ are flat. The fact that the number of flat bands is infinite follows from Lemma~\ref{lem:inf} below. 
\end{proof}

\begin{remark}
It is easy to construct a more general family of periodic Hankel operators with no non-flat bands. 
Assume that ${\mathfrak s}(\xi)$ is a trigonometric polynomial of the form 
\begin{align*}
\mathfrak s(\xi) = \sum_{j} \wt {\mathfrak s}_j e^{i(2j+1)l \omega},
\end{align*}
where $l$ is a fixed natural number and 
$\overline{\wt{\mathfrak s}_j} = \wt {\mathfrak s}_{-j-1}$, 
so that $\mathfrak s$ is real-valued. 
Taking $l=1$, $\wt{\mathfrak s}_{-1} = \wt{\mathfrak s}_0 = 1/2$ 
and zero for the rest of the coefficients recovers the Mathieu-Hankel operator with $A = 0$.  
Define $K_l$ as the operator of multiplication by the sequence $e^{in\pi/l}$ in $\ell^2(\bbZ)$. Since 
$K_l {\mathfrak S}K_l^*=-{\mathfrak S}$, we obtain, as in the proof of Lemma~\ref{lem:Asym}, that 
$\Delta(s; \lambda) = \Delta(s; -\lambda)$, for all $s$ and $\lambda\not = 0$. Thus the secular determinant is doubly periodic with periods $i$ and $\omega$, and hence it is a constant function. This  implies that the spectrum of the corresponding Hankel operator consists of flat bands only.
\end{remark}
 
Note that flat bands also appear for $A\not = 0$, see Theorem~\ref{thm:coexist} below. 

In the rest of this section, we consider he Mathieu-Hankel operator $\mathbf H(A)$ for arbitrary $A\in \bbR$.

\subsection{Infinitude of band functions}

\begin{lemma}\label{lem:inf}
If $A\geq1$ (resp. $A\leq-1$), 
then for all $k\in\bbR$ the operator $H(k;A)$ is positive (resp. negative) semi-definite.
If $A>-1$ (resp. $A<1$), then $H(k;A)$ has infinitely many positive (resp. negative) 
eigenvalues for all $k\in\bbR$.
\end{lemma}

Heuristically, the conclusions of Lemma~\ref{lem:inf} can be ``deduced" from spectral properties of $\mathfrak S(A)$. Since the operator $\mathfrak S(A)$ in $\ell^2(\bbZ)$ is unitarily equivalent to the operator of multiplication by $A+\cos x$ in $L^2(0,2\pi)$, its spectrum is absolutely continuous and it coincides with the interval $[-1+A,1+A]$. If $A>-1$ (resp. $A < 1$), the positive (resp. negative) part of the spectrum is non-empty, which translates into infinitude of the positive (resp. negative) discrete spectrum for the operator $H(k; A)$. 

\begin{proof}
If $A\geq1$, then $\mathfrak S(A)$ is trivially positive semi-definite. 
Since $H(k;A) = \calM[\gamma]^*{\mathfrak S}(A)\calM[\gamma]$, it follows that $H(k;A)$ is also positive semi-definite, as claimed. 

Assume $A>-1$. Let us fix an arbitrary $k\in\bbR$ and prove that the number of positive eigenvalues of $H(k;A)$ is infinite.
To this end  for each natural number $M$ we will construct an $M$-dimensional subspace $\mathcal L_M\subset\ell^2(\mathbb Z)$ such that 
$\jap{H(k;A) F, F}_{\ell^2} > 0$ for all $F\in \mathcal L_M$.

Below $k$ is fixed and $\gamma_n=\Gamma(\frac12+i(\omega n+k))$. 
Let us fix a natural number $N>(1+A)^{-1}$  and define 
\begin{align*}
f_n = 
\begin{cases}
0, & \textup{if} \ n \le 0\ \textup{or}\ n\ge N+1,
\\
\gamma_n^{-1}, & \textup{if}\ 1\le n\le N.
\end{cases}
\end{align*} 
Now for each $j=1,\dots,M$, we define the element $F_j\in\ell^2(\bbZ)$ by specifying its coordinates as follows:
$$
[F_j]_n=f_{n - (j-1)(N+2)}, \quad n\in\bbZ.
$$
Observe that by construction, the supports of the sequences $F_j$ are pairwise disjoint and therefore the vectors $F_j$ are pairwise orthogonal in $\ell^2(\bbZ)$. Defining $\mathcal L_M$ as the span of the vectors $F_j$, $j = 1, \dots, M$, we find that  $\dim \mathcal L_M = M$.

Furthermore, by construction for each pair of indices $m\not = j$ 
the supports of $H(k;A) F_j$ and $F_m$ are disjoint and therefore $H(k;A) F_j\perp F_m$. 
Thus, in order to show that $\jap{H(k;A) F, F}_{\ell^2}>0$ for all $F\in \mathcal L_M$, it suffices to check that $\jap{H(k;A) F_j, F_j}_{\ell^2}>0$ for each $j = 1, \dots, M$. 
Consider, for example, $j=1$. Then 
\begin{align*}
\jap{H(k;A) F_1, F_1}_{\ell^2}&=\sum_{n, m} 
\wt{\mathfrak s}_{n-m} \gamma_n f_{n}\overline{\gamma_m f_m} 
= \frac12\sum_{\genfrac{}{}{0pt}{1}{1\leq n,m\leq N}{\abs{n-m}=1}}\, 1  
+ A\sum_{n=1}^N \abs{\gamma_n}^2\abs{f_n}^2
\\
& = \frac12\sum_{\genfrac{}{}{0pt}{1}{1\leq n,m\leq N}{\abs{n-m}=1}}1
+
A\sum_{n=1}^N 1
= (N-1)+AN= (A+1)N - 1.
\end{align*}
By our choice $N>(1+A)^{-1}$, the right-hand side is positive.  
Thus we have $\jap{H(k;A)F, F}_{\ell^2}>0$ for all $F\in \mathcal L_M$, as claimed.
Since $\dim \mathcal L_M=M$ can be taken arbitrarily large, the 
number of negative eigevalues is infinite, as claimed.

The remaining statements of the Lemma follow from the unitary equivalence of $H(k; A)$ and $-H(k; -A)$ (see Lemma~\ref{lem:Asym}) combined with the conclusions made above. For example, for $A <1$ the number of negative eigenvalues is infinite because $H(k; -A)$ has infinitely many positive eigenvalues. This concludes the proof.
\end{proof}

\subsection{All spectral gaps are open}
Let us define the spectral bands as in \eqref{eq:bandsgen}, i.e. 
\begin{align}\label{eq:mhbands}
\sigma = \{E(k): k\in [0, \omega/2]\},
\end{align}
for each band function $E\in\mathcal E$.
\begin{theorem}\label{prop:simple} 
Let $A\in\bbR$ and let $E_1$, $E_2$ be any two distinct band functions (either flat or non-flat) of the Mathieu-Hankel operator $\mathbf H(A)$. 
Then $E_1$ and $E_2$ do not intersect, and 
the corresponding spectral bands are disjoint. 
Moreover, each band function $E$ of $\mathbf H(A)$ is an even $\omega$-periodic function of $k\in\bbR$, 
so that \eqref{eq:mhbands} translates into $\sigma = \{E(k): k\in \bbR\}$.
 \end{theorem}

This theorem shows that any pair of neighbouring spectral bands 
is separated by an open interval, called \emph{spectral gap}. Using the terminology of the spectral theory of Schr\"odinger operators, we say that \emph{all spectral gaps are open}, i.e. no spectral gap degenerates into a point, or in other words, the bands do not touch. This picture is 
in a sharp contrast with the one for the Carleman operator (see Section~\ref{sec.b8}), where all gaps are closed. Furthermore, none of the band functions for the Carleman operator is periodic. 

\begin{proof}[Proof of Theorem~\ref{prop:simple}]
Let $A\in\bbR$ and $k\in\bbR$; 
let us prove that all  eigenvalues of $H(k;A)$ are simple. 
This is an elementary and well-known fact of the theory of Jacobi matrices, see e.g. \cite{Teschl2000}. 
Indeed, suppose that some eigenvalue $E_*$ of $H(k;A)$ has multiplicity $\geq2$. Then there exist 
two linearly independent eigenvectors $f,g\in\ell^2(\bbZ)$ of $H(k;A)$ corresponding to the eigenvalue $E_*$. 
Consider the Wronskian
\begin{align*}
W_m(f, g) = \gamma_m\gamma_{m+1} \big(f_{m} g_{m+1} - f_{m+1} g_m \big).
\end{align*} 
Since both $f$ and $g$ satisfy the same eigenvalue equation, one easily checks that $W_m(f,g)$ is independent of $m$. On the other hand, since both $f$ and $g$ are in $\ell^2$, we find that $W_m(f,g)\to0$ as $m\to\infty$. We conclude that the Wronskian vanishes identically. From here it follows that $f$ and $g$ are linearly dependent -- contradiction. 

From the simplicity of eigenvalues for all $k\in\bbR$ we conclude that the band functions cannot intersect. 
By Remark~\ref{rem:bands}(ii), this means that the bands cannot touch, and hence all gaps are open, as required.

Let us now prove that each band function is even and $\omega$-periodic. 
By Theorem~\ref{thm.branches} the functions $E(-k)$ and $E(k+\omega m)$, $m\in\bbZ$, are also band functions. Since $E(k) = E(-k)$ for $k = 0$ and the distinct band functions cannot intersect, this means that $E(k) = E(-k)$ for all $k\in\bbR$, i.e. $E$ is even. As a consequence, $E(-k-\omega m)$ is also a band function. Since  $E(-k-\omega m) = E(k)$ for $k = -\omega m/2$, we have $E(k) = E(-k-\omega m) = E(k+\omega m)$ for all $k\in\bbR$, so the function $E$ is periodic with period $\omega$, as claimed.  
The formula for the associated spectral band immediately follows from \eqref{eq:mhbands}. 
The proof of Theorem~\ref{prop:simple} is complete. 
\end{proof}

\subsection{Coexistence of flat and non-flat band functions}

\begin{theorem}\label{thm:coexist}
There exist $A_*\in(0,1)$ such that at least one band function of $\mathbf H(A_*)$ is flat while at least one band function is non-flat. 
\end{theorem}

Our immediate aim is to study the dependence  of the band functions of $\mathbf H(A)$ on $A\in(0,1)$; at the heart of the proof is the use of the intermediate value theorem for a concrete pair of these band functions as $A$ varies from $0$ to $1$. 

For $A\in(0,1)$, we enumerate the positive $E^+_n(k; A)$ and negative $E^-_n(k; A)$ 
band functions of $\mathbf H(A)$ in such a way that for all $k\in\bbR$, we have
\begin{align}
E^+_0(k;A) > E^+_1(k;A) > \dots>0,
\label{eq:enumplus}
\\
E^-_0(k;A) < E^-_1(k;A) < \dots<0.
\label{eq:enumminus}
\end{align}
This ordering is possible due to Theorem~\ref{prop:simple}.  By Lemma~\ref{lem:inf}, the number of both positive and negative band functions is infinite for $A\in(0,1)$. 
Recall for completeness that the associated spectral bands $\sigma^\pm_n(A)$ 
(as defined in \eqref{eq:mhbands}) are all 
pairwise disjoint, by Theorem~\ref{prop:simple}.

\begin{lemma}\label{lem:monotone}
\begin{enumerate}[\rm (i)]
\item
For all $k\in \bbR$ the eigenvalues $E^+_n(k; A)$ and $E^-_n(k; A)$
are strictly increasing real analytic functions of  $A\in (0, 1)$, with non-vanishing first derivatives. 
\item
We have $E^-_n(k; A)\to0_-$ as  $A\to1_-$ for all $n\geq0$ and $k\in\bbR$. 
\end{enumerate}
\end{lemma}  

\begin{proof}  
Let us prove (i). For definiteness, let us consider the positive band functions. 
Since $A\mapsto H(k; A)$ is a self-adjoint holomorphic family in a neighbourhood of each $A\in\bbR$, by the standard analytic perturbation theory (see e.g. \cite[Theorem XII.13]{RS4}),
each eigenvalue $E^+_n(k; A)$ is a real analytic function of $A\in (0,1)$. 
Moreover, we can apply the standard formula 
\begin{align*}
\tfrac{d}{d\, A} E^+_n(k; A) 
= \big\langle\tfrac{d}{dA}H(k; A) \phi, \phi\big\rangle_{\ell^2} = \norm{\calM[\gamma]\phi}^2>0, 
\end{align*} 
where $\phi$ is the normalized eigenfunction associated with the eigenvalue 
$E^+_n(k; A)$ (we have already used such a formula in the proof of Theorem~\ref{thm.branches}).
Thus each function $A\mapsto E^{+}_n(k;A)$ is strictly increasing with a non-vanishing first derivative.

Let us prove (ii). Fix $k\in\bbR$ and $n\geq0$. Since $\mathfrak S(A)\ge A-1$, we have 
\begin{align*}
H(k; A)\ge (A-1) \calM[\gamma]^* \calM[\gamma]\ge (A-1)\|\calM[\gamma]\|^2.
\end{align*}
Remembering also that $\inf H(k; A) \le 0$ for all $A\in\bbR$, we conclude that 
$\inf H(k; A)\to 0_-$ as $A\to 1_-$. 
Therefore $E^-_n(k;A)\to 0$, as claimed.
\end{proof}

\begin{proof}[Proof of Theorem~\ref{thm:coexist}]
Fix $k\in\bbR$ and consider 
the second band function $E^+_1(k;A)$ and the ``reflected'' band function $-E^-_0(k;A)$. 
For $A=0$, all band functions are flat and symmetric around zero, and in 
particular $-E^-_0(k;0)=E^+_0(k;0)$ 
and 
\[
-E^-_0(k;0)>E^+_1(k;0).
\]  
Let us now see  
how the left and right hand sides of this inequality change as $A$ increases from $0$ to $1$. 
By Lemma~\ref{lem:monotone}(i),  
as $A$ grows, the function $E^+_1(k;A)$ increases, while $-E^-_0(k;A)$ 
decreases  (see Figure~\ref{figure1}).
Furthermore, by Lemma~\ref{lem:monotone}(ii), as $A\to1_-$, we have $-E^-_0(k;A)\to0_+$. 
Therefore, there exists a (unique) point $A_*\in(0,1)$ such that 
$$
-E^-_0(k;A_*)=E^+_1(k;A_*).
$$
By Theorem~\ref{thm:anbr}\eqref{item:anbr4}, at least one of the band functions $E^-_0(\cdot;A_*)$ and $E^+_1(\cdot;A_*)$ is flat. (In fact, by Lemma~\ref{lma.d1} both of them are flat.)
Finally, let us show that the top band function $E^+_0(\cdot;A_*)$ is non-flat. The situation here is precisely as in Theorem~\ref{thm:noman}: the bottom of the spectrum of ${\mathbf H}(A_*)$ is $E^-_0(A_*)$, while the whole of the range of the band function $E^+_0(\cdot;A_*)$ is above the value $-E^-_0(A_*)$, and therefore it is non-flat. The proof is complete.
\end{proof}
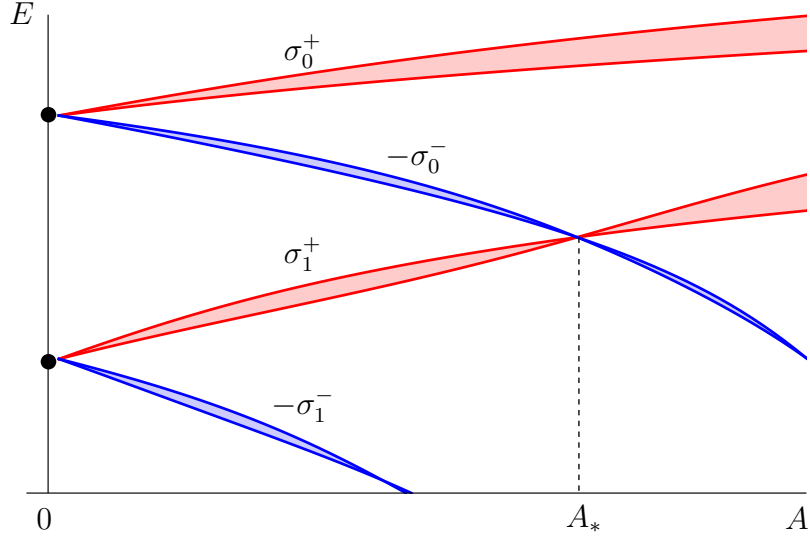
\begin{figure}
\begin{tikzpicture}
\node at (0,0){\includegraphics[width=0.7\textwidth]{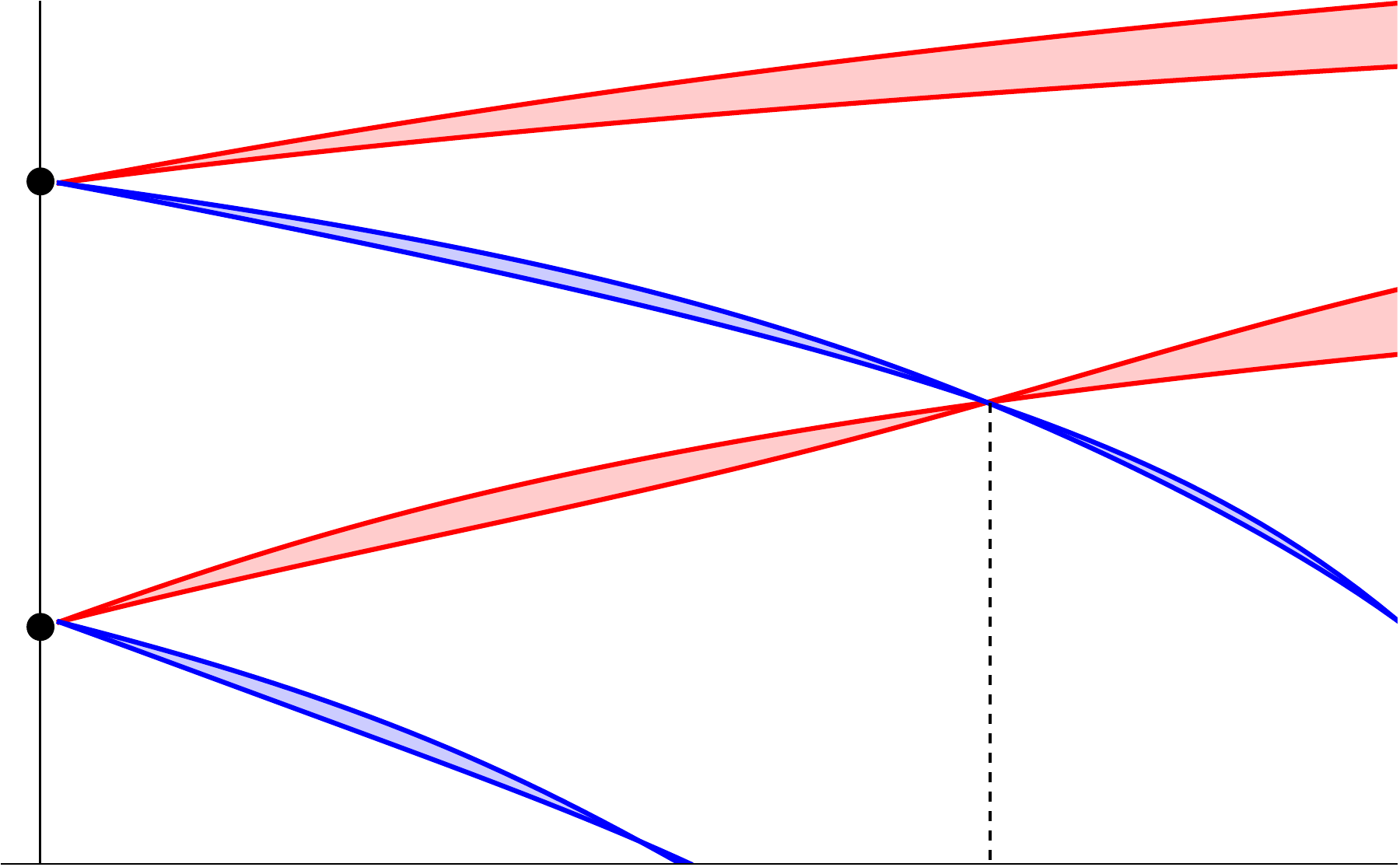}};
\draw (-1.5,2.7) node{$\sigma^+_0$};
\draw (0,1.3)  node{$-\sigma^-_0$};
\draw (-1.5,0) node{$\sigma^+_1$};
\draw (-1.5,-2) node{$-\sigma^-_1$};
\draw (2.2,-3.5) node{$A_*$};
\draw (5,-3.5) node{$A$};
\draw (-5.2,3.2) node{$E$};
\draw (-4.9,-3.5) node{$0$};
\end{tikzpicture}
\caption{
Spectral bands as functions of the parameter $A$ in logarithmic scale,  obtained by a numerical computation in \emph{Mathematica}. The bands 
$\sigma_1^+(A)$ and $-\sigma_0^-(A)$ both become flat at the point $A_*\approx0.48$ where the corresponding band functions intersect. }
\label{figure1} 
\end{figure}
\begin{remark*}
It is easy to develop the argument in the proof of Theorem~\ref{thm:coexist} to get a more complete picture of the dependence of spectral bands on the parameter $A\in(0,1)$. As $A$ increases, every ``reflected'' band function $-E^-_n$ will intersect exactly once each of the band functions $E^+_{n+1}$, $E^+_{n+2}$, $\cdots$. At each point of intersection, both band functions become flat. Thus, we have an infinite countable set of points $A$ where $\mathbf H(A)$ has flat bands. Between these intersection points, the band function $E^-_n$ remains non-flat. 
\end{remark*}

\subsection{Monotonicity of band functions}
Let us now address the following question: how does the monotonicity of the band functions, as functions of $k\in(0,\omega/2)$, change for different values of the parameter $A$? We keep the enumeration of the band functions as in \eqref{eq:enumplus}, \eqref{eq:enumminus}. 

As a warm-up, we consider the case $A>1$. In this case the Mathieu-Hankel operator $\mathbf H(A)$ is positive semi-definite (see Lemma~\ref{lem:inf}) and so Theorem~\ref{thm:alterpos} applies: the top band function $E_0^+(k;A)$ is decreasing in $k\in(0,\omega/2)$, and then the sign of the derivative of $E_n^+(k;A)$ alternates with $n$:
\[
(-1)^{n+1}\frac{d}{dk}E_n^+(k;A)>0
\]

This contrasts with the situation for small $A>0$: 

\begin{theorem}\label{thm.order}
Let $n\geq0$; then for all sufficiently small $A>0$, the band function $E^+_n(k;A)$ is decreasing in $k\in(0,\omega/2)$.
\end{theorem}
\begin{proof}
Recall that by Theorem~\ref{thm:flat}, for $A=0$ all band functions are flat; we will denote these flat band functions by 
\[
{\mathbf E}^\pm_n:=E^\pm_n(k;0),
\]
with the enumeration as in \eqref{eq:enumplus}, \eqref{eq:enumminus},
and bear in mind that ${\mathbf E}^-_n=-{\mathbf E}^+_n$. 

For a fixed $n$, let us consider $E^+_n(k;A)$ and $E^-_n(k;A)$ for sufficiently small $A>0$. 
By the strict monotonicity in $A$, we have
\[
E_n^+(k;A)>{\mathbf E}^+_n \quad\text{ and }\quad E_n^-(k;A)>{\mathbf E}^-_n
\]
for all $k$ and therefore
\[
E_n^+(k;A)>{\mathbf E}^+_n>-E_n^-(k;A).
\]
It follows that the spectral band $\sigma^+_n(A)$ 
is strictly above ${\mathbf E}^+_n$, while 
the reflected band $-\sigma^-_n(A)$ 
is strictly below ${\mathbf E}^+_n$, see Figure~\ref{figure1}. In particular, these two bands
are disjoint from one another. 
Now, by \eqref{eq:hka} and \eqref{eq:gamco}, $\| H(k; A) - H(k; 0)\|\le \pi A$, 
and hence, by a straightforward perturbation argument, 
 for all $A\in(0,A_n)$ with a sufficiently small $A_n$ the bands 
$\sigma^+_n(A)$  and $-\sigma^-_n(A)$ 
 are also 
disjoint from all the other positive and reflected negative bands. 
As a consequence,
for $A\in(0,A_n)$ both $E^+_n(\cdot;A)$ and $E_n^-(\cdot;A)$ and non-flat. 
Applying this argument to each of the band functions 
\begin{equation}
E^-_0,\dots,E^-_n,E^+_0,\dots,E^+_n,
\label{eq:0-n}
\end{equation}
considered simultaneously, we obtain that
for $0<A<\min\{A_0,\dots,A_n\}$
the band functions \eqref{eq:0-n} form the list
\begin{equation}
\abs{E^+_0}>\abs{E^-_0}>\abs{E^+_1}>\abs{E^-_1}> \cdots > |E^+_n|>|E^-_n|,
\label{eq:list}
\end{equation}
(see Figure~\ref{figure1}) in which the positive band functions $E^+_0$,\dots,$E^+_n$ have \emph{even} indices. By Theorem~\ref{thm:mod} we conclude that these positive band functions are decreasing in $k\in(0,\omega/2)$. 
\end{proof}

\begin{figure}[h!] 
\begin{tikzpicture}
\node at (0,0){\includegraphics[width=0.4\textwidth,height=0.3\textwidth]{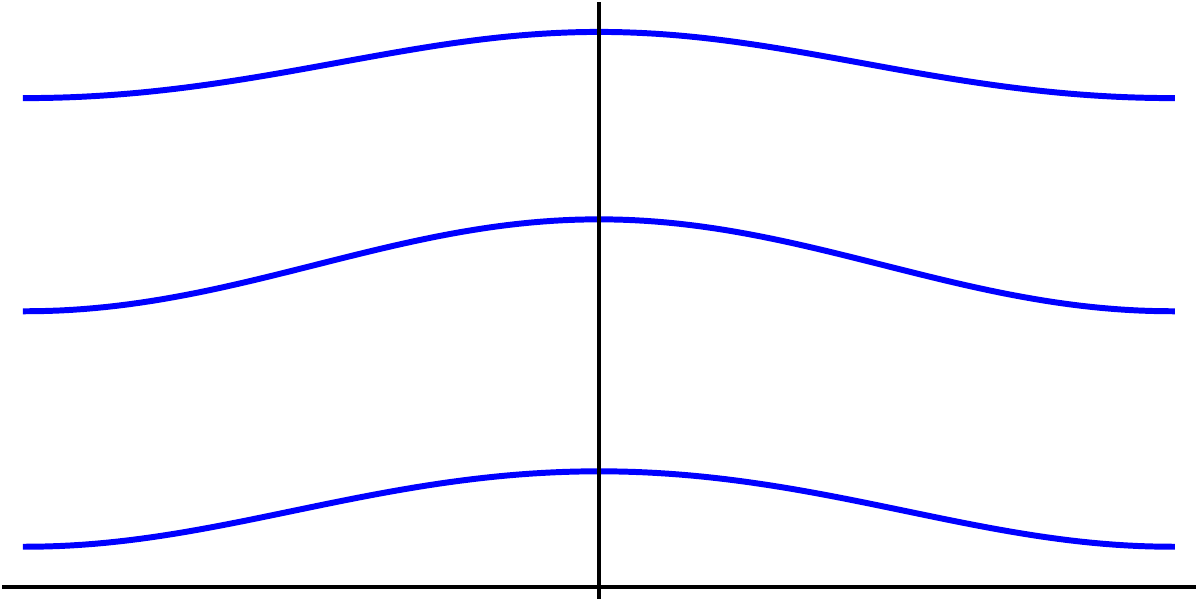}};
\draw (0.3,2.3) node{$E$};
\draw (0.1,-2.5) node{$0$};
\draw (3,-2.5) node{$k$};
\end{tikzpicture}
\quad
\begin{tikzpicture}
\node at (0,0){\includegraphics[width=0.4\textwidth,height=0.3\textwidth]{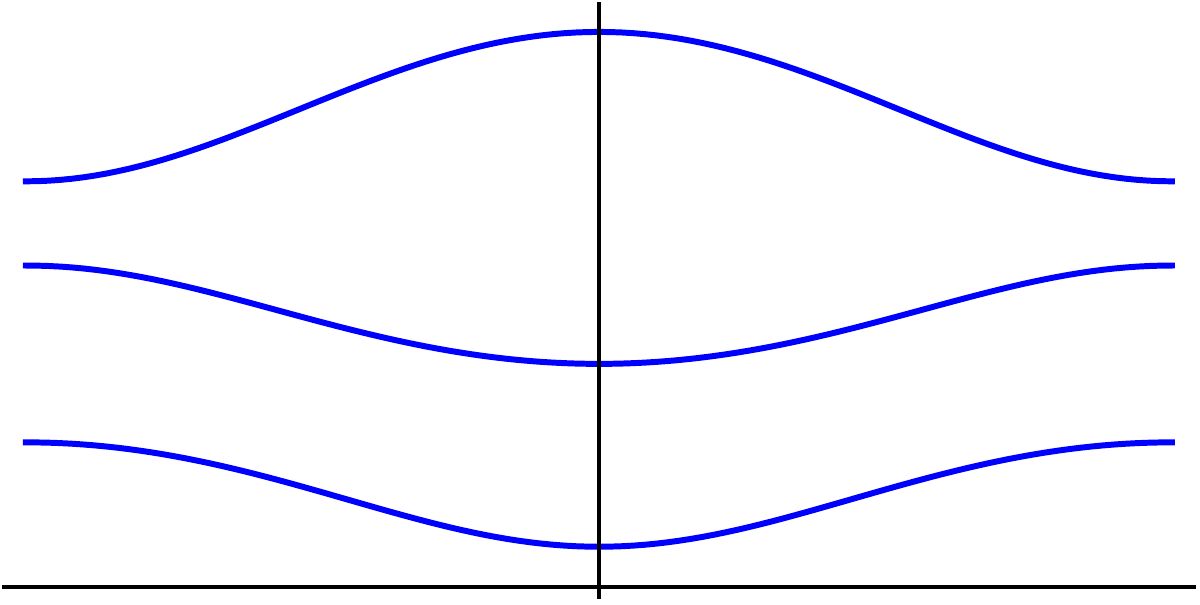}};
\draw (0.3,2.3) node{$E$};
\draw (0.1,-2.5) node{$0$};
\draw (3,-2.5) node{$k$};
\end{tikzpicture}
\caption{The top three band functions (not to scale) on the period for $A=0.2$ and $A=0.6$.}
    \label{figure3a}
\end{figure}
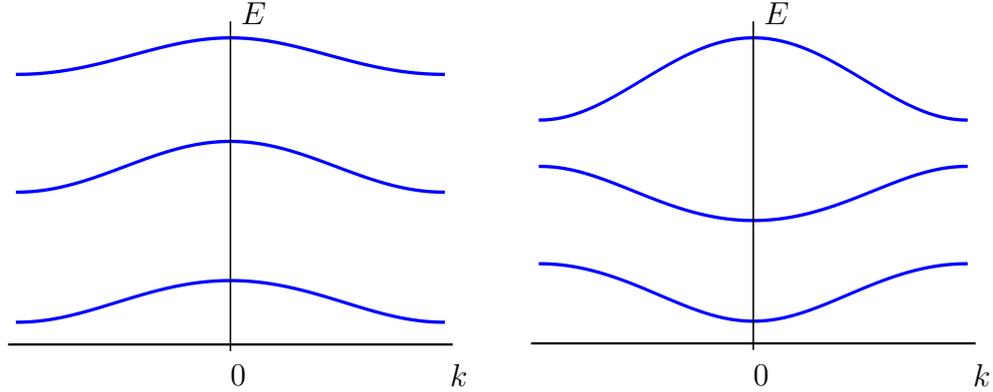

Now let us comment on the scenario depicted in Figure~\ref{figure1}.
The top band function $E^+_0(\cdot;A)$ is decreasing in $k\in(0,\omega/2)$ for all $A>0$ by Theorem~\ref{thm:alterpos}. 
Let $A_*$ be as in the proof of Theorem~\ref{thm:coexist}. 
For $0<A<A_*$ we have the ordering as in \eqref{eq:list}, i.e. 
\[
\abs{E^+_0}>\abs{E^-_0}>\abs{E^+_1},
\]
while for for $A>A^*$ we have
\[
\abs{E^+_0}>\abs{E^+_1}>\abs{E^-_0}.
\]
By Theorem~\ref{thm:mod}, it follows that the second band function $E^+_1(\cdot;A)$ is decreasing for  $0<A<A_*$ and increasing for $A>A_*$. Thus, the direction of monotonicity of $E^+_1(\cdot;A)$ changes precisely at the point $A_*$ where $E^+_1(\cdot;A)$ becomes flat.

It is not difficult to generalise this observation 
to the $n$'th band function $E^+_n(\cdot;A)$. It is decreasing in $k\in(0,\omega/2)$ for small $A>0$, and then, as $A$ increases, the sign of monotonicity changes every time when the $n$'th band becomes flat. As already discussed, this happens precisely 
when $E^+_n(k;A)$ intersects $-E^-_j(k;A)$ with $j<n$. 
For the top three positive band functions this change is illustrated in Figure~\ref{figure3a}.

\subsection{ The kernel of $\mathbf H(A)$ is trivial}
\label{sec.trivker}

\begin{theorem}\label{thm:kernel}
For any $A\in\bbR$, the kernel of $\mathbf H(A)$ is trivial.
\end{theorem}
\begin{proof}
Assume, to get a contradiction, that for some $A\in\bbR$ the kernel of $\mathbf H(A)$ is non-trivial. 
By the Floquet-Bloch decomposition, the kernel of $H(k;A)$ must be non-trivial for $k$ on a 
set of positive measure. 
Fix one such $k$ and consider $f\in\Ker H(k;A)$. 
Denote
$$
g_m=\Gamma(\tfrac12+i\omega m+ik)f_m
$$ 
and let 
\[
G(\xi)=\sum_{n\in\bbZ}g_n e^{i\omega \xi n};
\]
since $g_n$ decay exponentially as $\abs{n}\to\infty$, the non-zero 
function $G(\xi)$ is continuous (and even real-analytic) in $\xi$.
From the condition $H(k;A)f=0$ we find
$$
\sum_{m\in\bbZ}\widetilde{\mathfrak s}_{n-m}g_m=0
$$
for all $n\in\bbZ$. Denoting (as in \eqref{eq:sfunct1})
$$
{\mathfrak s}(\xi)=\wt{\mathfrak s}_{-1}e^{-i\omega\xi}+\wt{\mathfrak s}_{0}+\wt{\mathfrak s}_{1}e^{i\omega\xi}=A+\cos(\omega\xi),
$$
this rewrites as
$$
\int_0^{T} G(\xi) {\mathfrak s}(\xi)e^{-i\omega\xi n}d\xi=0
$$
for all $n$. 
Therefore $G(\xi){\mathfrak s}(\xi)=0$ for almost all $\xi$, and hence $G(\xi) = 0$, 
which gives  a contradiction.
\end{proof}

Clearly, the conclusion of Theorem \ref{thm:kernel} still holds for Hankel operators 
$\mathbf H$ such that $\wt{\mathfrak s}_n\to 0$ as $n\to\infty,$ sufficiently fast. For example, 
the absolute convergence of the series \eqref{eq:sfunct1} will suffice.

\appendix

\section{Proof of Proposition~\ref{prop:positive}}

We rely on the following standard result, see 
\cite[Theorem 5.5.4]{Akhiezer2021} and \cite[Ch. VI, Theorem 21]{Widder1941}:

\begin{proposition} 
Suppose that the function $h$ is continuous on $(0,\infty)$ and positive-definite, i.e. 
\begin{align}\label{eq:pos}
\sum_{j, k = 1}^n h(t_j+t_k) \xi_j\xi_k\ge 0,\quad \textup{for all}\quad 
(\xi_1, \xi_2, \dots, \xi_n)\in \bbR^n, 
\end{align}
for all $n$-tuples of points $t_j>0$, $j = 1, 2, \dots, n$.
Then there exists a positive measure $d\sigma$ on $\bbR$ 
such that 
\begin{align}\label{eq:genrep}
h(t) = \int_{-\infty}^\infty e^{-t\lambda} \,d\sigma(\lambda), \quad t>0,
\end{align}
and the integral converges for all $t>0$.
\end{proposition}

\begin{proof}[Proof of Proposition~\ref{prop:positive}]
Since the operator $\mathbf H$ is positive semi-definite, we have  
\[
\int_0^\infty\int_0^\infty h(t_1+t_2) u(t_1) u(t_2) \, dt_1 \, dt_2 \ge 0,
\]
 for any continuous real-valued function $u$ supported on $(0, \infty)$. Since $h$ is continuous on 
 $(0, \infty)$, this condition is equivalent to condition \eqref{eq:pos}, 
see e.g.  \cite[Ch. VI, Theorem 20]{Widder1941}. 
Therefore $h$ can be represented in the form \eqref{eq:genrep}. 

It is easy to see that if the support of $d\sigma$ contains an interval in $(-\infty,0)$, then $h(t)$ grows exponentially as $t\to\infty$. In this case the operator $\mathbf H$ cannot be bounded. This shows that the support of $d\sigma$ must be confined to $[0,\infty)$. 
 The uniqueness follows from \cite[Ch. II, Theorem 6.3]{Widder1941}.

The rest of the proof follows Widom's proof of \cite[Theorem 3.1]{Widom66} for Hankel matrices.

Proof of \eqref{item:meas}. Since $\mathbf H$ is bounded, we have 
\begin{align*}
\jap{\mathbf H u, u}\le C\|u\|^2
\end{align*}
for all $u\in L^2(\bbR_+)$. Take $u(t) = e^{-\mu t}$ with arbitrary $\mu>0$. Then 
\begin{align*}
\jap{\mathbf H u, u} = &\ \iiint e^{-\lambda(t_1+t_2)} e^{-\mu(t_1+t_2)} \, dt_1\, dt_2 \, d\sigma(\lambda)\\
= &\ \int_0^\infty  \frac{1}{(\lambda+\mu)^2} \, d\sigma(\lambda) \le C\|u\|^2 = \frac{C}{2\mu}.
\end{align*}
Therefore, integrating by parts again and denoting $\sigma(\lambda) :=\sigma([0, \lambda])$, 
\begin{align*}
\frac{C}{2\mu}
\ge &\ 2 \int_0^\infty  \frac{1}{(\lambda+\mu)^3} \, \sigma(\lambda)\, d\lambda\\
\ge &\ 2 \int_{\mu}^{2\mu}
\frac{\sigma(\lambda)}{(\lambda+\mu)^3} \, d \lambda
\ge \frac{2}{(3\mu)^3} \,\sigma(\mu) \int_{\mu}^{2\mu} d\lambda\\
= &\ 2 \frac{1}{3^3 \mu^2}\, \sigma(\mu),
\end{align*}
and consequently, $\sigma(\mu)\le C' \mu$.
The infinite differentiability of $h$ for $t >0$ follows from here. 

Proof of \eqref{item:ker}. Integrating by parts we get 
from \eqref{item:meas} that 
\begin{align*}
h(t) = t \int_0^\infty e^{-\lambda t} \sigma(\lambda) \, d\lambda
\le C t \int_0^{\infty} e^{-\lambda t} \lambda \, d\lambda
\le \frac{C'}{t},
\end{align*}
as required. 
The proof of Theorem~\ref{prop:positive}  is now complete. 
\end{proof}

\end{document}